 \documentclass[final,3p,times]{elsarticle}

\usepackage{amsmath,amssymb,amsfonts}
\usepackage{amsthm}
\usepackage[dvipsnames,table]{xcolor}
\usepackage{colortbl}
\usepackage{dsfont}
\usepackage{lmodern}
\usepackage{bbm}

\usepackage[labelfont=bf]{caption}
\usepackage{placeins}
\usepackage{rotating}
\usepackage{lscape}

\usepackage{tikz}
\definecolor{orange}{HTML}{d9944f}
\usetikzlibrary{patterns,arrows,positioning,fit}
\usetikzlibrary{decorations.pathreplacing}
\usetikzlibrary{decorations.pathmorphing}

\usepackage{graphicx,subcaption,lipsum}
\usepackage{url}
\usepackage{doi}
\usepackage{hyperref}
\usepackage[capitalize]{cleveref}
\usepackage{float}
\usepackage{orcidlink}
\usepackage{booktabs}

\newtheorem{theorem}{Theorem}[section]

\newtheorem{corollary}[theorem]{Corollary}
\newtheorem{remark}[theorem]{Remark}

\newcommand{\AW}[1]{{\color{black}#1}}

\newcommand{\LP}[1]{{\color{black}#1}}

\newcommand{\MK}[1]{{\color{black}#1}}

\journal{Elsevier}

\usepackage{natbib}
\allowdisplaybreaks[2]

\begin{document}

\begin{frontmatter}



\title{Revisiting the Linear Chain Trick in epidemiological models: Implications of underlying assumptions for numerical solutions}

\author[label1]{Lena Pl\"otzke\hspace{0.5mm}\orcidlink{0000-0003-0440-1429}}
\ead{lena.ploetzke@dlr.de}
\author[label1]{Anna Wendler\hspace{0.5mm}\orcidlink{0000-0002-1816-8907}}
\author[label1]{Ren\'{e} Schmieding\hspace{0.5mm}\orcidlink{0000-0002-2769-0270}}
\author[label1,label2]{Martin J. Kühn\hspace{0.5mm}\orcidlink{0000-0002-0906-6984}}
\ead{martin.kuehn@dlr.de}

\affiliation[label1]{organization={Institute of Software Technology, Department of High-Performance Computing, German Aerospace Center},
            city={Cologne},
            country={Germany}}
\affiliation[label2]{organization={Life and Medical Sciences Institute and Bonn Center for Mathematical Life Sciences, University of Bonn},
            city={Bonn},
            country={Germany}}

\begin{abstract}
In order to simulate the spread of infectious diseases, many epidemiological models use systems of ordinary differential equations (ODEs) to describe the underlying dynamics. These models incorporate the implicit assumption, that the stay time in each disease state follows an exponential distribution. However, a substantial number of epidemiological, data-based studies indicate that this assumption is not plausible. One method to alleviate this limitation is to employ the Linear Chain Trick (LCT) for ODE systems, which realizes the use of Erlang distributed stay times. As indicated by data, this approach allows for more realistic models while maintaining the advantages of using ODEs.

In this work, we propose an advanced LCT \MK{SECIR-type} model incorporating eight infection states with demographic stratification. We review key properties of \MK{the corresponding} LCT model and demonstrate that predictions derived from a simple ODE-based model can be significantly distorted, potentially leading to wrong political decisions. Our findings demonstrate that the influence of distribution assumptions on the behavior at change points and on the prediction of epidemic peaks is substantial, while the assumption has no effect on the final size of the epidemic. \MK{With respect to prior findings in literature, we demonstrate that the influence of the number of subcompartments on the timing and size of the epidemic peak is nontrivial and that a general statement cannot be obtained. We, then, show how these age-resolved LCT SECIR-type models capture the spread of SARS-CoV-2 in Germany in 2020. Eventually, we study the implications on the time-to-solution for different LCT models using fixed and adaptive step-size Runge-Kutta methods and provide computational performance for these models in the MEmilio software framework, also using distributed memory parallelism to speed up ensemble runs.}
\end{abstract}


\begin{highlights}
\item Generalization of statements on incorrect predictions for peak size and timing with simple ordinary differential equation-based models
\item Assessment of the impact of underlying distribution assumptions at change points
\item Development of an age-resolved SECIR-type model using the Linear Chain Trick to allow for Erlang distributed stay times
\item Validation of our model with demographic resolution for a COVID-19-inspired scenario
\item Assessment of the performance of Linear Chain Trick models \MK{in the MEmilio software framework}
\end{highlights}


\end{frontmatter}
\section{Introduction}

Infectious diseases have long been a major challenge to the health and well-being of society. Despite improvements in hygiene standards and the development of vaccines and other medical breakthroughs, existing and emerging infectious diseases remain a major global health concern~\cite{hethcote_mathematics_2000,daszak_workshop_2020}.
A recent example is the COVID-19 pandemic, caused by the coronavirus SARS-CoV-2, which has resulted in $14.9$ million excess deaths in 2020 and 2021~\cite{who_statistics_2023}.

Mathematical models are a crucial tool for understanding the dynamics of infectious disease spread and analyzing possible nonpharmaceutical interventions (NPIs)~\cite{wallinga_how_2007}. Numerous methodological approaches, using mathematical models, are available for the prediction of the spread of infectious diseases. Among these approaches, those based on ordinary differential equations (ODE) are the most prevalent in the scientific literature~\cite{hurtado_generalizations_2019}. Simple ODE-based models can be generalized or extended differently to achieve a more realistic representation of the infection dynamics, e.g., to ODE-based metapopulation models~\cite{kuhn_assessment_2021,pei_differential_2020,chen_compliance_2021,levin_effects_2021,liu_modelling_2022,zunker_novel_2024} or integral-based differential equations~\cite{medlock_integro-differential-equation_2004,messina_non-standard_2022,wendler2024nonstandardnumericalschemenovel}. Substantially different approaches rely on the modeling of individuals and are given by agent-based~\cite{collier_parallel_2013,willem_optimizing_2015,bershteyn_implementation_2018,Kerkmann_ABM_24} or even use a hybrid combination of metapopulation and agent-based models~\cite{bradhurst_hybrid_2015,hunter_hybrid_2020,bicker_hybrid_2025}. While models of artificial intelligence (AI) have also been directly applied to disease dynamics data, also AI surrogate models for ABMs or spatially resolved metapopulation models have been proposed~\cite{robertson_bayesian_2024,schmidt_surrogates_2024}. 

ODE-based models are known for their simple formulation, well established mathematical analysis and straightforward implementation and numerical solution~\cite{hurtado_generalizations_2019}. In these simple ODE-based models, each disease state corresponds to one ODE. The formulation using linear transition rates between two states leads to the implicit assumption of exponentially distributed stay times~\cite{donofrio_mixed_2004,lloyd_realistic_2001}. From the epidemiological application, this model assumption is considered unrealistic~\cite{donofrio_mixed_2004,lloyd_realistic_2001,wearing_appropriate_2005,feng_epidemiological_2007,krylova_effects_2013,wang_evaluations_2017}.
One solution is to use models based on integro-differential equations (IDE), which allow a flexible choice of the distributions of the stay times in the disease states, see e.g.~\cite{wendler2024nonstandardnumericalschemenovel,messina_non-standard_2022}. This approach has already been presented by Kermack and McKendrick in~\cite{kermack_contribution_1927}, which is, however, mostly cited for its simplified ODE formulation~\cite{breda_formulation_2012}. From our understanding, this is a result of the IDE formulations being mathematically more challenging to analyze, formulate, and implement in software. 

To bypass the complexity of IDE formulations, a method using linear chains of substates in ODE formulations was developed~\cite{macdonald_time_1978}. This concept, called the Linear Chain Trick (LCT), generates Erlang distributed stay times for the initial compartments~\cite{hurtado_generalizations_2019,macdonald_time_1978}. As Erlang distributed stay times are considered to be more realistic than exponentially distributed stay times, an alternative, situated between IDE-based and simple ODE-based formulations, is given.

In this work, we propose a model using the LCT \LP{w}ith eight disease states including exposed, pre- and asymptomatic, symptomatic, severe, and critical states, all stratified by age. It is a versatile model for early epidemic states of, e.g., respiratory diseases. \LP{Research on LCT models has been conducted by several authors, including ~\cite{hurtado_generalizations_2019,feng_mathematical_2016,champredon_equivalence_2018,krylova_effects_2013,wearing_appropriate_2005}. In order to incorporate more realistic distribution assumptions, the approach has been applied in different settings. In particular, LCT models were used in recent studies to model the dynamics of the COVID-19 pandemic~\cite{contento_integrative_2023,rozhnova_model-based_2021,blyuss_effects_2021,overton_epibeds_2022,birrell_real-time_2021}.} However, most of these models were not age-resolved or considered less disease states. Our objective is to revisit the LCT with a generic model to review corresponding, partially contradictory statements made in the literature and to also consider implications for the numerical solution process.

This paper is structured as follows.
In~\cref{sec:model}, we introduce our detailed age-resolved LCT-based model and provide a description of the model parameters. Subsequently, in~\cref{sec:properties}, we revise some mathematical properties of the Erlang distribution and, more precisely, of our model. Then, in~\cref{sec:numerics}, we investigate the model behavior in detail by numerical experiments. In particular, we examine the influence of the distribution assumption on the behavior at change points and on the size and timing of epidemic peaks, as well as on the final size of the epidemic. We then demonstrate the significance of using an age-resolved model and apply our model to a scenario based on the spread of COVID-19 in Germany. \MK{Furthermore, we conduct a run time and performance study for various LCT model realizations} before, eventually, providing discussion and conclusion.

\section{An age-resolved SECIR-type Linear Chain Trick model}\label{sec:model}
In this section, we present a detailed SECIR-type model with eight compartments realizing Erlang distributed stay times and a stratification by age. We extend a version of the ODE model presented in~\cite{kuhn_assessment_2021} by the Linear Chain Trick concept. The version of the ODE model that serves as the foundation for this is also presented in~\cite[Appendix~A]{wendler2024nonstandardnumericalschemenovel} (without age resolution). As our model is formulated using the LCT, we denote the model \textit{LCT-SECIR model} or simply \textit{LCT model}.

In the model, individuals are classified according to their disease state and assigned to a specific compartment. The compartment \textit{Susceptible} ($S$) is used for individuals who are susceptible to infection with the considered disease and have a default immune protection; \textit{Exposed} ($E$) for individuals in their latent period that are infected but not yet infectious and \textit{Carrier} ($C$) for people that are infectious but do not show symptoms, which may be pre- or asymptomatic. We use \textit{Infected} ($I$) for people who are infectious and mildly symptomatic; \textit{Hospitalized} ($H$) for people suffering from severe symptoms; In \textit{Intensive Care Unit} ($U$); \textit{Recovered} ($R$) for people that are immune to any future infection and \textit{Dead} ($D$) for people that died from the disease. We define the set of compartments as $\mathcal{Z}:=\{S,E,C,I,H,U,R,D\}$. As infectious disease parameters can be highly dependent on age, see e.g.~\cite{kuhn_assessment_2021}, we further divide the population in $m\in\mathbb{N}$ different sociodemographic groups. While one could also stratify the population according to gender or education, we, here, use a stratification into age groups. We use the notation $Z_i(t)$ for the number of people of age group $i\in\{1,\dots, m\}$ with the disease state $Z\in\mathcal{Z}$ at simulation time $t\in\mathbb{R}$. 

To realize Erlang distributed stay times with an ODE model, we divide the compartments $Z\in\mathcal{A}:=\{E,C,I,H,U\}$ into $n_Z\in\mathbb{N}$ subcompartments. The compartments $\{S,R,D\}$ are not divided as they serve as initial or absorbing states. According to the LCT, replacing a compartment by a chain of subcompartments with linear transition rates leads to an Erlang distributed stay time in the compartment itself. A more detailed examination of this topic will be provided in~\cref{sec:properties}. We allow different numbers of subcompartments for different age groups, which is denoted by $n_{Z,i}\in\mathbb{N}$ for $Z\in\mathcal{A}$ and $i\in\{1,\dots,m\}$. For compartment $Z\in\mathcal{A}$, the number of people of age group $i\in\{1,\dots,m\}$ in subcompartment $j\in\{1,\dots,n_{Z,i}\}$ at simulation time $t$ is denoted by $Z_{i,j}(t)$.
A schematic presentation of the model including subcompartments, omitting age group visualization, is depicted in~\cref{fig:LCTSECIR}.

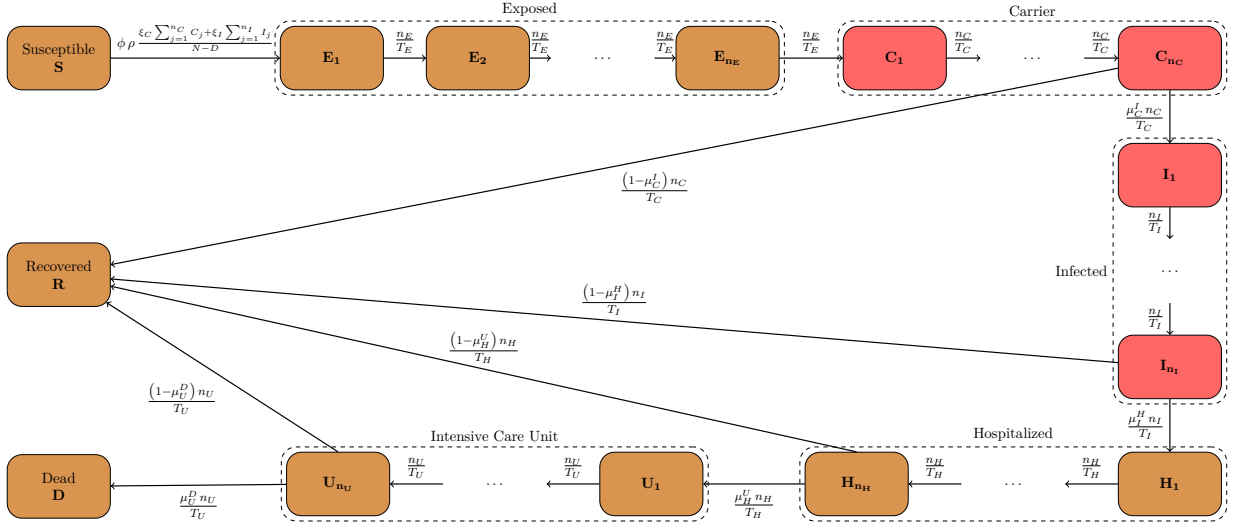
\begin{figure}[tb]
\scalebox{.56}{
\centering
 \newcommand{\EQPLT}[1]{#1} 
\tikzstyle{roundbox}=[draw, fill=orange, rounded corners=10pt, minimum size=2.4cm, text width = 2.2cm, align = center, minimum height = 1.5cm]
	\tikzstyle{arrow}=[->, black, thick, text = black]
\begin{tikzpicture}[auto]
	\node [roundbox] (s) {Susceptible \\ $\mathbf{S}$};

    \node [roundbox, right = 4cm of s] (e1) {$\mathbf{E_{1}}$};
    \node [roundbox, right =  1cm of  e1] (e2) {$\mathbf{E_{2}}$};
    \node [roundbox,fill=none,draw=white, right = 0.5cm of  e2](edots){$\mathbf{\dots}$};
    \node [roundbox, right = 0.5cm of edots](en)  {$\mathbf{E_{n_E}}$};
    \node [roundbox, dashed ,fill=none, fit=(e1.north west) (en.south east),label=above:{Exposed}] (e) {};

	\draw [arrow] (s) -- node [above, font=\small] {\EQPLT{$ \phi\, \rho\,\frac{\xi_{C}\sum_{j=1}^{n_{C}}C_{j}+\xi_{I}\sum_{j=1}^{n_{I}}I_{j}}{N-D}$}}(e1);
    \draw [arrow] (e1) -- node [above, font=\large] {\EQPLT{$ \frac{n_E}{T_E}$}}(e2);
    \draw [arrow] (e2) -- node [above, font=\large] {\EQPLT{$ \frac{n_E}{T_E}$}}(edots);
    \draw [arrow] (edots) -- node [above, font=\large] {\EQPLT{$ \frac{n_E}{T_E}$}}(en);

    \node [roundbox, right = 1.5cm of en,fill=red!60] (c1) {$\mathbf{C_{1}}$};
    \node [roundbox,fill=none,draw=white, right = 0.8cm of  c1](cdots){$\mathbf{\dots}$};
    \node [roundbox, right = 0.8cm of cdots,fill=red!60](cn)  {$\mathbf{C_{n_{C}}}$};
    \node [roundbox, dashed ,fill=none, fit=(c1.north west) (cn.south east),label=above:{Carrier}] (c) {};

	\draw [arrow] (en) -- node [above, font=\large] {\EQPLT{$ \frac{n_E}{T_E}$}}(c1);
    \draw [arrow] (c1) -- node [above, font=\large] {\EQPLT{$\frac{n_{C}}{T_{C}}$}}(cdots);
    \draw [arrow] (cdots) -- node [above, font=\large] {\EQPLT{$\frac{n_{C}}{T_{C}}$}}(cn);

    \node [roundbox, below = 1.25cm of cn,fill=red!60] (i1) {$\mathbf{I_{1}}$};
    \node [roundbox,fill=none,draw=white, below = 0.75cm of  i1](idots){$\mathbf{\dots}$};
    \node [roundbox,below = 0.75cm of idots,fill=red!60](in)  {$\mathbf{I_{n_{I}}}$};
    \node [roundbox, dashed ,fill=none, fit=(i1.north west) (in.south east),label=left:{Infected}] (c) {};

	\draw [arrow] (cn) -- node [left, font=\large] {\EQPLT{$ \frac{\mu_{C}^{I}\,n_{C}}{T_{C}}$}}(i1);
    \draw [arrow] (i1) -- node [left, font=\large] {\EQPLT{$ \frac{n_{I}}{T_{I}}$}}(idots);
    \draw [arrow] (idots) -- node [left, font=\large] {\EQPLT{$ \frac{n_{I}}{T_{I}}$}}(in);

    \node [roundbox, below = 1.25cm of in] (h1) {$\mathbf{H_{1}}$};
    \node [roundbox,fill=none,draw=white, left = 1.25cm of  h1](hdots){$\mathbf{\dots}$};
    \node [roundbox, left =1.25cm of hdots](hn)  {$\mathbf{H_{n_{H}}}$};
    \node [roundbox, dashed ,fill=none, fit=(hn.north west) (h1.south east),label=above:{Hospitalized}] (h) {};

	\draw [arrow] (in) -- node [left, font=\large] {\EQPLT{$ \frac{\mu_{I}^{H}\,n_{I}}{T_{I}}$}}(h1);
    \draw [arrow] (h1) -- node [above, font=\large] {\EQPLT{$ \frac{n_{H}}{T_{H}}$}}(hdots);
    \draw [arrow] (hdots) -- node [above, font=\large] {\EQPLT{$ \frac{n_{H}}{T_{H}}$}}(hn);

    \node [roundbox, left = 2.4cm of hn] (u1) {$\mathbf{U_{1}}$};
    \node [roundbox,fill=none,draw=white, left = 1.25cm of  u1](udots){$\mathbf{\dots}$};
    \node [roundbox, left = 1.25cm of udots](un)  {$\mathbf{U_{n_{U}}}$};
    \node [roundbox, dashed ,fill=none, fit=(un.north west) (u1.south east),label=above:{Intensive Care Unit}] (u) {};

	\draw [arrow] (hn) -- node [below, font=\large] {\EQPLT{$ \frac{\mu_{H}^{U}\,n_{H}}{T_{H}}$}}(u1);
    \draw [arrow] (u1) -- node [above, font=\large] {\EQPLT{$ \frac{n_{U}}{T_{U}}$}}(udots);
    \draw [arrow] (udots) -- node [above, font=\large] {\EQPLT{$ \frac{n_{U}}{T_{U}}$}}(un);

   \node [roundbox, below = 3.6cm of s] (r) {Recovered \\ $\mathbf{R}$};
	\draw [arrow] (cn) -- node [below right, font=\large] {\EQPLT{$ \frac{\left(1-\mu_{C}^{I}\right)\,n_{C}}{T_{C}}$}}(r);
     \draw [arrow] (in) -- node [above, font=\large] {\EQPLT{$ \frac{\left(1-\mu_{I}^{H}\right)\,n_{I}}{T_{I}}$}}(r);
     \draw [arrow] (hn.north) -- node [above, font=\large] {\EQPLT{$ \frac{\left(1-\mu_{H}^{U}\right)\,n_{H}}{T_{H}}$}}(r);
     \coordinate[xshift=-1mm,yshift=+1mm] (faker) at (r.south east);
     \draw [arrow] (un.north) -- node [below left, font=\large] {\EQPLT{$\frac{\left(1-\mu_{U}^{D}\right)\,n_{U}}{T_{U}}$}}(faker);
    
    \node [roundbox, below = 8.6cm of s] (d) {Dead \\ $\mathbf{D}$};
    \draw [arrow] (un) -- node [below, font=\large] {\EQPLT{$ \frac{\mu_{U}^{D}\,n_{U}}{T_{U}}$}}(d);

\end{tikzpicture}}
\caption{\textbf{Structure of the LCT-SECIR model, omitting age groups visualization.} Schematic illustration of the possible transitions between compartments and subcompartments according to the LCT-SECIR model. For the sake of clarity, we have omitted the indices for age groups. The subcompartments, in which individuals are infectious and can infect people from the Susceptible compartment, are highlighted in red. A description of the model parameters can be found in~\cref{tab:parameters}.}
\label{fig:LCTSECIR}
\end{figure}\normalsize 
\begin{table}
    \centering
    \begin{tabular}{c | c}
     \hline
     Parameter &  Description\\
     \hline
     $Z_{i,j}(t)$ &  Number of people of age group $i$ in subcompartment $j$ of compartment $Z$\\& at simulation time $t$. \\
      $n_{Z,i}$ & Number of subcompartments of the compartment $Z$ of age group $i$.\\
      $\phi_{i,k}(t)$ & Average number of daily contacts of a person of age group $i$\\& with persons from group $k$ at simulation time $t$. \\
     $\rho_i(t)$ & Transmission risk on contact of age group $i$ at simulation time $t$.\\
    $\xi_{C,i}(t)$ & Proportion of Carrier individuals of age group $i$ not isolated at simulation time $t$.\\
    $\xi_{I,i}(t)$& Proportion of Infected individuals of age group $i$ not isolated at simulation time $t$.\\
    $N_i(t)$ & Total number of living people of age group $i$ at simulation time $t$. \\
      $T_{Z,i}$ & Average stay time in days in compartment $Z$ of individuals of age group $i$.\\
     $\mu_{Y_i}^{Z_i}$ & Expected probability of transition from compartment $Y$ to $Z$ of age group $i$.
     \end{tabular}
\caption{\textbf{Description of the parameters used in the LCT-SECIR model.}}\label{tab:parameters}
\end{table}

Finally, we define the model equations of the LCT-SECIR model for each age group $i\in\{1,\dots,m\}$ as
\begin{align} 
  \label{eq:LCTSECIR}
    S_i'(t) &=- S_i(t) \,\rho_i(t)\,\sum_{k=1}^{m}\frac{1}{N_k(t)}\,\phi_{i,k}(t)\, \Big( \xi_{C,k}(t) \, C_{k,*}(t)+\xi_{I,k}(t)\, I_{k,*}(t)\Big)
    \nonumber\\
    E_{i,1}'(t) &= S_i(t) \,\rho_i(t)\,\sum_{k=1}^{m}\frac{1}{N_k(t)}\,\phi_{i,k}(t)\, \Big( \xi_{C,k}(t) \, C_{k,*}(t)+\xi_{I,k}(t)\, I_{k,*}(t)\Big)
            - \frac{n_{E,i}}{T_{E,i}}\, E_{i,1}(t)
            \nonumber\\
    E_{i,j}'(t) &= \frac{n_{E,i}}{T_{E,i}}\, E_{i,j-1}(t)
            - \frac{n_{E,i}}{T_{E,i}}\, E_{i,j}(t) \hspace{7em} \text{ for }j\in\{2,\dots,n_{E,i}\}
            \nonumber\\[3pt]
    C_{i,1}'(t) &= \frac{n_{E,i}}{T_{E,i}}\, E_{i,n_{E,i}}(t)
            -\frac{n_{C,i}}{T_{C,i}}\, C_{i,1}(t) 
            \nonumber\\[3pt]
    C_{i,j}'(t) &= \frac{n_{C,i}}{T_{C,i}}\, C_{i,j-1}(t)
            - \frac{n_{C,i}}{T_{C,i}}\, C_{i,j}(t)\hspace{7.1em}  \text{ for }j\in\{2,\dots,n_{C,i}\}
            \nonumber\\[3pt]
    I_{i,1}'(t) &= \mu_{C_i}^{I_i}\,\frac{n_{C,i}}{T_{C,i}}\,             C_{i,n_{C,i}}(t) 
            -\frac{n_{I,i}}{T_{I,i}}\, I_{i,1}(t)
            \nonumber\\[3pt]
    I_{i,j}'(t) &= \frac{n_{I,i}}{T_{I,i}}\, I_{i,j-1}(t)
            - \frac{n_{I,i}}{T_{I,i}}\, I_{i,j}(t) \hspace{8.01em} \text{ for }j\in\{2,\dots,n_{I,i}\}
            \\[3pt]
    H_{i,1}'(t)&= \mu_{I_i}^{H_i}\,\frac{n_{I,i}}{T_{I,i}}\,                I_{i,n_{I,i}}(t)
            - \frac{n_{H,i}}{T_{H,i}}\,  H_{i,1}(t)
            \nonumber\\[3pt]
    H_{i,j}'(t) &= \frac{n_{H,i}}{T_{H,i}}\, H_{i,j-1}(t)
            - \frac{n_{H,i}}{T_{H,i}}\, H_{i,j}(t) \hspace{6.66em}\text{ for }j\in\{2,\dots,n_{H,i}\}
            \nonumber\\[3pt]
    U_{i,1}'(t)&=  \mu_{H_i}^{U_i}\,\frac{n_{H,i}}{T_{H,i}} \,            H_{i,n_{H,i}}(t)
            -\frac{n_{U,i}}{T_{U,i}}\, U_{i,1}(t)
            \nonumber\\[3pt]
    U_{i,j}'(t) &= \frac{n_{U,i}}{T_{U,i}}\, U_{i,j-1}(t)
            - \frac{n_{U,i}}{T_{U,i}}\, U_{i,j}(t)\hspace{7.16em} \text{ for }j\in\{2,\dots,n_{U,i}\}
            \nonumber\\[3pt]
    R_i'(t)&=  \left(1-\mu_{C_i}^{I_i}\right)\frac{n_{C,i}}{T_{C,i}}\,C_{i,n_{C,i}}(t)
            + \left(1-\mu_{I_i}^{H_i}\right)\frac{n_{I,i}}{T_{I,i}}\, I_{i,n_{I,i}}(t)
            \nonumber\\*
            &\quad+ \left(1-\mu_{H_i}^{U_i}\right)\frac{n_{H,i}}{T_{H,i}}\, H_{i,n_{H,i}}(t)
            + \left(1-\mu_{U_i}^{D_i}\right)\frac{n_{U,i}}{T_{U,i}}\, U_{i,n_{U,i}}(t) 
            \nonumber\\[1pt]
   D_i'(t)&=  \mu_{U_i}^{D_i}\,\frac{n_{U,i}}{T_{U,i}} \, U_{i,n_{U,i}}(t) \nonumber
\end{align}
for $t\geq0$, where
\begin{align*}
Z_{i,*}(t):=\sum_{j=1}^{n_{Z,i}}Z_{i,j}(t)
\end{align*}
is the total number of individuals of age group $i$ in compartment $Z\in \mathcal{A}$. Hereby, $N_i(t):=\sum_{Z\in\mathcal{Z}\setminus D}Z_i(t)$ is the total number of living people of age group $i$ where $N_i^+:=N_i(t)+D_i(t)$ is constant in time. The model excludes birth and disease unrelated death events. The parameters $\rho_i(t)\in[0,1]$ refer to the average transmission risk on a contact of age group $i$ at simulation time $t$. The entry $\phi_{i,k}(t)\geq0$ in the contact matrix $\phi(t)$ is the average number of daily contacts that a person of age group $i$ has with people belonging to age group $k$. Additionally, $\xi_{C,i}(t)\in[0,1]$ and $\xi_{I,i}(t)\in[0,1]$ represent the average proportion of Carrier and Infected individuals, respectively, of age group $i$ that are not isolated at simulation time $t$. The parameters $T_{Z,i}\geq 0$ depict the average stay time in days in compartment $Z\in\mathcal{A}$ for each age group $i$. The last set of remaining parameters to be described is $\mu_{Y_i}^{Z_i}\in[0,1]$, which is the expected probability for individuals of age group $i$ to move from disease state $Y\in\mathcal{Z}$ to a consecutive state $Z\in\mathcal{Z}$. Note that the parameter is only defined if a transition from $Y$ to $Z$ is possible according to~\cref{fig:LCTSECIR}. An overview of the model parameters can be found in~\cref{tab:parameters}.

\section{Properties of the LCT model}\label{sec:properties}
This section presents a review of the most significant properties of the proposed LCT-SECIR model~\eqref{eq:LCTSECIR}. The objective is to establish a foundation for explaining the differences in the simulation results that will be discussed in the following section. The findings are not exclusive to our model; they can be applied to general uses of the LCT.

We begin by demonstrating that the model formulation indeed results in Erlang distributed stay times in each compartment. The proofs can be done with elementary calculus, stochastics and solution theory for differential equations. For detailed proofs of the theorems presented, see, e.g.,~\cite{plotzke_ma_2023}. In order to provide a comprehensive overview, we recall the defining characteristics of an Erlang distribution.
\begin{remark}
   The probability density function of the Erlang distribution is given by
\begin{align*}
    f_{\lambda,\alpha}(x)=\begin{cases}
        \frac{\lambda^{\alpha}}{(\alpha-1)!}\, x^{\alpha-1}\,e^{-\lambda\,x} &\text{ for }x\geq0\\
        0&\text{ for }x<0
    \end{cases}
\end{align*} with a rate parameter $\lambda\in\mathbb{R}^{+}$ and an integer shape parameter $\alpha\in\mathbb{N}$. The cumulative distribution function of the Erlang distribution is \begin{align*}
    F_{\lambda,\alpha}(x)&=1-e^{-\lambda\, x}\sum_{j=1}^{\alpha} \frac{(\lambda \, x)^{j-1}}{(j-1)!}=1-\sum_{j=1}^{\alpha} \frac{1}{\lambda}\,f_{\lambda,j}(x)
\end{align*} for $x\geq0$.
\end{remark}
\begin{remark}
The Erlang distribution is a special case of the gamma distribution, with the restriction that only integer shape parameters are allowed. Therefore, the Linear Chain Trick is also sometimes called Gamma Chain Trick.
\end{remark}
As a first step to show that the overall stay time in a compartment is Erlang distributed, we examine the stay time distribution in each subcompartment.
\begin{theorem}\label{thm:subcompartments}
For each compartment $Z\in\mathcal{A}$ and age group $i\in\{1,\dots,m\}$, let $X_{Z,i,j}$ be the random variable describing the stay time in subcompartment $Z_{i,j}$ for each $j\in\{1,\dots,n_{Z,i}\}$. 

\noindent Then, the random variable $X_{Z,i,j}$ is exponentially distributed with parameter $\frac{n_{Z,i}}{T_{Z,i}}$ for each $j\in\{1,\dots,n_{Z,i}\}$.
\end{theorem}
\begin{proof}
    For this proof, we omit the age index $i$. Let $\sigma_{Z,j}(t)$ be the number of people entering subcompartment $Z_j$ at time $t$. Reformulating~\eqref{eq:LCTSECIR} leads to 
\begin{align}\label{eq:proof1}
   Z_{j}'(t)= \sigma_{Z,j}(t)
            -\frac{n_{Z}}{T_{Z}}\,Z_{j}(t).
\end{align} Let the function $\gamma_{Z,j}(\tau)=\mathbb{P}(X_{Z,j}>\tau)$ describe the probability that an individual is still in subcompartment $Z_j$ after $\tau$ days since entering. Therefore, $1-\gamma_{Z,j}(\tau)$ is the cumulative density function (CDF) of the random variable $X_{Z,j}$. Applying the definition of $\gamma_{Z,j}(\tau)$, the equality
\begin{align}\label{eq:proof2}
     Z_j(t)= \int_{0}^{\infty}\gamma_{Z,j}(\tau)\,\sigma_{Z,j}(t-\tau)\,\mathrm{d}\tau
\end{align} must hold. Together with the condition $\gamma_{Z,j}(0)=1$ resulting from the CDF property, equation~\eqref{eq:proof2} is equivalent to~\eqref{eq:proof1} if 
\begin{align*}
   \gamma_{Z,j}'(t)= - \frac{n_{Z}}{T_{Z}}\,\gamma_{Z,j}(t)
\end{align*} is satisfied. The unique exponential solution to this initial value problem finalizes the proof.
\end{proof}
Additionally, we need the statement that an Erlang distribution can be considered as a sum of exponential distributions, see also~\cite{hurtado_generalizations_2019}.
\begin{theorem}\label{thm:sum}
Let a set of random variables, $X_i$, with $i\in\{1,\dots,n\}$ and $n\in\mathbb{N}$, be independent and identically distributed according to the exponential distribution with parameter $\lambda\in\mathbb{R}_{>0}$. 

\noindent Then, the sum of the random variables, $X=\sum_{i=1}^{n}X_i$, is Erlang distributed with rate $\lambda$ and shape $n$.
\end{theorem}
\begin{proof}
    The proof can be conducted by induction, using the calculation rules for the density function of the sum of two independent random variables.
\end{proof}

Finally, the combination of the two preceding theorems yields the following corollary. The corollary demonstrates that the application of the LCT indeed results in Erlang distributed stay times within the compartments.
\begin{corollary}\label{thm:erlang}
    For each compartment $Z\in\mathcal{A}$ and age group $i\in\{1,\dots,m\}$, let 
    \begin{align*}
        X_{Z,i}=\sum_{j=1}^{n_{Z,i}}X_{Z,i,j}
    \end{align*}be the random variable representing the overall stay time in the compartment.

    \noindent Then, the random variable $X_{Z,i}$ is Erlang distributed with rate parameter $\frac{n_{Z,i}}{T_{Z,i}}$ and shape parameter $n_{Z,i}$.
\end{corollary}

\begin{figure}[bt]
\begin{minipage}[t]{0.5\textwidth}
    \centering
    \includegraphics[width=\textwidth]{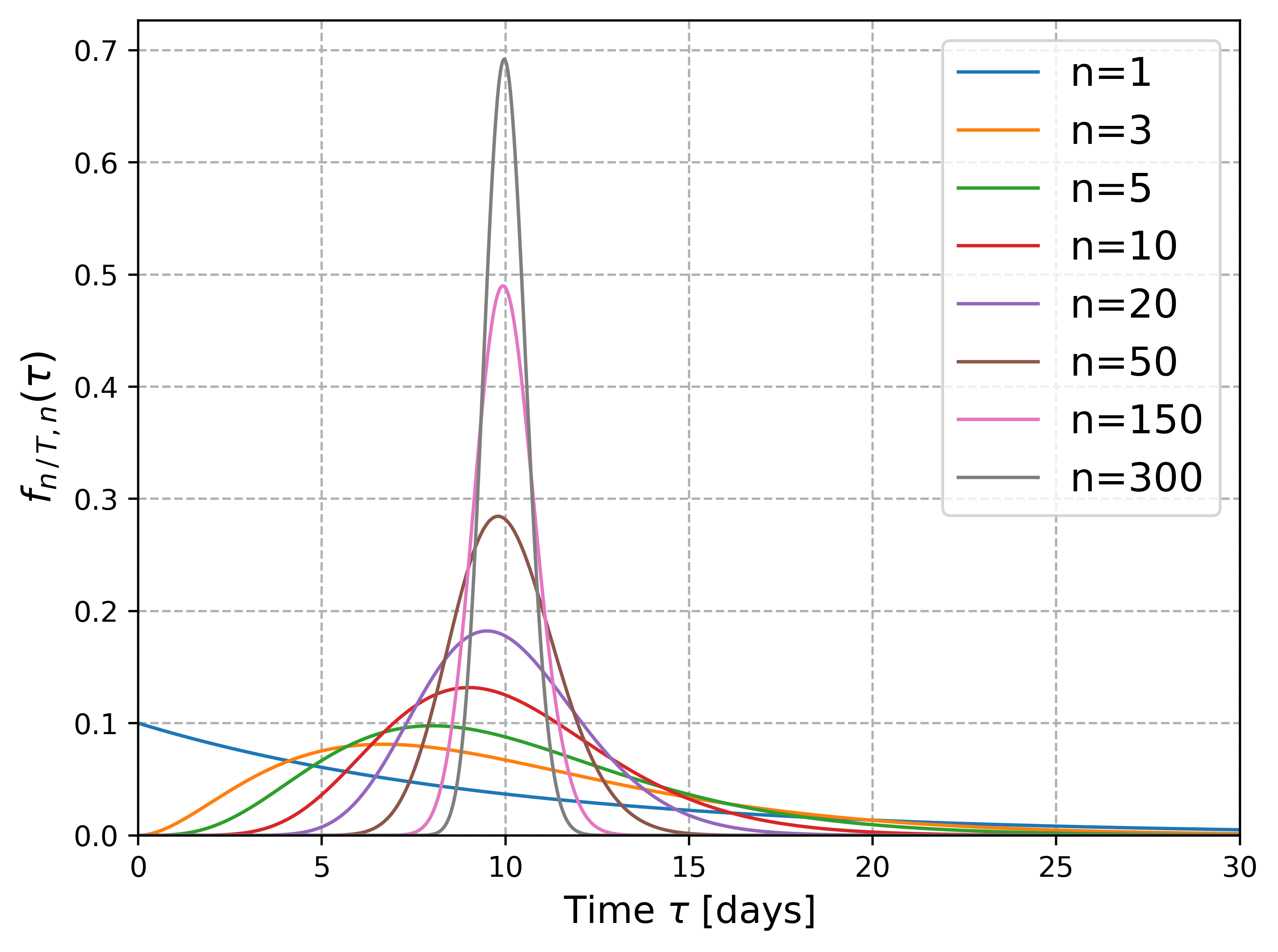}
\end{minipage}
\begin{minipage}[t]{0.5\textwidth}
 \centering
    \includegraphics[width=\textwidth]{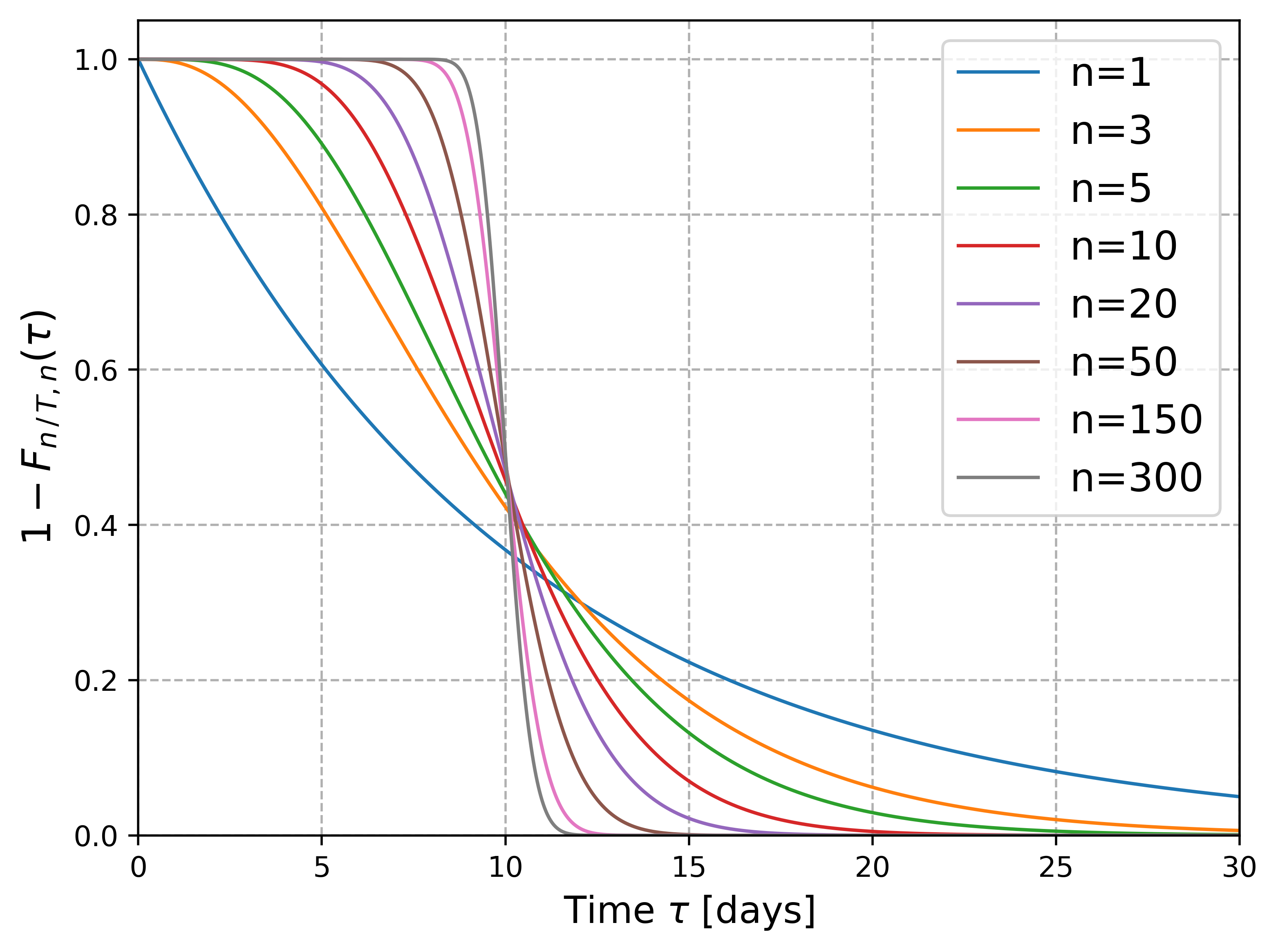}
\end{minipage}
\caption{\textbf{Density and survival function in an LCT model.} Representation of the density function $f_{n/T,\,n}(\tau)$ of the Erlang distribution (left) and the associated survival function $1-F_{n/T,\,n}(\tau)$ (right) for different choices of the parameter $n$. The average stay time $T=10$ is set for all functions. Here, we omit the indices for compartments and age groups.}
    \label{fig:Erlang}
\end{figure}

One can also show that a SECIR-type model based on integro-differential equations as the one presented in~\cite{wendler2024nonstandardnumericalschemenovel} can be reduced to an LCT model if all stay time distributions are chosen Erlang distributed. An idea of the proof can be taken from~\cite[Appendix]{feng_mathematical_2016}. Additionally, it is obvious that the LCT-SECIR model is a generalization of an ODE-SECIR model as of~\cite{kuhn_assessment_2021}. The choice of only one subcompartment, $n_{Z,i}=1$, for all compartments $Z\in\mathcal{A}$ and each age group $i$, leads to an ODE model with exponentially distributed stay times. Therefore, the LCT model is a generalization of a corresponding ODE model and a specialization of a corresponding IDE model. That means, we can include more general stay time distribution assumptions without the need to use more complicated integro-differential equations.

\cref{thm:erlang} implies that the mean stay time in compartment $Z_i$ for $Z\in\mathcal{A}$ and age group $i$, 
\begin{align}
   \mathbb{E}(X_{Z,i})=T_{Z,i}, \label{eq:mean}
\end{align} matches the definition of $T_{Z,i}$. The variance is given by
\begin{align}\label{eq:variance}
    \mathrm{Var}\,(X_{Z,i})=\frac{(T_{Z,i})^2}{n_{Z,i}}.
\end{align}

The variance~\eqref{eq:variance} of the \LP{Erlang} distribution decreases for an increasing number of subcompartments \LP{for a fixed parameter $T_{Z,i}$}. This can also be observed in~\cref{fig:Erlang}, where the probability density and the survival function of Erlang distributions with a fixed mean and varying numbers of subcompartments are plotted. For $n_{Z,i}\to\infty$ the variance tends toward zero and the Erlang distribution converges to the Delta distribution. This implies a fixed stay time $T_{Z,i}$ in the respective compartment $Z_i$.

The model parameters $T_{Z,i}$ and $n_{Z,i}$ can be determined in the case of known mean and variance.  \LP{Note that the constraint of the Erlang distribution, where the number of subcompartments must be a natural number, could mean that the observed variance can only be approximated, i.e., if the mean stay time $T_{Z,i}$ is set with exact precision according to~\eqref{eq:mean}, then, inserting the observed variance might result in a noninteger value for the number of subcompartments. A solution to prevent rounded values is an approach by Cassidy et al.~\cite{cassidy_numerical_2022}, in which hypoexponential distributions instead of Erlang distributions are used to preserve the mean and variance accurately, also maintaining an ODE model formulation. Another alternative is given by Hurtado and Kirosingh~\cite{hurtado_generalizations_2019}, who present the \textit{Generalized Linear Chain Trick}, which allows for phase-type distributions. Any positive-valued stay time distribution can be approximated with arbitrary accuracy, while the model is still formulated by an ODE-system. 
}

\begin{remark} \label{remark:epi_meaning}
   In some published works, the subcompartments are described as purely mathematical constructs utilized to achieve an Erlang distributed stay time in the disease states, see e.g.~\cite{lloyd_dependence_2001,ma_generality_2006}. \MK{While the} subcompartments do not \MK{represent a particular} biological \MK{state}, this statement can be questioned\MK{, nevertheless}.
   
   The subcompartments can be assigned a biological \MK{interpretation} in the sense that the expected remaining stay time in the disease state $Z\in\mathcal{A}$ depends on the current subcompartment, as the following consideration shows. Let us consider an individual of age group $i$ in the subcompartment $Z_{i,j}$ with $j\in\{1,\dots,n_{Z,i}\}$.
   For the individual, there are $n_{Z,i}-j$ \LP{remaining} subcompartments in compartment $Z_i$ which are ordered subsequently according to the course of the disease. \LP{T}he expected remaining stay time in $Z_{i,j}$ is the same as for the subsequent subcompartments. Therefore, by applying the formula for the mean of exponential distributions, together with the linearity of mean values and the result of~\cref{thm:subcompartments}, we obtain that the expected remaining stay time in $Z_i$ for an individual in $Z_{i,j}$ is 
   \begin{align*}
       \left(n_{Z,i}-j+1\right)\frac{T_{Z,i}}{n_{Z,i}}.
   \end{align*}
   The remaining stay time is different for different subcompartments, as the result depends on $j$. 
\end{remark}
There is substantial evidence that, for most infectious diseases, the Erlang distribution is more realistic than the exponential distribution for the stay times in disease states,  cf.~\cite{wearing_appropriate_2005,lloyd_realistic_2001,wang_evaluations_2017,donofrio_mixed_2004}. Real distributions tend to have a lower variance than the exponential distribution, making other distributions such as Erlang distributions more suitable~\cite{lloyd_realistic_2001}.
The \textit{memoryless property} of the exponential distribution could be a key factor in explaining why the assumption of an exponentially distributed stay time is considered as unrealistic. That means, that the expected remaining stay time in a compartment is independent of the time already spent. For the Erlang distribution, this expected remaining time decreases the longer the time already spent, which is more realistic for most infectious diseases. For more details on the last paragraph, see also~\cite{feng_epidemiological_2007}.

\section{Numerical simulations}\label{sec:numerics}
In this section, we conduct numerical experiments to evaluate the impact of a more realistic distribution assumption and the use of age groups on the simulation results. Furthermore, we present a scenario inspired by the spread of COVID-19 in Germany to illustrate the utility of this approach for realistic simulations. \MK{Note that across the different sections, the simulation time differs to set different foci, i.e., short- or long-term. In~\cref{sec:change,sec:covid} and in the first part of~\cref{sec:ageresolution}, short-term developments are considered as, in general, interventions are implemented to prevent epidemic peaking right from the start. In~\cref{sec:peaks} and the second part of~\cref{sec:ageresolution}, we consider the more theoretical outcomes of epidemic peaks and final sizes in long-term scenarios in the what-if no intervention scenario.}

The proposed age-resolved LCT-SECIR model, along with the related numerical scenarios, are incorporated open-source into our high-performance, modular epidemics simulation software MEmilio~\cite{memiliov1.3}. \MK{Our model, as well as MEmilio, is written in efficient C++ to allow for fast and scalable execution. Together with already existing software infrastructure of MEmilio, we can use distributed memory parallelism through the MPI standard to execute ensemble runs in parallel, which is also demonstrated in~\cref{sec:runtime}.}
\subsection{Parameter selection and data}\label{sec:parameters_data}
\begin{table}[tb]
    \centering
\def\arraystretch{1.25}
\begin{tabular}{>{\columncolor{gray!20}}c | >{\centering}p{1.4cm}>{\centering}p{1.4cm}>{\centering}p{1.4cm}>{\centering}p{1.4cm} >{\centering}p{1.4cm}>{\centering}p{1.4cm}|>{\centering\arraybackslash}p{1.8cm}}
\toprule[0.8pt]
     \rowcolor{gray!20} Parameter &  $0$--$4$ & $5$--$14$ & $15$--$34$ & $35$--$59$ & $60$--$79$ & $80+$ & Weighted Average \\
\midrule[0.8pt]
    $\rho_i(t)$ & $0.03$ & $0.06$ & $0.06$ & $0.06$ & $0.09$ & $0.175$ & $0.07333$\\
\midrule
   $\mu_{C_i}^{I_i}$ & $0.75$ & $0.75$ & $0.8$ & $0.8$ & $0.8$ & $0.8$ & $ 0.79310$\\
     $\mu_{I_i}^{H_i}$ & $0.0075$ & $0.0075$ & $0.019$ & $0.0615$ & $0.165$ & $0.225$ & $0.07864$\\
     $\mu_{H_i}^{U_i}$ & $0.075$ & $0.075$ & $0.075$  & $0.15$ & $0.3$ & $0.4$ & $0.17318$\\
     $\mu_{U_i}^{D_i}$ & $0.05$ & $0.05$ & $0.14$ & $0.14$ & $0.4$ & $0.6$ & $0.21718$\\
\midrule
    $T_{E,i}$ & $3.335$ & $3.335$ & $3.335$ & $3.335$ & $3.335$ & $3.335$ & $3.335$\\
     $T_{C,i}$ & $2.74$ & $2.74$ & $2.565$ & $2.565$ & $2.565$ & $2.565$ & $2.58916$\\
      $T_{I,i}$ & $7.02625$ & $7.02625$ & $7.0665$ & $6.9385$ & $6.835$ & $6.775$ & $6.94547$\\
    $T_{H,i}$ & $5$ & $5$ & $5.925$ & $7.55$ & $8.5$ & $11$ & $7.28196$\\
     $T_{U,i}$ & $6.95$ & $6.95$ & $6.86$ & $17.36$ & $17.1$ & $11.6$ & $13.066$ \\
\bottomrule
\end{tabular}
\caption{\textbf{Age-resolved parameters for wild-type SARS-CoV-2.} Epidemiological parameters used to simulate the spread of SARS-CoV-2 in Germany in the year $2020$. The parameters are either directly set or calculated based on~\cite[Table~2]{kuhn_assessment_2021}. For the average values, the age specific values are weighted in accordance with the relative share of the age group in the total population according to~\cite{regionaldatenbank_deutschland_fortschreibung}. }\label{tab:COVID-19parameters}
\end{table}
For the numerical simulations, we use parameters and data on the spread of the SARS-CoV-2 virus in Germany in $2020$. For SARS-CoV-2, the Robert Koch Institute (RKI) publishes daily, age-resolved data on the total number of confirmed cases and deaths in Germany~\cite{RKI_data_2024}. Furthermore, we incorporate data on COVID-19 patients in intensive care unit that is reported by~\cite{divi2024}. Our model utilizes six age groups as defined by the RKI data. The population sizes $N_i^+$ are set according to~\cite{regionaldatenbank_deutschland_fortschreibung}. 

We adopt age-resolved transition probabilities, mean stay times and the transmission probabilities from~\cite[Table~2]{kuhn_assessment_2021} for an ODE-based model. Note that in~\cite{kuhn_assessment_2021}, the mean stay time could be dependent not only on the starting compartment but also on the destination compartment. However, the parameters $\mu_{Y_i}^{Z_i}$ then lose their interpretation as probabilities, which is a consequence of the memoryless property. To preserve the original interpretation, we calculate our required mean stay times by weighting the given mean stay times with the given probabilities. The values obtained for the epidemiological parameters are presented in~\cref{tab:COVID-19parameters}. In scenarios where we require parameters that are not stratified by age, the age-resolved parameters are weighted in accordance with the relative share of the age group in the total population. The results are also shown in~\cref{tab:COVID-19parameters}. We set $\xi_{C,i}(t)=1$ and $\xi_{I,i}(t)=0.3$ for all age groups $i\in\{1,\dots,m\}$.
This implies that individuals who are not symptomatic do not isolate themselves, whereas those who are symptomatic do so more often.

\LP{For the LCT model, we assess different assumptions regarding the number of subcompartments. We use the notation \textit{LCTX} for an LCT model with $X=n_{Z,i}$ subcompartments for all compartments $Z\in\mathcal{A}$ and age groups $i\in\{1,\dots,m\}$. Furthermore, we consider an LCT model applying an idea used in~\cite{keeling_understanding_2002}, i.e., for every compartment $Z \in \mathcal{A}$, the number of subcompartments is chosen such that $n_{Z,i}\approx T_{Z,i}$. For every value, the mean stay time in~\cref{tab:COVID-19parameters} is rounded to the nearest integer value. We denote this LCT model by \textit{LCTvar} as the subcompartment numbers are variable according to the mean stay times.} \AW{The ideal way to set the number of subcompartments would be to use the corresponding variances to the mean stay times from~\cref{tab:COVID-19parameters} and set the numbers of subcompartments accordingly by applying~\cref{eq:variance}. Unfortunately, the appropriate variances are not available for all mean stay times, which is why we proceed as described. Although all models are based on ODE systems, we use the simplified notation \textit{ODE} to refer to a simple ODE-based model without Linear Chain Trick. Note that notation \textit{LCT$1$} corresponds to \textit{ODE}.}

\LP{T}he contact matrix $\phi(t)$ as well as the initial values are provided for each numerical experiment individually.
Except for~\cref{sec:runtime}, the ODE system\LP{s} describing the models are solved using a Runge-Kutta scheme of fifth order with a fixed time step of $\Delta t=10^{-2}$. 

\subsection{Impact of the distribution assumption on model behavior}
In order to compare the qualitative behavior of LCT models against simple ODE models, we examine the dynamics at change points and analyze epidemic peaks \MK{and final sizes}. To investigate only the effect of the distribution assumptions, the population is not divided into age groups for these experiments. \LP{We therefore omit the age index in the notation of the parameters.} 
\subsubsection{Behavior at change points}\label{sec:change}
\begin{figure}[tb]
    \centering
    \begin{minipage}[t]{0.47\textwidth}
    \centering
     \includegraphics[width=\textwidth]{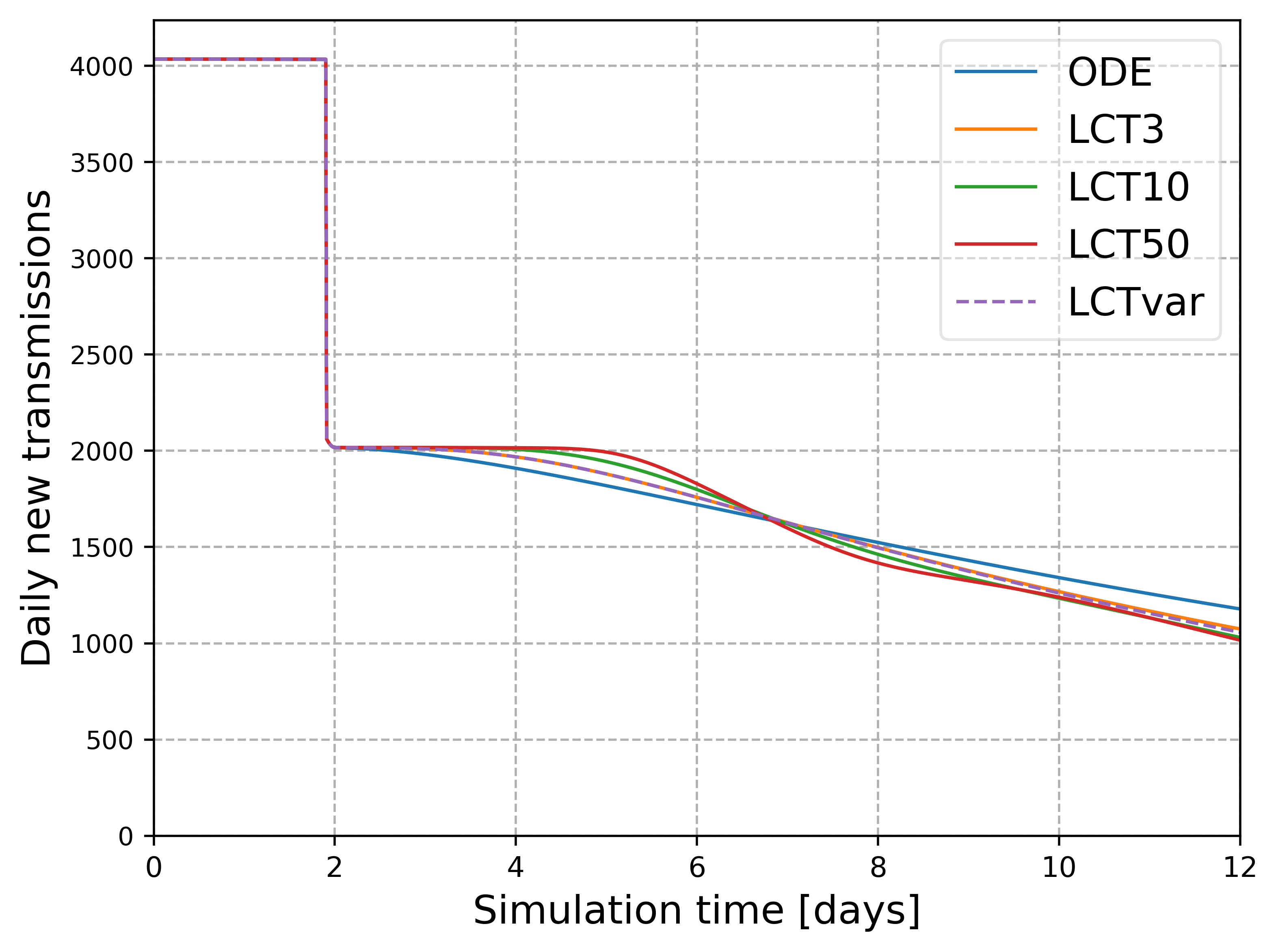}
\end{minipage}
\begin{minipage}[t]{0.47\textwidth}
 \centering
    \includegraphics[width=\textwidth]{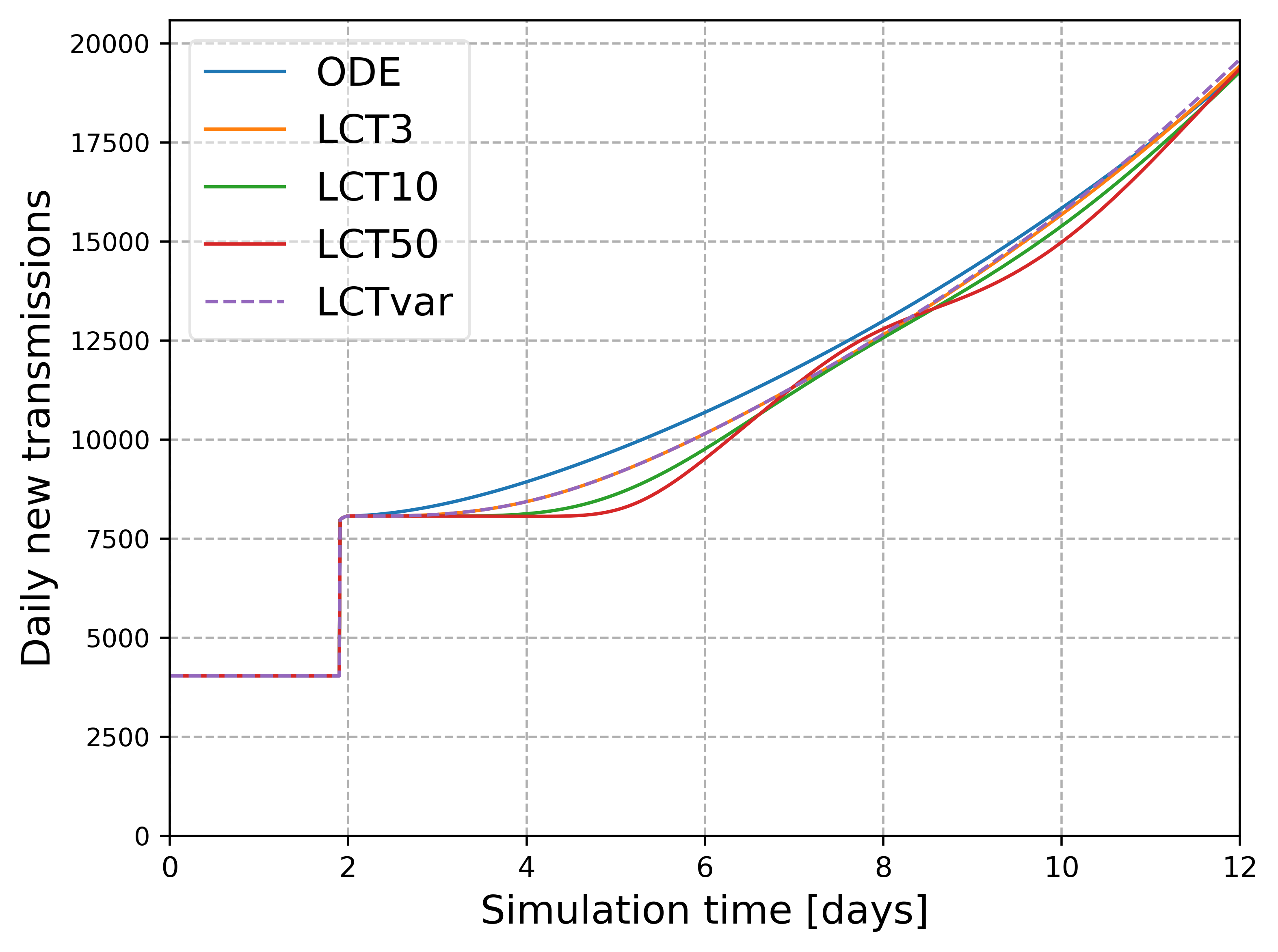}
\end{minipage}  
\caption{\textbf{Daily new transmissions \MK{around} change points.} Comparison of the daily new transmissions, i.e., the number of people transiting from compartment $S$ to $E$ within one simulation day, of different LCT models against a simple ODE model at change points. The contact rate $\phi(t)$ is halved (left) or doubled (right) after the second simulation day. The naming of the LCT models refer to different assumptions regarding the number of subcompartments, e.g., LCT$3$ refers to an LCT model with $n_Z=3$ subcompartments for each compartment $Z\in\mathcal{A}$ and ODE corresponds to LCT$1$. \LP{LCTvar refers to an LCT model with $n_{Z}\approx T_{Z}$ for each compartment $Z\in\mathcal{A}$.}}
    \label{fig:changepoints}
\end{figure}
\begin{figure}[tb]
    \centering
    \begin{minipage}[t]{0.45\textwidth}
    \centering
     \includegraphics[width=\textwidth]{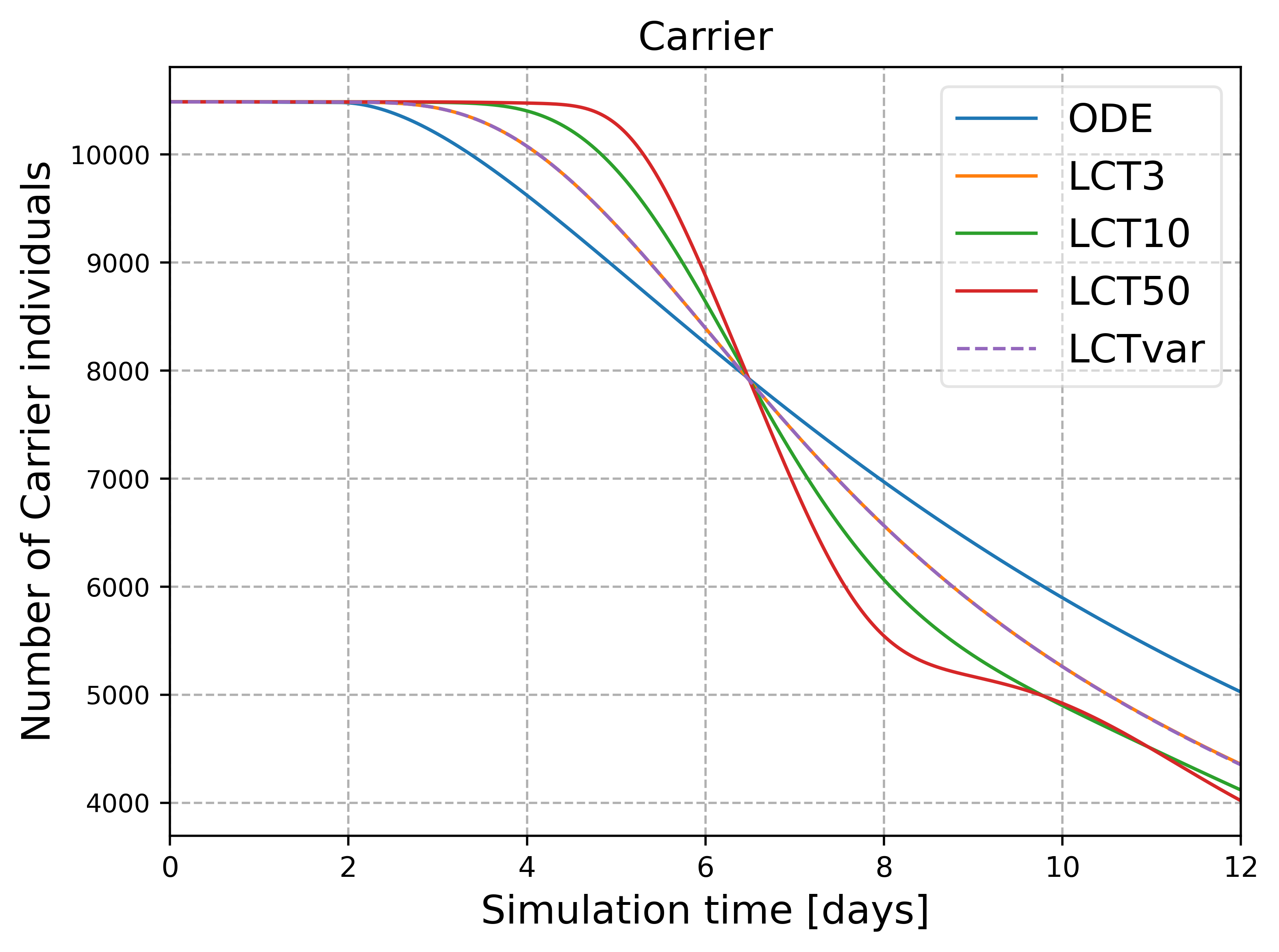}
\end{minipage}
\begin{minipage}[t]{0.45\textwidth}
 \centering
    \includegraphics[width=\textwidth]{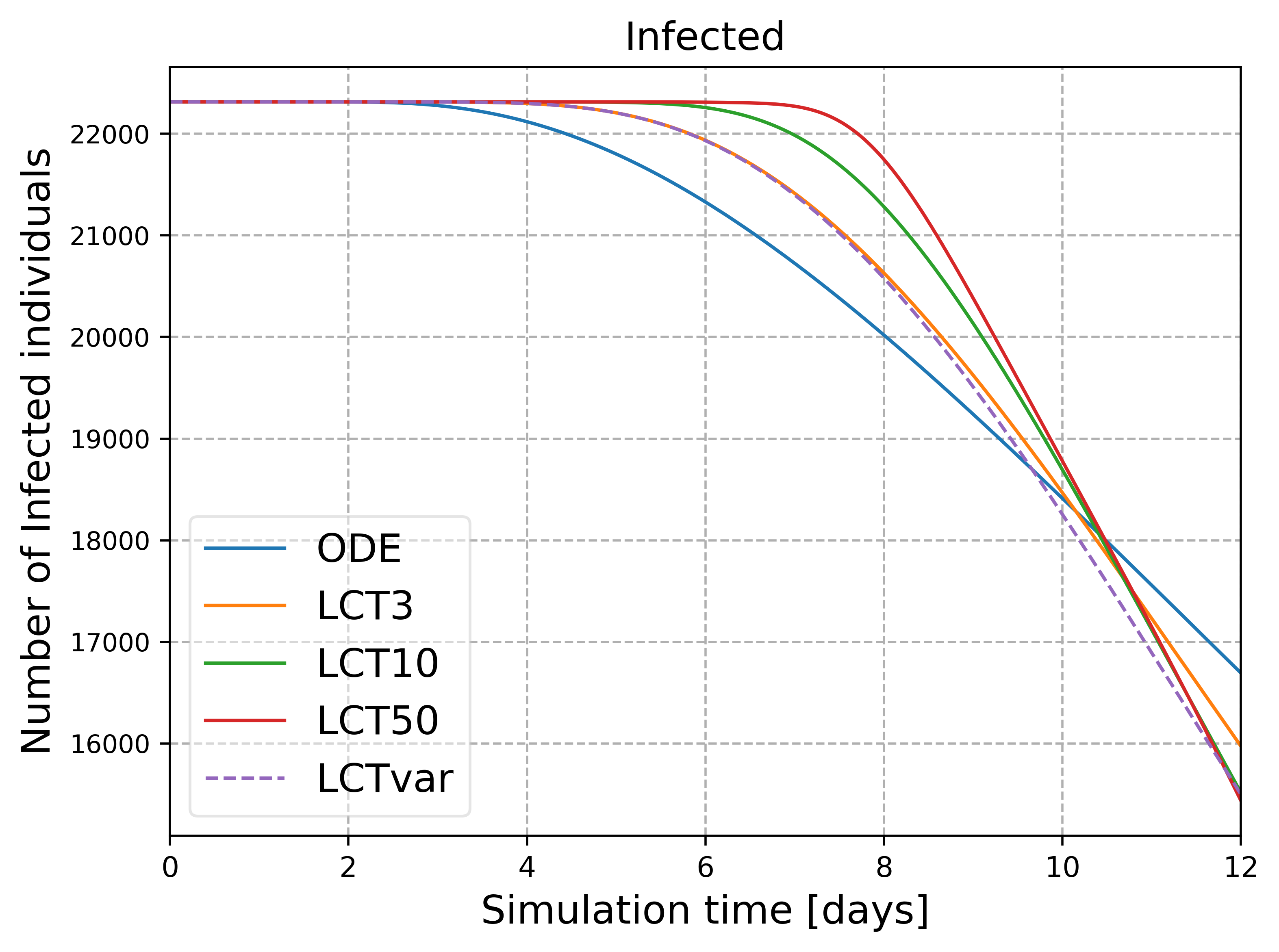}
\end{minipage}  
 \centering
    \begin{minipage}[t]{0.45\textwidth}
    \centering
     \includegraphics[width=\textwidth]{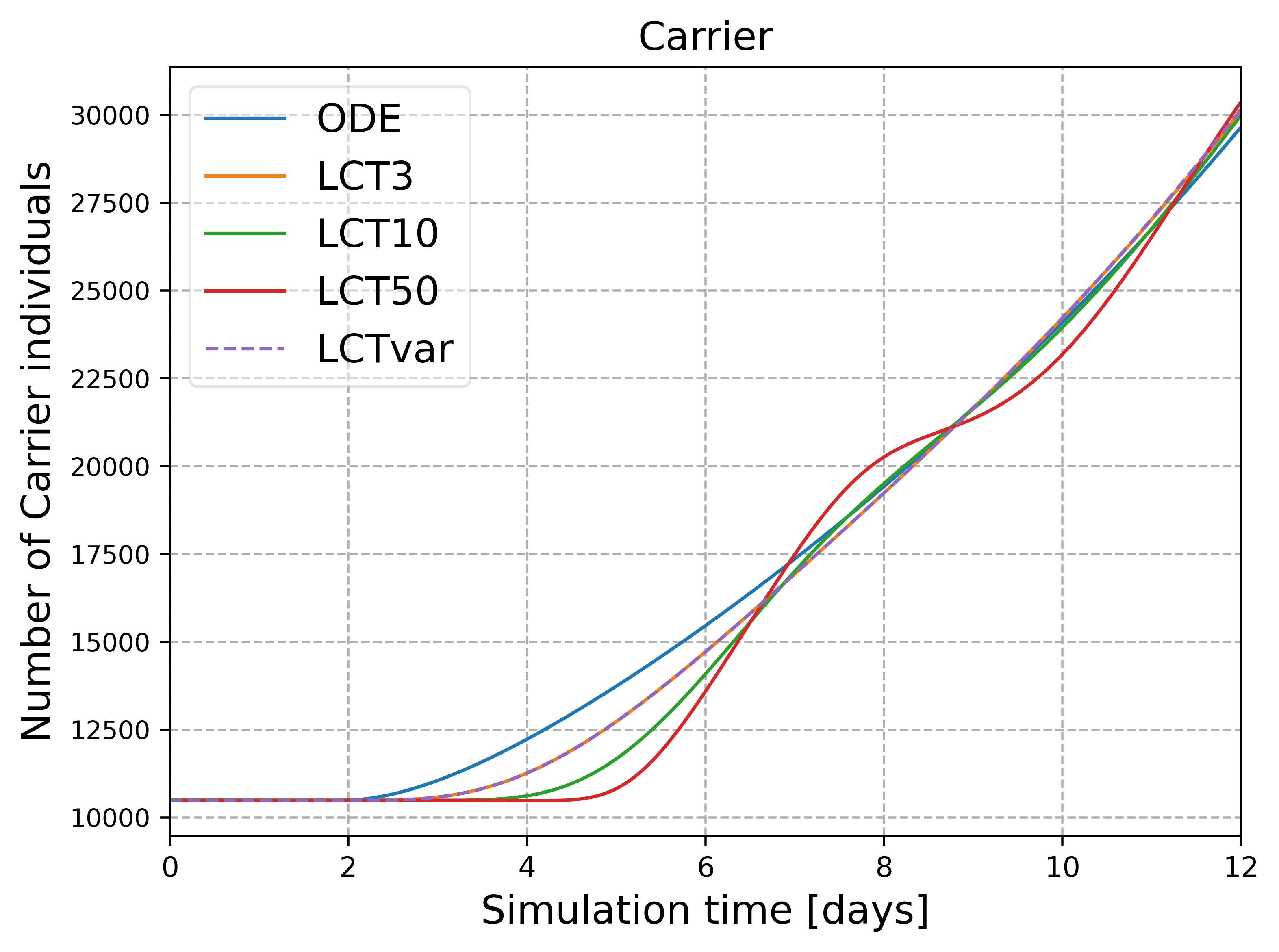}
\end{minipage}
\begin{minipage}[t]{0.45\textwidth}
 \centering
    \includegraphics[width=\textwidth]{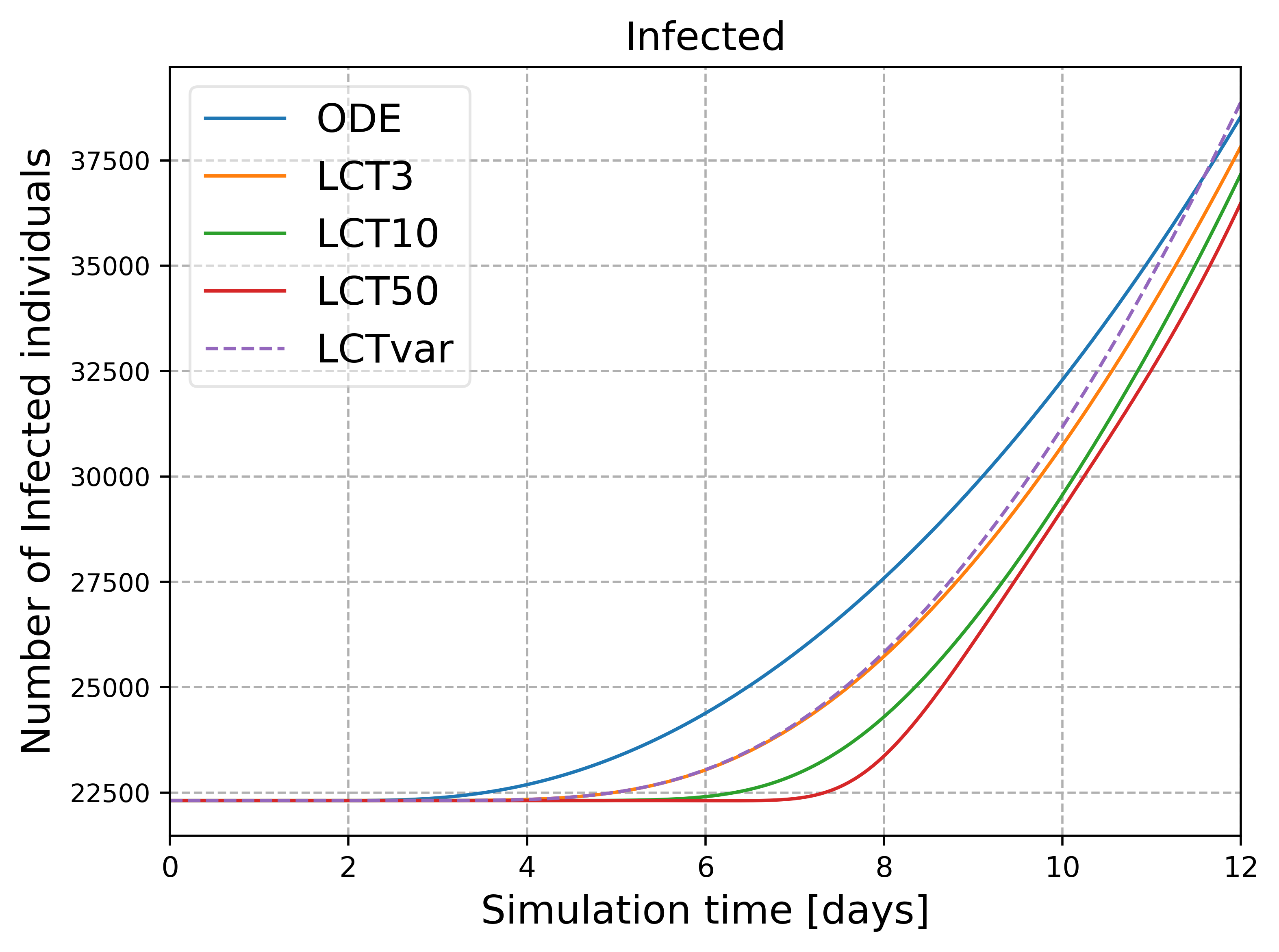}
\end{minipage}  
\caption{\textbf{Number of individuals in Carrier and Infected state \LP{around} different change points.} Comparison of the number of individuals in the Carrier (left) and Infected (right) compartment of different LCT models against an ODE model for the case of a halved (top) or doubled (bottom) contact rate $\phi(t)$ after two simulation days. \MK{Further notation as in~\cref{fig:changepoints}.}}
    \label{fig:compartments_changepoints}
\end{figure}
\begin{figure}[!bt]
\centering
    \begin{minipage}[t]{0.325\textwidth}\includegraphics[width=\textwidth]{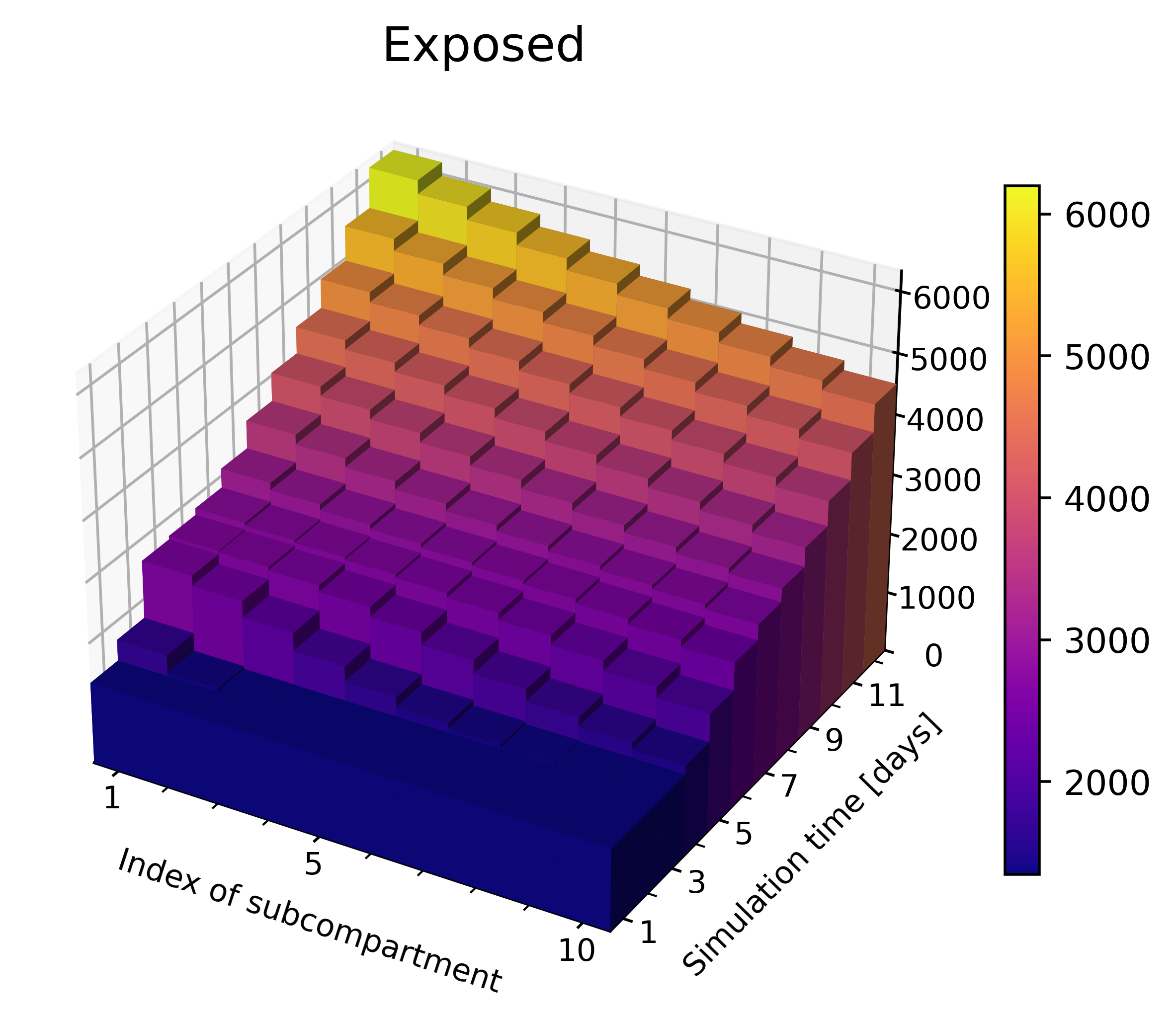}
\end{minipage}
\begin{minipage}[t]{0.325\textwidth}
 \centering
\includegraphics[width=\textwidth]{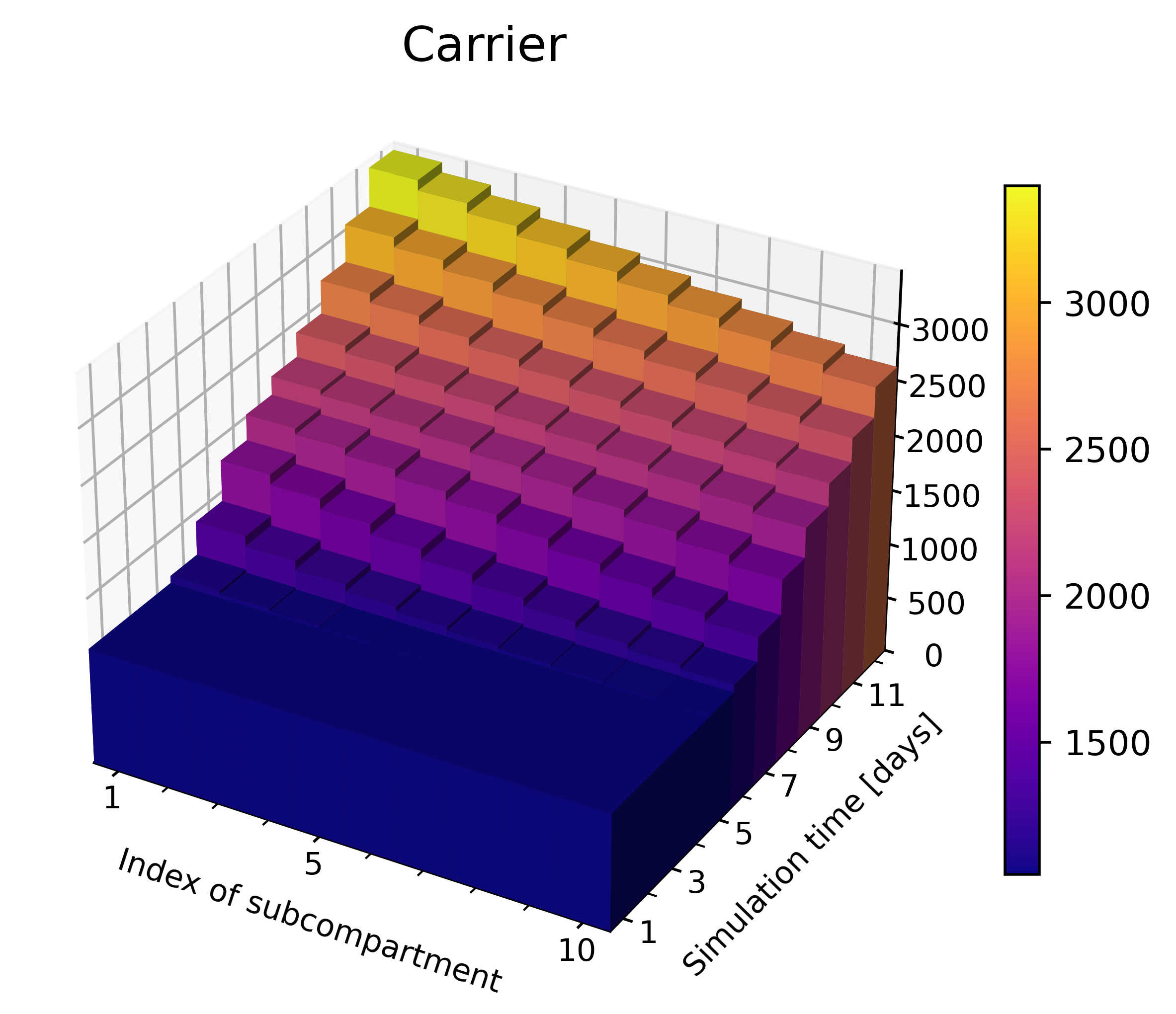}
\end{minipage}  
\begin{minipage}[t]{0.325\textwidth}
    \centering
\includegraphics[width=\textwidth]{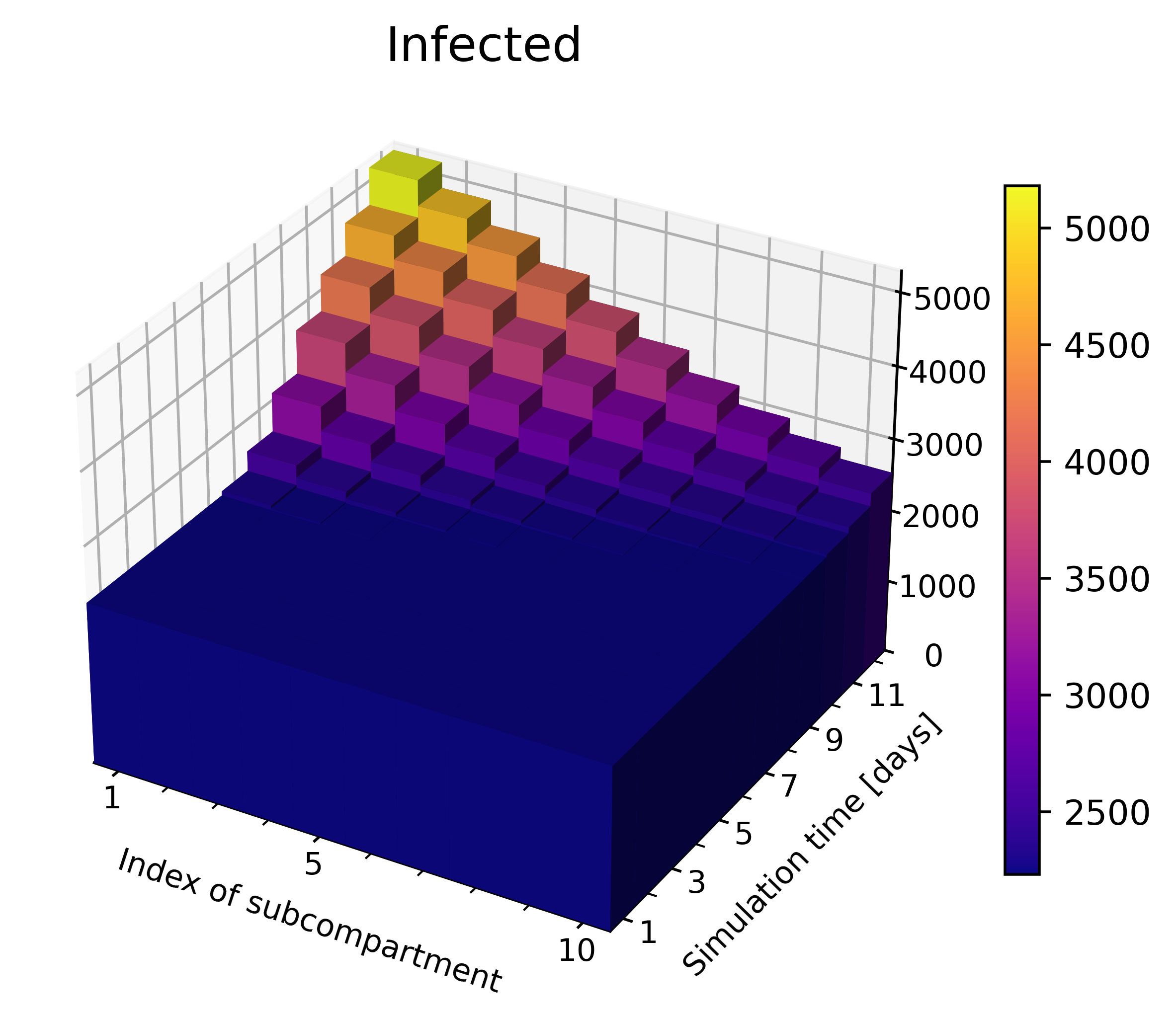}
\end{minipage}\\
\vspace{0.5cm}
    \centering
    \begin{minipage}[t]{0.325\textwidth}
    \centering     \includegraphics[width=\textwidth]{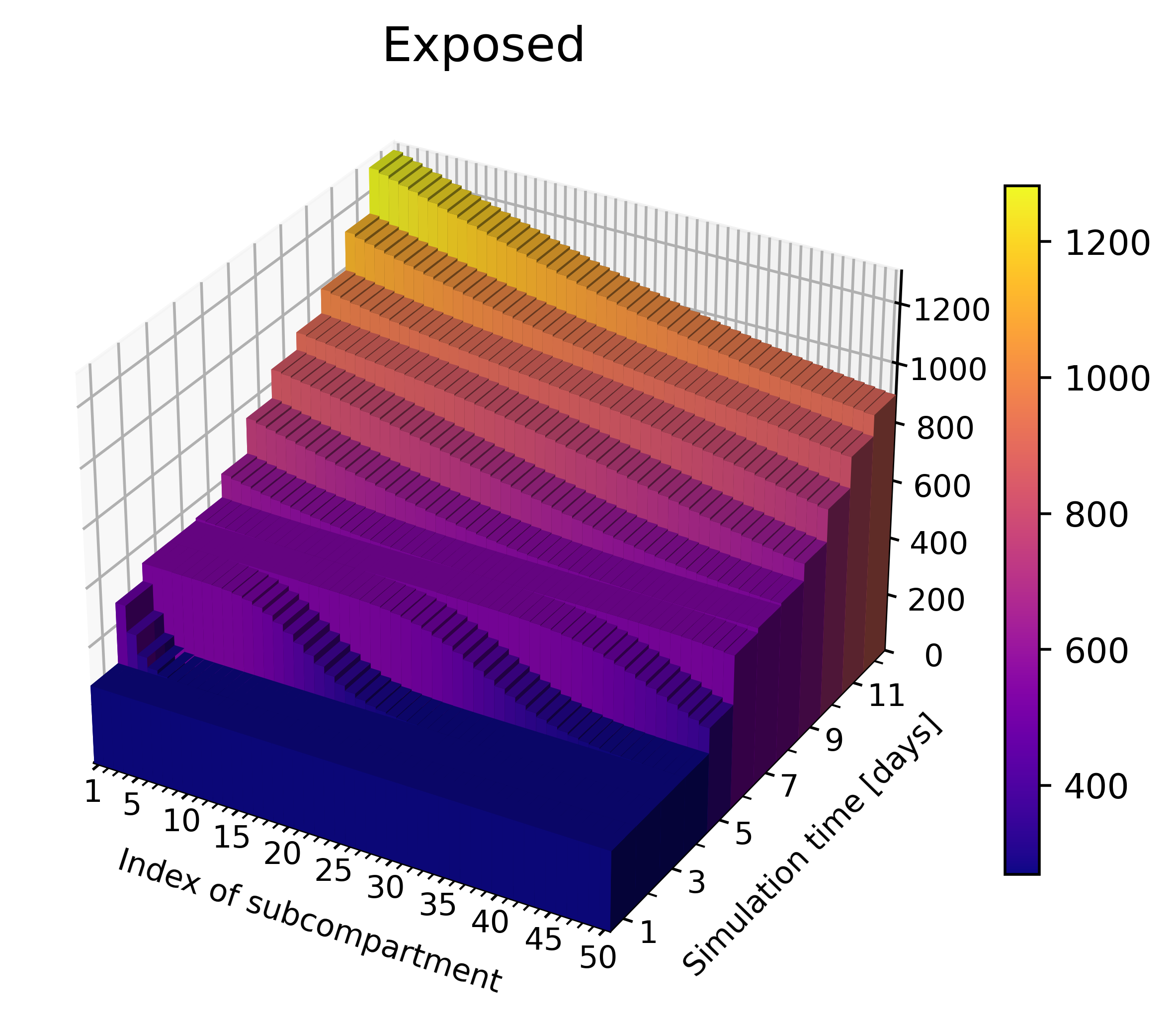}
\end{minipage}
\begin{minipage}[t]{0.325\textwidth}
 \centering
\includegraphics[width=\textwidth]{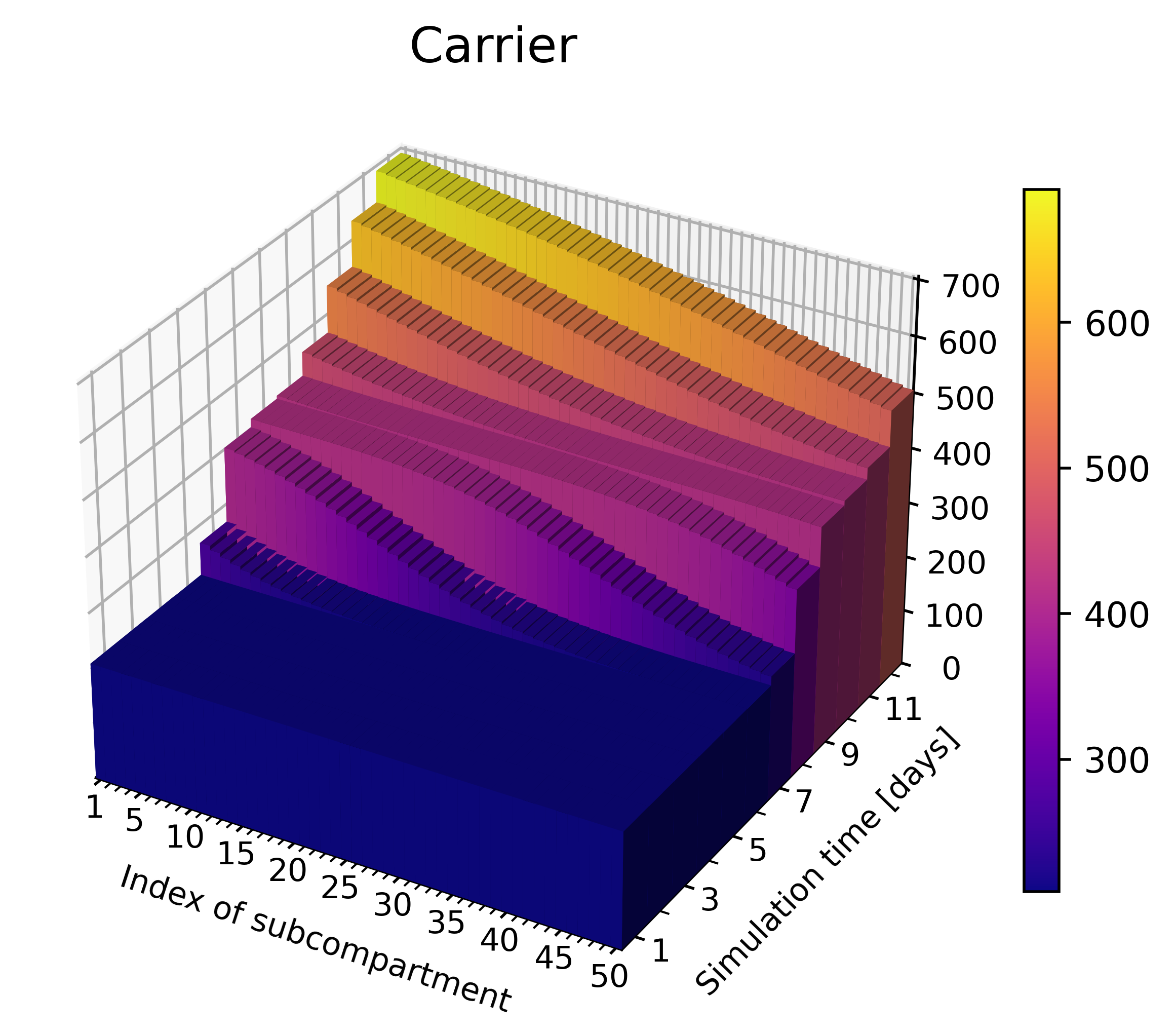}
\end{minipage}  
\begin{minipage}[t]{0.325\textwidth}
    \centering
 \includegraphics[width=\textwidth]{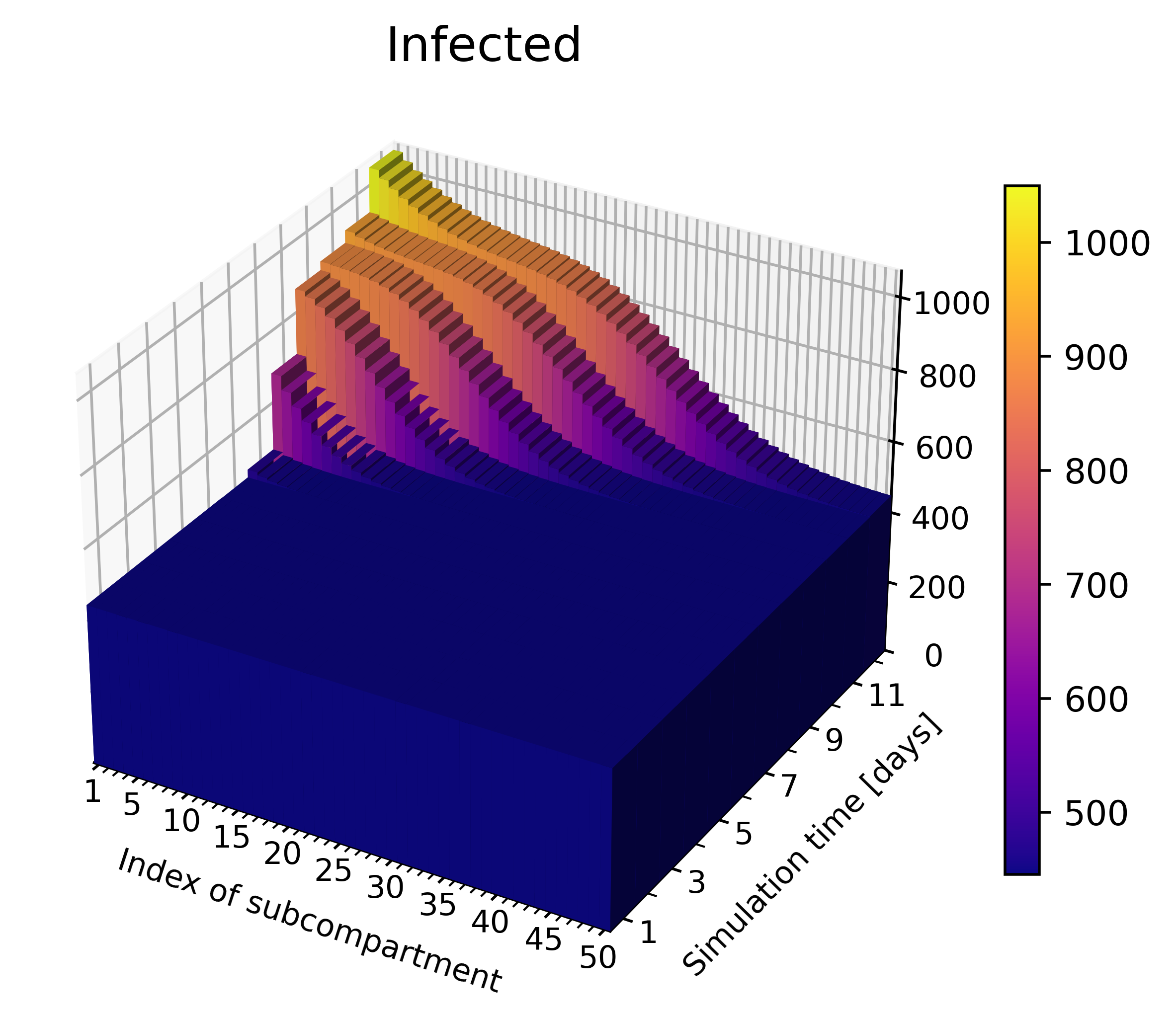}
\end{minipage}
\caption{\textbf{Distribution of individuals in the subcompartments for an increased contact rate.} The figures depict the number of individuals in the $n=10$ (top) or $n=50$ (bottom) subcompartments of the Exposed compartment (left), the Carrier compartment (center), and the Infected compartment (right) for each simulation day in the event of a doubled contact rate $\phi(t)$ after two days. The data for day $0$ are omitted, as no change can be observed in comparison to day $1$.}
    \label{fig:3drisesubcompartments}
\end{figure}

To gain insight into the behavior of simulation results obtained with the simple ODE model in comparison to LCT models, we analyze their reaction to change points.
A change point may be induced through the adoption or lifting of a NPI.

\LP{We set the contact rate and the initial values such that we obtain roughly constant infection dynamics at the start of the simulation; see the first two days in~\cref{fig:changepoints,fig:compartments_changepoints,fig:resultsshortTEhalved}. 
In particular, the contact rate is set to a level that results in a reproduction number of approximately one.
Using the next generation matrix, we compute the effective reproduction number for the ODE model,
\begin{align}\label{eq:reproduction}
    \mathcal{R}_{\text{eff}}(t) = \rho(t)\,\phi(t)\Big( \xi_{C}(t) \, T_C+\mu_C^{I}\,\xi_{I}(t)\,T_I\Big)\frac{S(t)}{N(t)},
\end{align}
to adapt the contact $\phi(0)$ such that $\mathcal{R}_{\text{eff}}(0)\approx 1$.
Under the assumption $S(0)\approx N(0)$ and with the parameters defined in~\cref{sec:parameters_data}, we get a contact rate of approximately $3.22$ (which is then also used for the LCT models). The initial compartment sizes are derived using the assumption of an approximately constant number of daily new transmissions and the parameters in~\cref{tab:COVID-19parameters}. Based on the official reporting numbers~\cite{RKI-Lagebericht-15-10-2020}, we use a value of $4050$ daily new transmissions as a starting point.
The resulting initial compartment sizes are distributed uniformly to the subcompartments.
}

To simulate a change point, we either double or halve the initial contact rate after two simulation days. The simulation results for the daily new transmissions using different models for both adaptions of the contact rates are visualized in~\cref{fig:changepoints}. The results for the number of individuals in the infectious compartments, Carrier and Infected, are shown in~\cref{fig:compartments_changepoints}.

Due to the changed contact rate, we first and correctly observe that the predicted number of daily new transmissions in~\cref{fig:changepoints} is halved or doubled, respectively, immediately at the change point. We furthermore see that the new transmissions for the case of the ODE model directly continue to increase or decrease. For LCT models with multiple subcompartments, a nontrivial lag time is observed before a subsequent change in the transmissions. This lag time is also evident in the compartment sizes depicted in~\cref{fig:compartments_changepoints}. It can be observed that the length of the lag time increases in accordance with the number of subcompartments used for the simulations. Furthermore, the slopes of the curves after the lag time vary according to the number of subcompartments. 

The discrepancy in the lag time is a result of the different variances of the stay time distribution in the respective disease states, 
see~\cref{fig:Erlang} and~\cref{eq:variance}. The exponential distributions used in the ODE model have the highest variance. A fraction of people leave immediately after entering a compartment, resulting in very short stays, cf.~\cref{fig:Erlang}. Therefore, the number of Carriers predicted by the ODE model increases immediately after the contact change,  see~\cref{fig:compartments_changepoints}.
The higher the number of subcompartments, the lower the variance~\eqref{eq:variance} and the fewer people go to the next state much before or much after the mean stay time. Consequently, the greater the choice for $n$, the longer the delay.

According to Dey et al.~\cite{dey_lag_2021}, a relaxation or implementation of a NPI leads to a change after $10$ to $14$ days in data on COVID-19 in the United States. Guglielmi et al.~\cite{guglielmi_identification_2023} also find a significant delay for data from Italy and Switzerland. Knowledge of lag times demonstrates the need for policymakers to proactively plan for NPIs. Therefore, it is essential that the delay is represented in simulation results without further adjustments.

\begin{figure}[!bt]
    \centering
    \begin{minipage}[t]{0.47\textwidth}
    \centering
     \includegraphics[width=\textwidth]{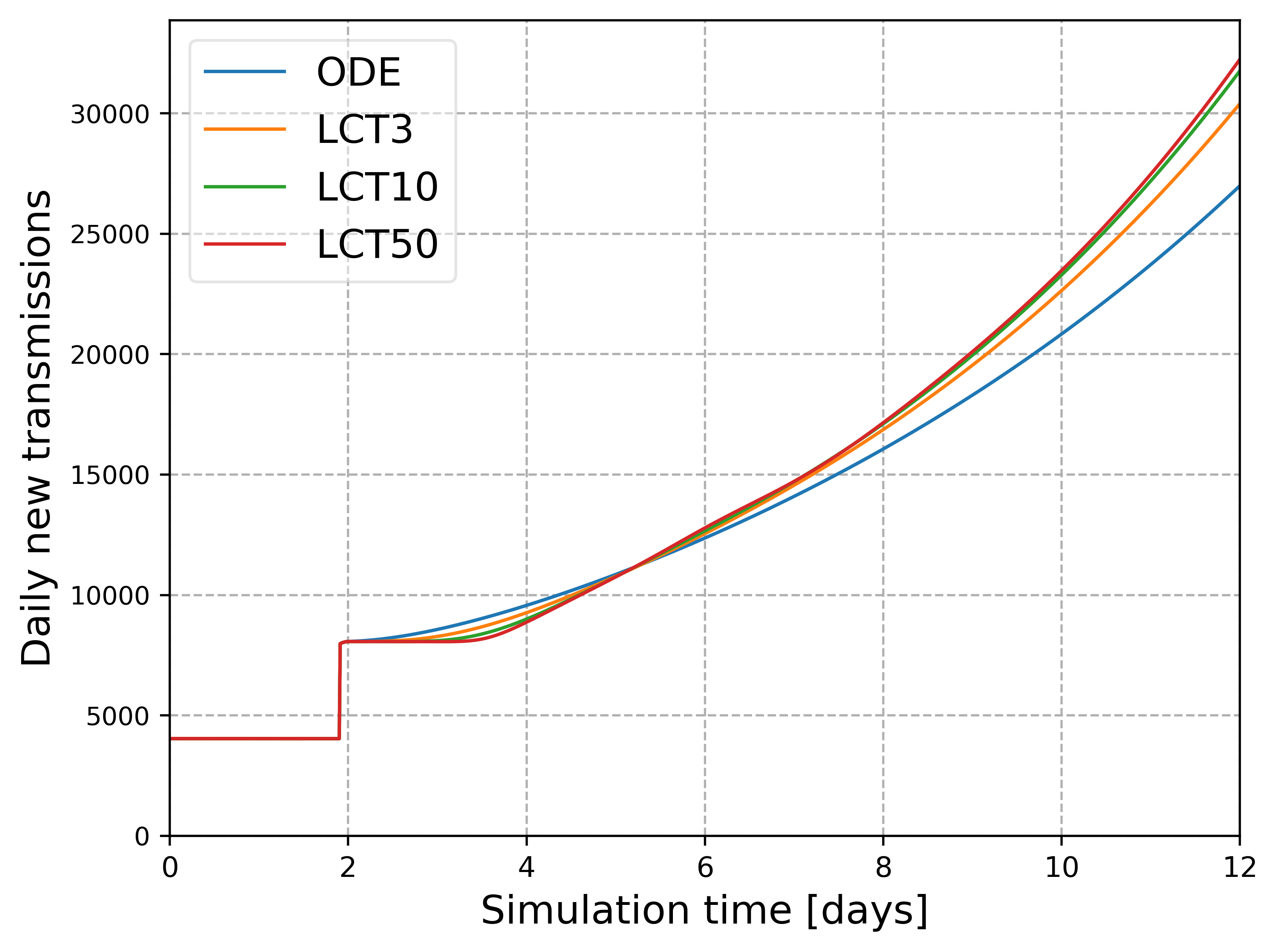}
\end{minipage}
\begin{minipage}[t]{0.47\textwidth}
 \centering
    \includegraphics[width=\textwidth]{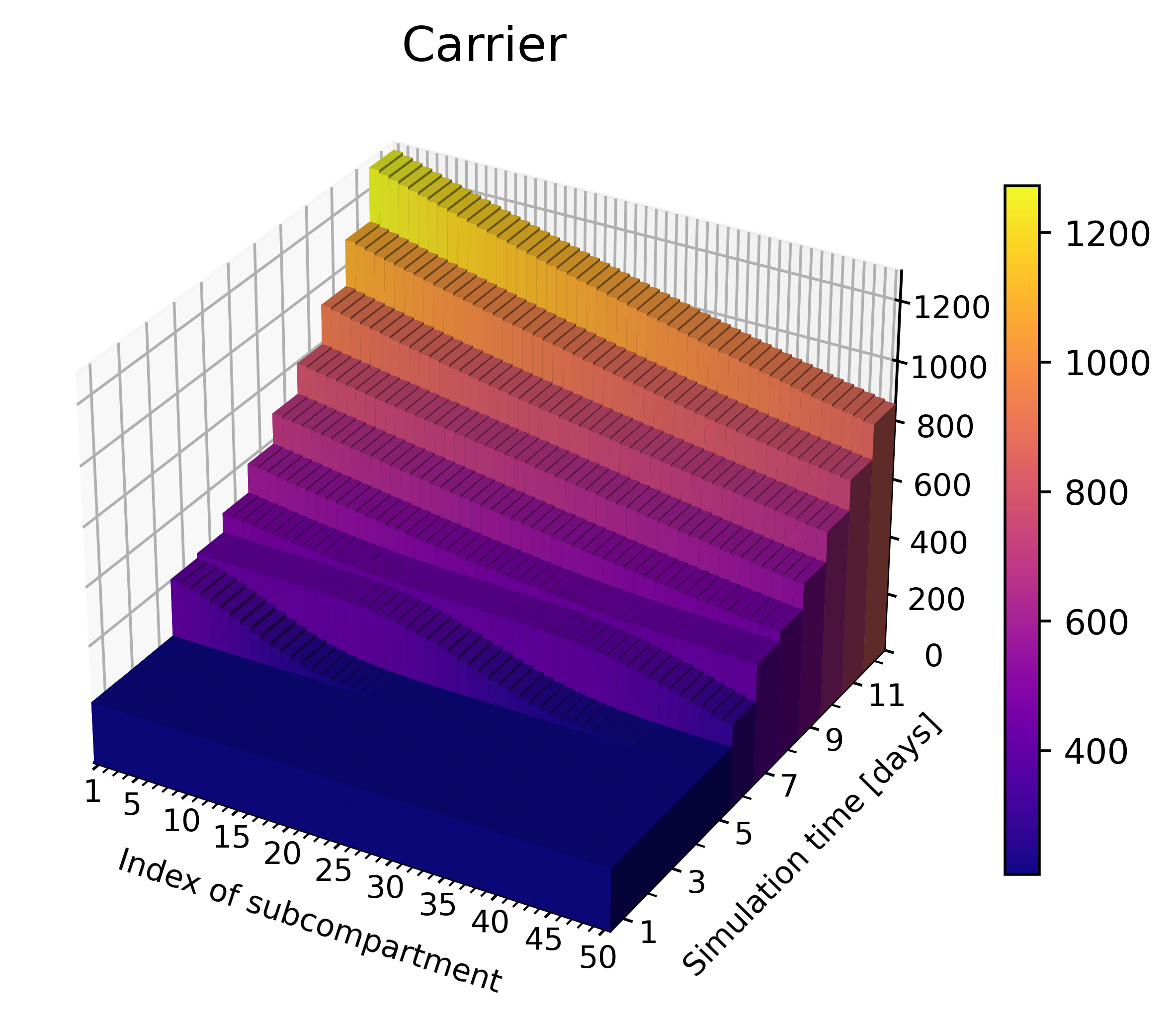}
\end{minipage}  
\caption{\textbf{Simulation results for a reduced latent period.} Daily new transmission (left) and the distribution of individuals in the Carrier compartment in $n=50$ subcompartments for the case of a doubled contact rate $\phi(t)$ after two simulation days and for a \textbf{reduced stay time in the Exposed compartment} $\mathbf{T_E}$, i.e., a reduced latent period. More precisely, we multiply the value defined in~\cref{tab:COVID-19parameters} \textbf{with a factor of} $\mathbf{0.5}$. \MK{Further notation as in~\cref{fig:changepoints}.}}
    \label{fig:resultsshortTEhalved}
\end{figure}
\begin{figure}[htb]
    \centering
    \begin{minipage}[t]{0.47\textwidth}
    \centering
     \includegraphics[width=\textwidth]{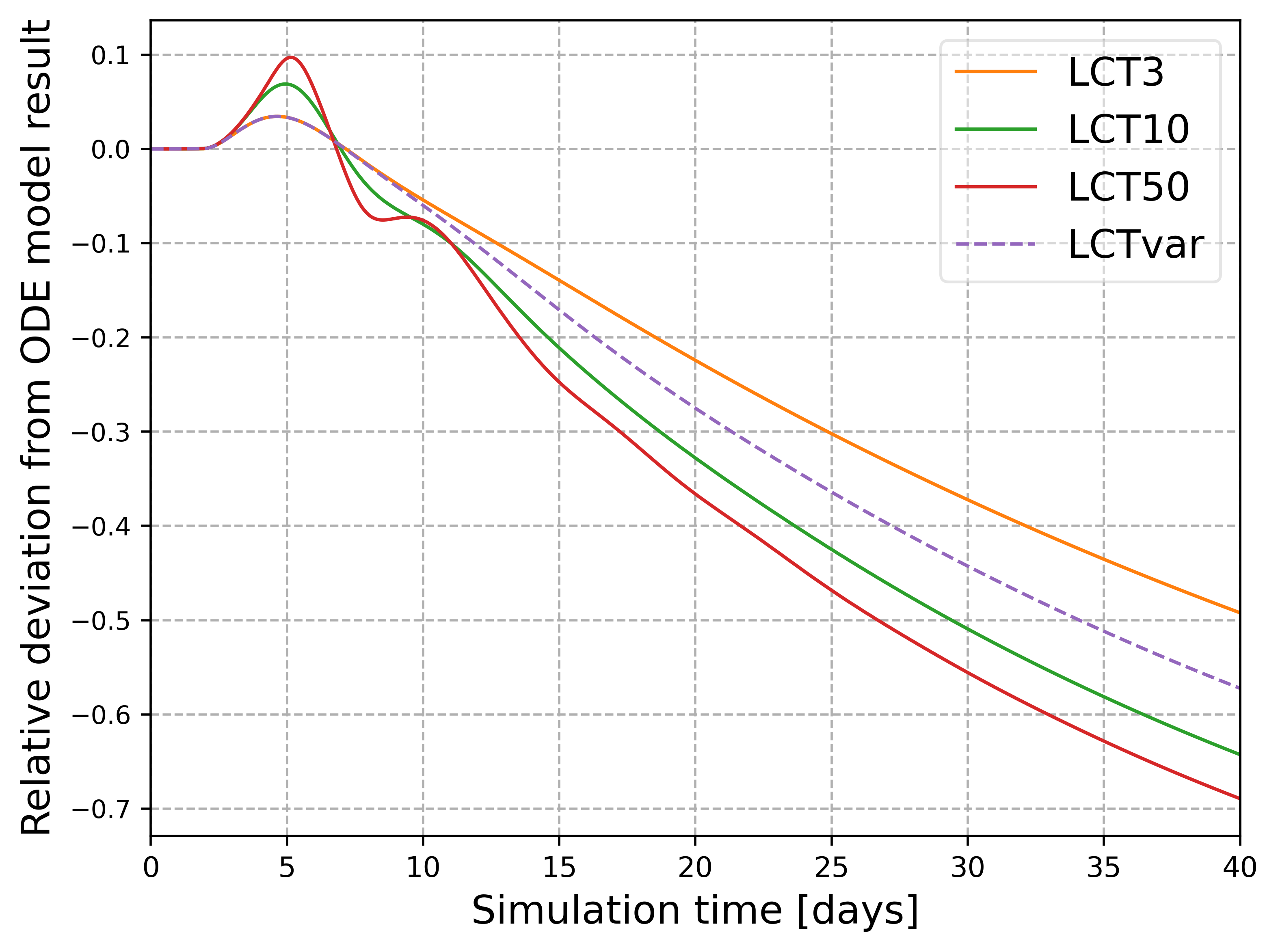}
\end{minipage}
\begin{minipage}[t]{0.47\textwidth}
 \centering
    \includegraphics[width=\textwidth]{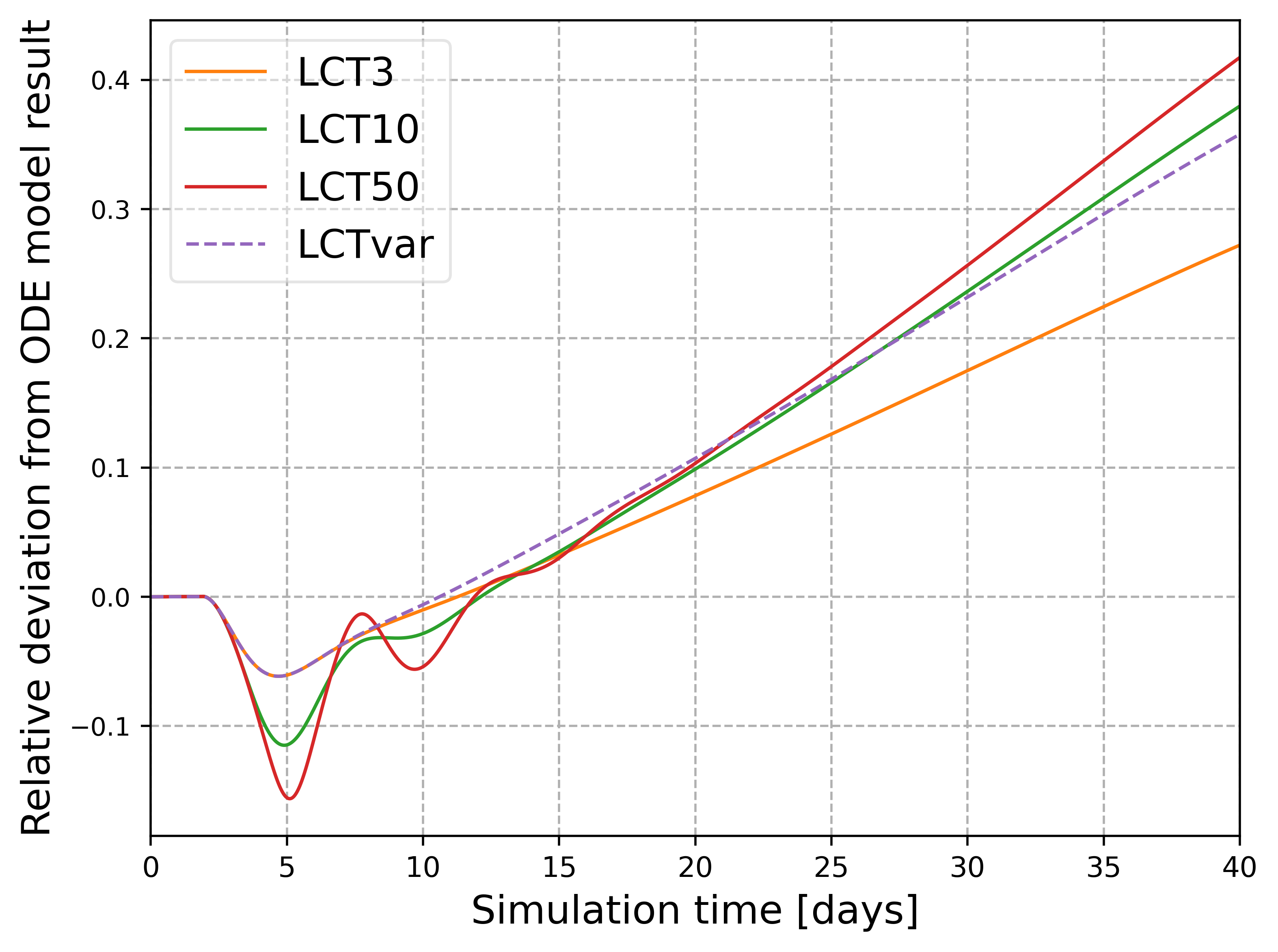}
\end{minipage}  
\caption{\MK{\textbf{Relative difference in daily new transmissions around change points.} Relative comparison of the daily new transmissions of different LCT models compared to a simple ODE model around change points. The contact rate $\phi(t)$ is halved (left) or doubled (right) after the second simulation day. Further notation as in~\cref{fig:changepoints}.}}
    \label{fig:changepoints_rel}
\end{figure}

Particularly in the case of a doubling of the contact rate, a wave pattern is noticeable in~\cref{fig:changepoints} and for the Carrier compartment in~\cref{fig:compartments_changepoints} for high numbers of subcompartments, such as $n=50$. 
This phenomenon is once more the result of a small variance in combination with the parameter choice for the stay times. To gain a deeper understanding of the wave pattern,~\cref{fig:3drisesubcompartments} illustrates the distributions within the subcompartments for LCT models with either $n=10$ or $n=50$ subcompartments and a doubled contact rate.

As can be observed in~\cref{fig:compartments_changepoints}, for $n=50$ the number of Carrier individuals begins to increase after a period of slightly less than $2+T_E=5.335$ days. At this point, the first individuals move from $E$ to $C$ that have been infected after the second simulation day, where the contact rate is increased. This is due to the fact that they \LP{reach} the last subcompartment, as illustrated in~\cref{fig:3drisesubcompartments}. This first rise is induced by the increased contact rate. In the model with $10$ subcompartments, the higher variance of the stay time distribution in $E$ leads to individuals that traverse faster through the chain of subcompartments. Therefore, the first individuals reach compartment $C$ earlier than in the case of $50$ subcompartments.

Once again, with $50$ subcompartments, after approximately $T_C\LP{\approx2.59}$ more days, a decreasing slope in compartment $C$ becomes apparent. 
As observable in~\cref{fig:3drisesubcompartments}, at this time, the first individuals infected after simulation day two leave compartment $C$ and move to either $I$ or $R$. The subsequent larger slope can be attributed to the increased number of transmissions resulting from the higher number of Carriers observed after approximately $2+T_E$ days. The individuals that have been infected during this period begin to transition to the Carrier compartment. Thus, this second rise of the slope is due to the increased number of infectious individuals. The transition from $C$ to $I$ is evident in the plot in~\cref{fig:3drisesubcompartments} and for the compartment Infected, we also get a wave pattern. A longer simulation period reveals that the wave pattern becomes rapidly unrecognizable as the ratio of inflow and outflow in the infectious compartments becomes balanced. For $10$ subcompartments, \LP{ where the variance is higher,} the time when the increased number of individuals in $C$ is driven by the higher contact rate overlaps with the first newly infected individuals leaving $C$ and the time when more infectious individuals drive the increase. The increased numbers are more evenly distributed across the subcompartments. Accordingly, we do not observe a significant wave pattern here.

The wave pattern and the explanation for the waves are highly specific to the particular parameter selection and the relationship between the average stay times. \cref{fig:resultsshortTEhalved} depicts the simulation results for a reduced stay time in the Exposed compartment $T_E$, i.e., a halved latent period compared to~\cref{tab:COVID-19parameters}, for the case where the contact rate was doubled after two simulation days. For this adapted parameter choice, the stay time in the Exposed compartment is shorter than in the Carrier compartment, i.e., $T_E<T_C$. As shown by the daily new transmissions, the shape of the curves differ for each model compared to the original parameter selection, and the wave pattern is less pronounced.

\MK{In~\cref{fig:changepoints,fig:compartments_changepoints}, we have presented the short-term outcomes up to 10 days after the change point. However, due to the discussed wave pattern and the different slopes, it is yet unclear to which extent the different model assumptions influence the daily new transmissions on a short- to mid-term horizon. 
Thus, we provide the relative difference between the daily new transmissions in the LCT models, compared to the ODE model, in~\cref{fig:changepoints_rel} for the same scenario but over a period of up to 40 simulation days. Here, we see that the number of new transmissions differs up to -70~\% (halved contact rate) and~40~\% (doubled contact rate), respectively.}

 \subsubsection{Epidemic peaks and final size}\label{sec:peaks}
\begin{figure}[tb]
    \centering
    \begin{minipage}[t]{0.47\textwidth}
    \centering
     \includegraphics[width=\textwidth]{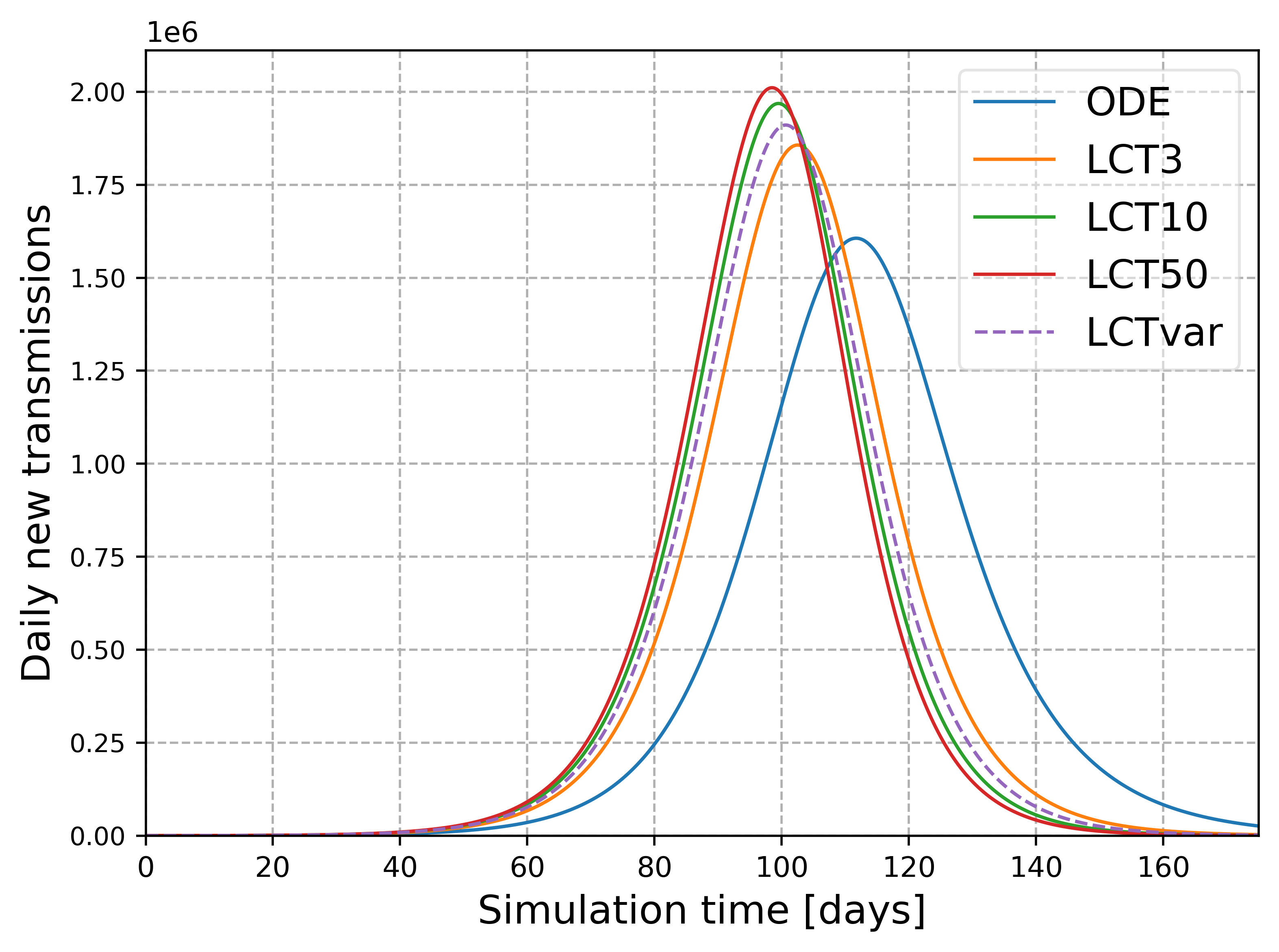}
\end{minipage}
\begin{minipage}[t]{0.47\textwidth}
 \centering
    \includegraphics[width=\textwidth]{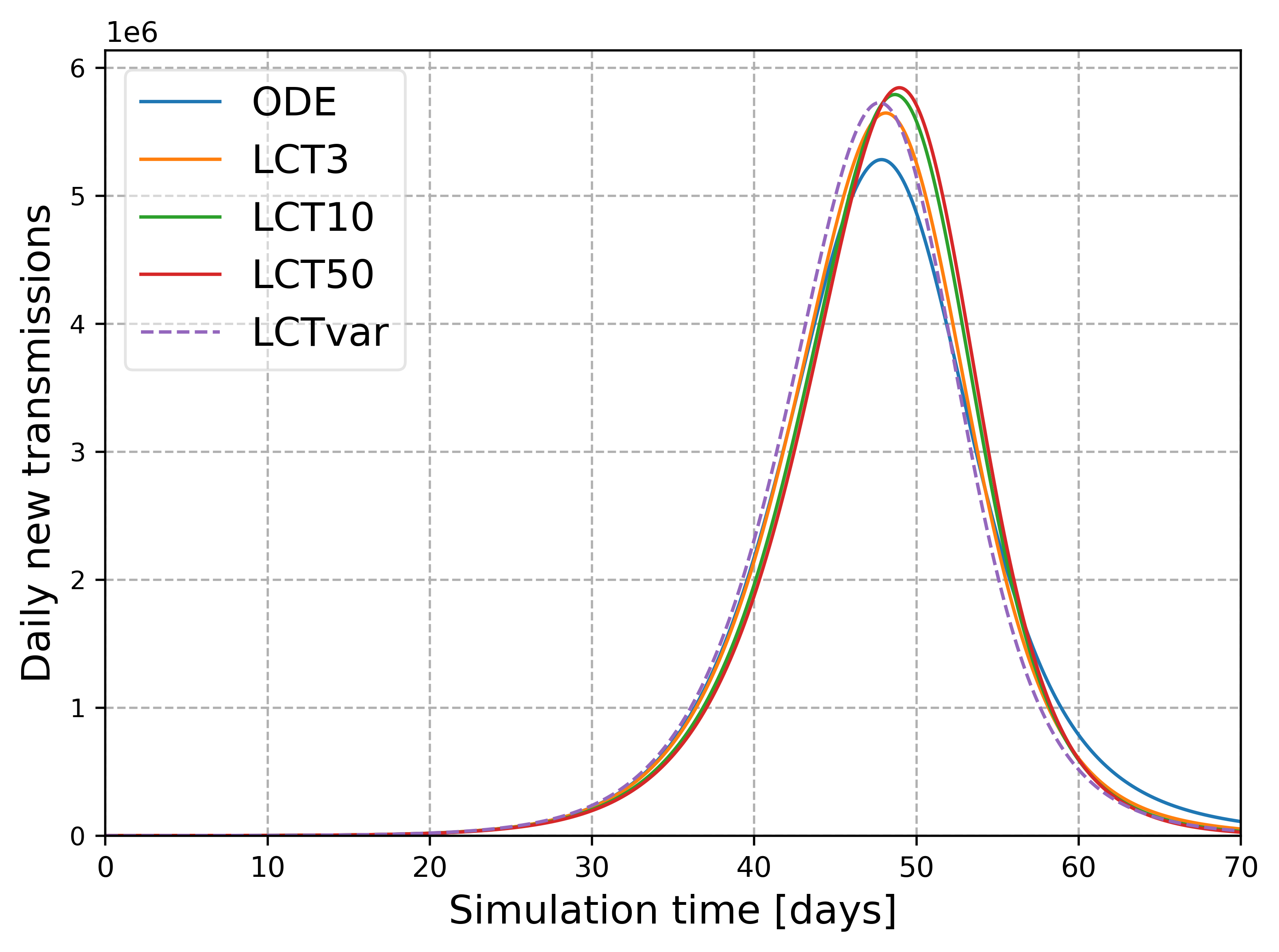}
\end{minipage}  
\caption{\textbf{Daily new transmissions to compare the predicted epidemic peaks.}
Illustration of the predicted peaks of the daily new transmissions of LCT models with different numbers of subcompartments against a simple ODE model. The contact rate is \LP{set such that $\mathcal{R}_{\text{eff}}(\LP{0})\approx2$ }(left) or \LP{$\mathcal{R}_{\text{eff}}(\LP{0})\approx4$ (right)}.
 \MK{Further notation as in~\cref{fig:changepoints}.}}
    \label{fig:peak_new_transmissions}
\end{figure}
\begin{figure}[!bt]
    \centering
    \begin{minipage}[t]{0.47\textwidth}
    \centering
     \includegraphics[width=\textwidth]{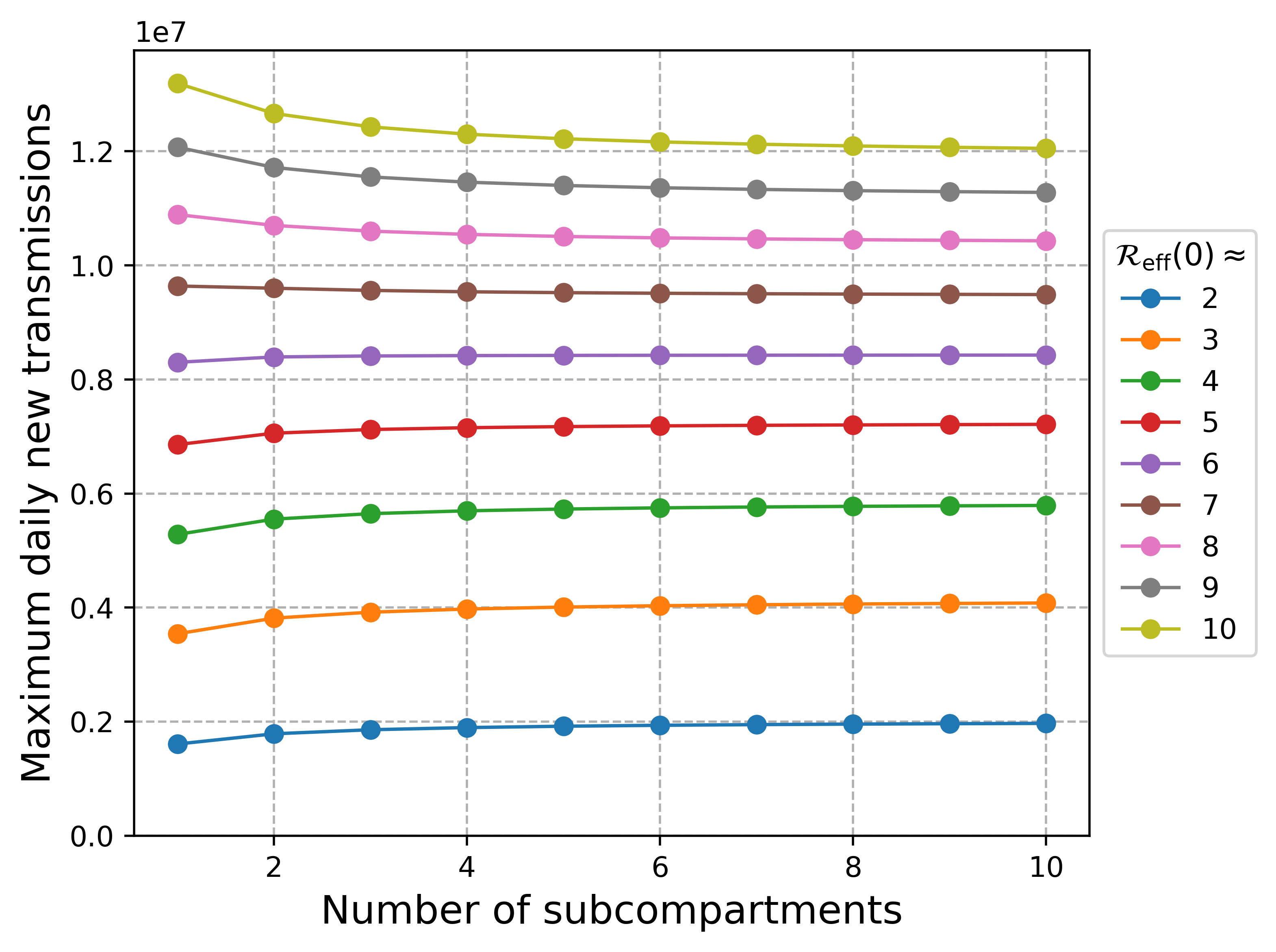}
\end{minipage}
\begin{minipage}[t]{0.47\textwidth}
 \centering
    \includegraphics[width=\textwidth]{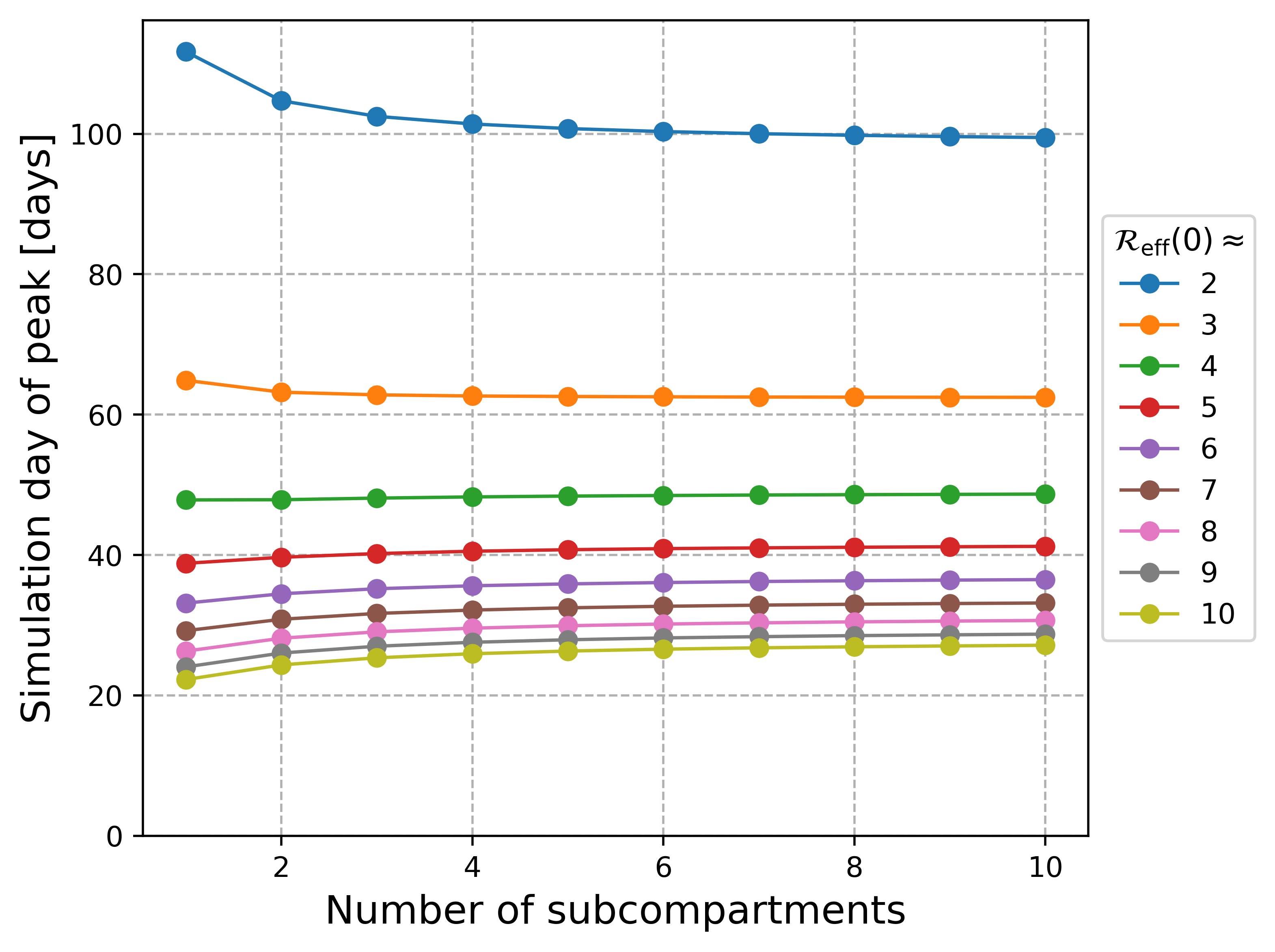}
\end{minipage}  
\caption{\textbf{Comparison of the maximum daily new transmissions and the day of the peak for different reproduction numbers and subcompartments.} Comparison of the predicted maximum number of daily new transmissions and the day, where the maximum number/peak is reached for different assumptions regarding the effective reproduction number \LP{$\mathcal{R}_{\text{eff}}(0)$} and with different choices for the number of subcompartments.}
\label{fig:compare_peaksize_timing}
\end{figure}

One objective of NPIs is to reduce the maximum number of infections in order to prevent the healthcare system from being overloaded and to keep disease dynamics on a manageable level. Based on the model selection, different simulations can lead to different predicted peaks and, thus, different assessments of the same NPIs. We therefore examine the impact of the distribution assumption on the predicted epidemic peaks to assess the level of error when choosing a too simple model. 

\LP{The models are initialized with $500$ exposed individuals and the remaining population in the Susceptible compartment. The exposed individuals are distributed uniformly to the subcompartments. We perform simulations with different values for the contacts $\phi(0)$ and consequently for the effective reproduction number $\mathcal{R}_{\text{eff}}(0)$. The contact rate remains constant throughout the duration of the simulation.}
\begin{figure}[!bt]
    \centering
    \begin{minipage}[t]{0.47\textwidth}
    \centering
     \includegraphics[width=\textwidth]{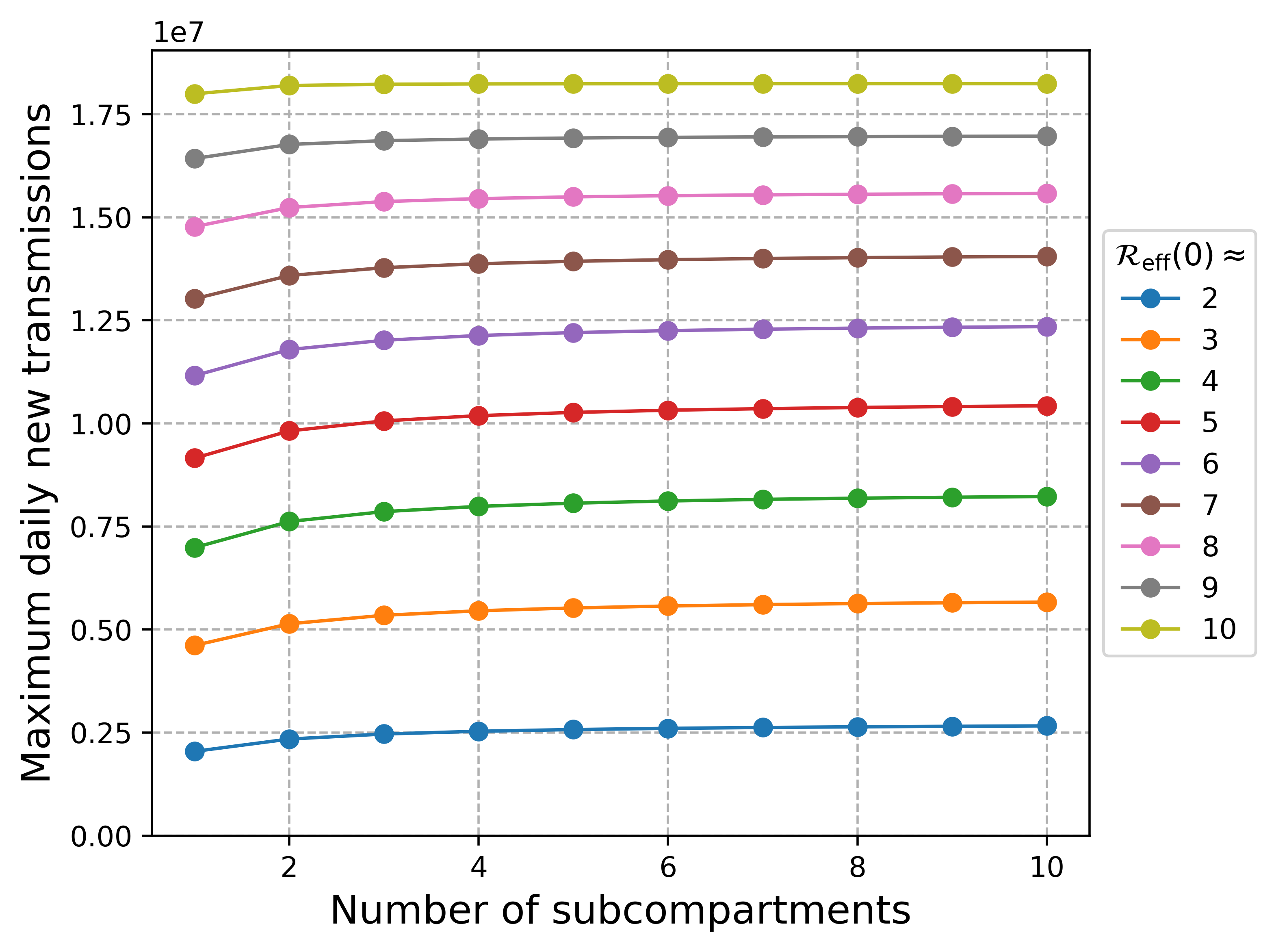}
\end{minipage}
\begin{minipage}[t]{0.47\textwidth}
 \centering
    \includegraphics[width=\textwidth]{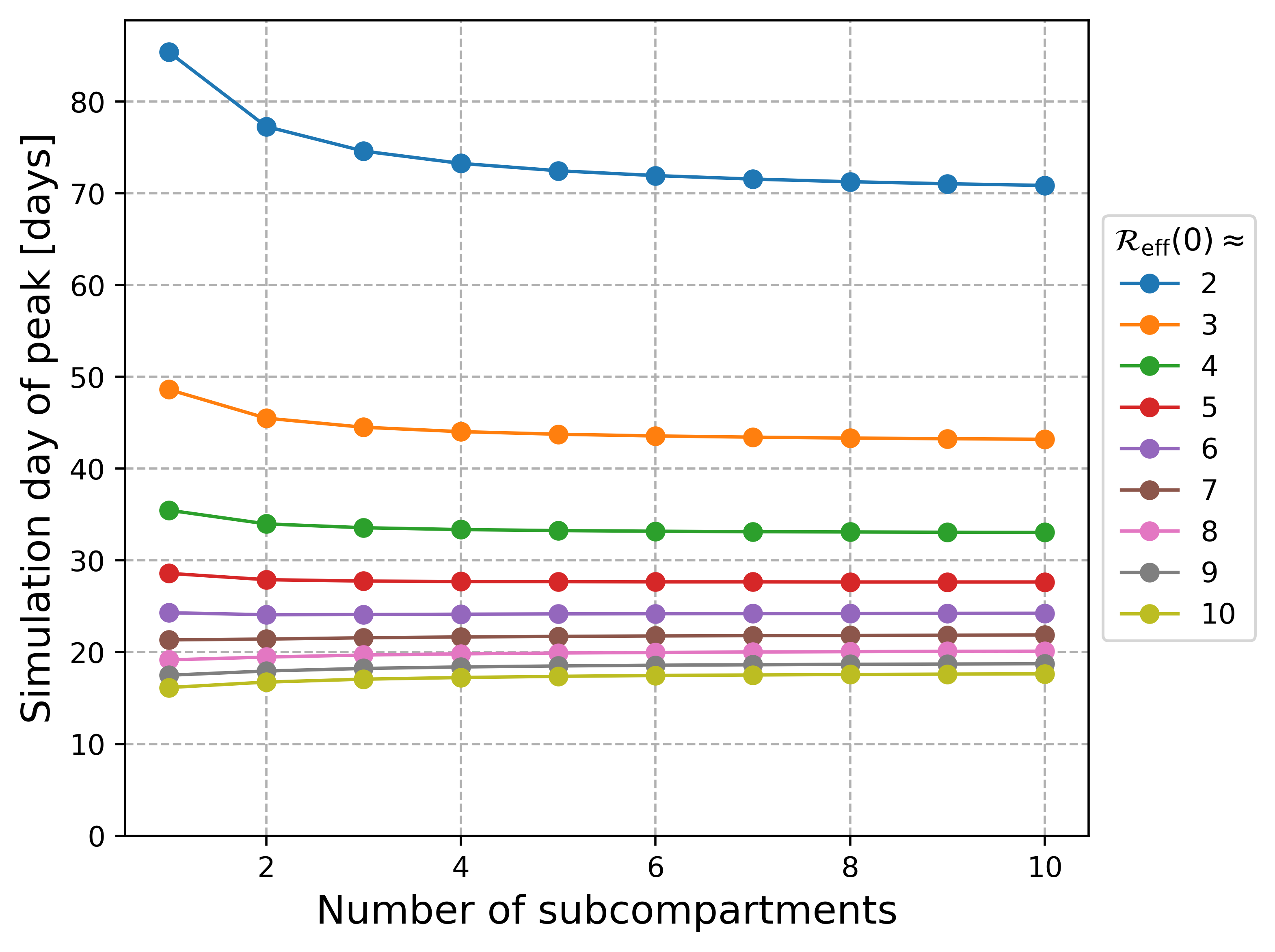}
\end{minipage}  
\caption{\textbf{Comparison of the maximum daily new transmissions and the day of the peak for different reproduction numbers and subcompartments for a reduced latent period.} Comparison of the predicted maximum number of daily new transmissions and the day, where the maximum number/peak is reached for different assumptions regarding the effective reproduction number \LP{$\mathcal{R}_{\text{eff}}(0)$} and with different choices for the number of subcompartments. Here we consider a \textbf{reduced stay time in the Exposed compartment} $\mathbf{T_E}$, i.e., a reduced latent period. More precisely, we multiply the value used in~\cref{fig:compare_peaksize_timing} \textbf{with a factor of} $\mathbf{0.5}$.}
\label{fig:compare_peaksize_timing_halved}
\end{figure}
\begin{figure}[!bt]
    \centering
    \begin{minipage}[t]{0.47\textwidth}
    \centering
     \includegraphics[width=\textwidth]{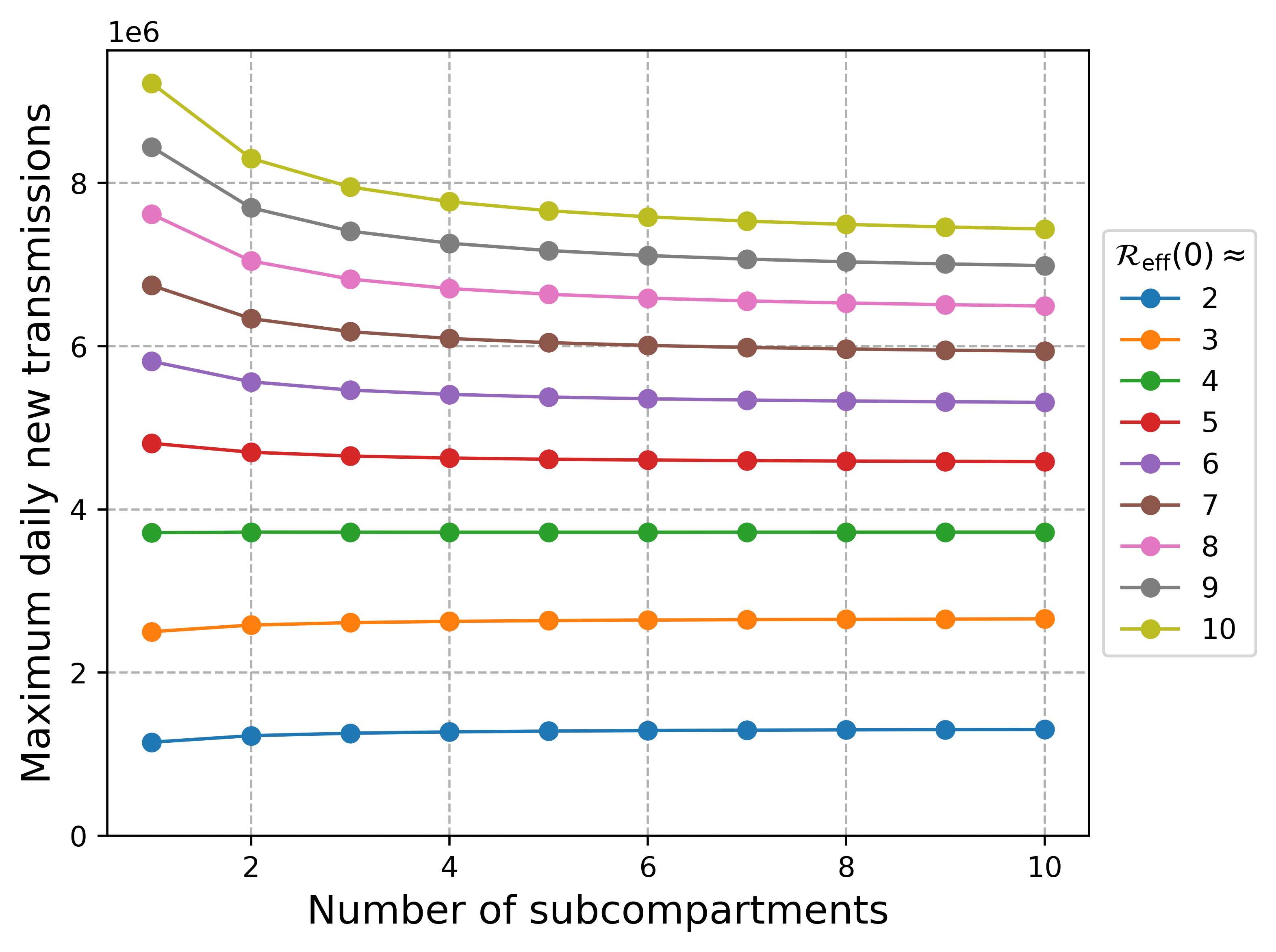}
\end{minipage}
\begin{minipage}[t]{0.47\textwidth}
 \centering
    \includegraphics[width=\textwidth]{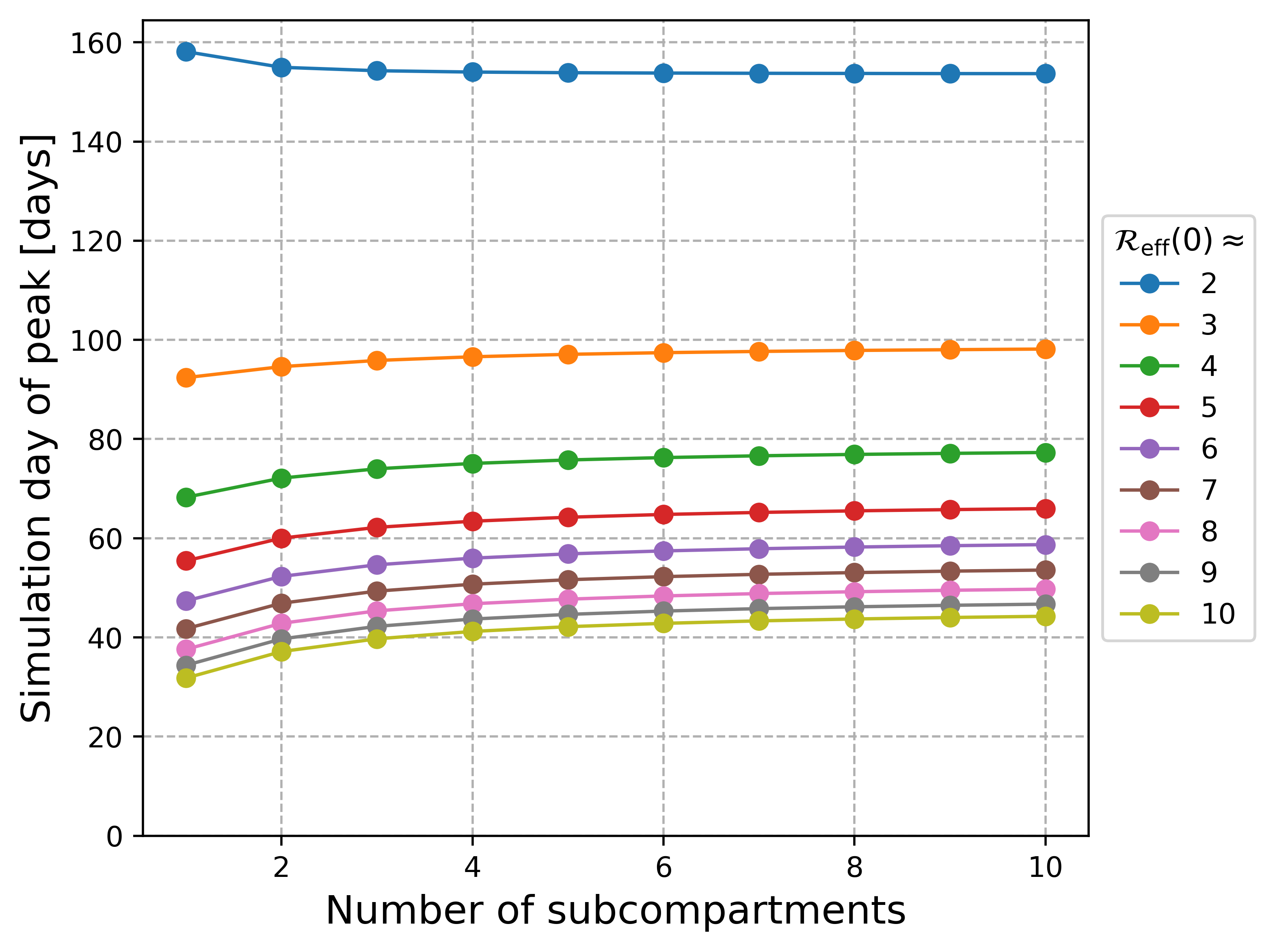}
\end{minipage}  
\caption{\textbf{Comparison of the maximum daily new transmissions and the day of the peak for different reproduction numbers and subcompartments for an extended latent period.} Comparison of the predicted maximum number of daily new transmissions and the day, where the maximum number/peak is reached for different assumptions regarding the effective reproduction number \LP{$\mathcal{R}_{\text{eff}}(0)$} and with different choices for the number of subcompartments. Here we consider an \textbf{extended stay time in the Exposed compartment} $\mathbf{T_E}$. More precisely, we multiply the value used in~\cref{fig:compare_peaksize_timing} \textbf{with a factor of} $\mathbf{2}$.}
\label{fig:compare_peaksize_timing_doubled}
\end{figure}

The results regarding the daily new transmissions for $\mathcal{R}_{\text{eff}}(\LP{0})\approx 2$ and $\mathcal{R}_{\text{eff}}(\LP{0})\approx 4$ are depicted in~\cref{fig:peak_new_transmissions}. For both reproduction numbers, a comparison of the LCT models shows that the maximum value of the curve increases with the number of subcompartments selected. The predicted peaks by the ODE model are notably lower than those predicted by the LCT model with three subcompartments. The relation of the times at which the models predict the maximum of the epidemic differs for the two effective reproduction numbers. In the case of $\mathcal{R}_{\text{eff}}(\LP{0})\approx 2$, we observe that the epidemic peak is reached earlier the higher the number of subcompartments is chosen. Conversely, for $\mathcal{R}_{\text{eff}}(\LP{0})\approx 4$, the peak is reached the later, the higher the number of subcompartments.

To compare the behavior for even more choices of $\mathcal{R}_{\text{eff}}\LP{(0)}$,~\cref{fig:compare_peaksize_timing} presents the predicted maximum value of the daily new transmissions and the time, where the epidemic peak is reached, for different reproduction numbers and numbers of subcompartments. Firstly, looking at one model with a fixed number of subcompartments, we observe that the higher the effective reproduction number is set, the greater the maximum attained value and the earlier that value \AW{is} reached. Comparing different numbers of subcompartments, we find that for reproduction numbers greater or equal to $7$, there is a tendency to lower maxima for increasing numbers of subcompartments. For reproduction numbers below $7$ there is a tendency to higher maxima for increasing numbers of subcompartments. In particular, for reproduction numbers greater or equal to $7$, the ODE model predicts a higher epidemic peak than the considered LCT models and a lower peak epidemic peak otherwise. For reproduction numbers less than or equal to $3$, there is a tendency to reach the peak earlier, while for reproduction numbers greater than $3$, there is a tendency to reach the peak later as the number of subcompartments increases. In addition, for a fixed reproduction number, the absolute value of the slope in both plots decreases with an increase in the number of subcompartments. Accordingly, the difference in peak size and timing between two consecutive numbers of subcompartments is negligible if the numbers are large. If the number of subcompartments is relatively low, the discrepancy in peak size and timing between each consecutive number of subcompartments is more pronounced.
However, an essential finding is that we cannot make a generally valid statement for all reproduction numbers about the relationship between the time and the size of the peak for different subcompartments.

In addition to the dependence on the reproduction number, the relation of the size and timing of epidemic peaks for different subcompartment numbers is dependent on the selected model parameters. In~\cref{fig:compare_peaksize_timing_halved} and in~\cref{fig:compare_peaksize_timing_doubled}, we provide analyses with a halved and a doubled stay time $T_E$, respectively, compared to~\cref{tab:COVID-19parameters}. We notice that the thresholds at which the ODE model predicts a higher \AW{or earlier} epidemic peak are shifting. 

For the case of a halved stay time in the Exposed compartment in~\cref{fig:compare_peaksize_timing_halved}, the ODE model has the lowest epidemic peak for all reproduction numbers. We find that the epidemic peak is reached later as the number of subcompartments increases for reproduction numbers bigger than \LP{$7$}. For the original case, this threshold was a reproduction number of $3$. In~\cref{fig:compare_peaksize_timing_doubled} with a doubled latent period, the epidemic peak is reached later for higher subcompartment numbers and reproduction numbers larger than $2$. Moreover, for reproduction numbers greater or equal to $5$, the ODE model predicts a higher epidemic peak than LCT models with more subcompartments.
Although omitted in the plots, we observed that for even longer latent periods, the ODE model predicts higher peaks for all reproduction numbers.
Therefore, we find that the relation of the size and timing of the epidemic peaks for different numbers of subcompartments is dependent of the parameter choices for, e.g., the relation between latent and infectious period and, in particular, on the reproduction number.

Let us consider our findings in the context of some existing studies dealing with epidemic peaks. Wearing et al.~\cite{wearing_appropriate_2005} demonstrate that an ODE-SIR model predicts a slower initial increase in the number of Infectives than a corresponding LCT model with $100$ subcompartments. Additionally, the authors illustrate that the maximum number of infected individuals is significantly lower for the ODE model. In their study, Wearing et al. selected a reproduction number of five. Blythe et al.~\cite{blythe_distributed_1988} analyze a model for HIV with birth and death rates and different distributions for the infectious and the incubation period.
The authors observed that the peak occurred later and was higher for distribution assumptions with lower variances. Lastly, Blyuss et al.~\cite{blyuss_effects_2021} examine a model for the spread of COVID-$19$ that incorporates subcompartments for the Exposed and the infectious compartments. The authors observe that the epidemic peak is reached earlier when the number of subcompartments in the Exposed compartment is increased. Furthermore, the maximum value of infectious individuals is determined by the number of subcompartments for $I$ and is observed to increase with an increase in the number of subcompartments. Accordingly, Blyuss et al.~\cite{blyuss_effects_2021} observe that the epidemic peak is reached earlier with higher numbers of subcompartments, whereas Blythe et al.~\cite{blythe_distributed_1988} observe a later peak. All studies conclude that the assumption of higher numbers of subcompartments results in a larger predicted epidemic peak. The results of our simulations indicate that the timing and size of the epidemic peak are significantly influenced by the effective reproduction number and the specific parameter values assumed. In light of these considerations, neither of the studies can be regarded as making universally valid statements.

\begin{remark}
Diekmann et al. provide in their work~\cite{diekmann_discrete-time_2021} a different solution to circumvent the complexity associated with IDE models. They formulate and analyze a discrete-time version of IDE models for epidemic outbreaks. One central finding is that simple ODE models predict smaller peak sizes than models with a fixed length of the latent and the infectious period. As observed in~\cref{sec:properties}, for $n\to\infty$, the LCT model also converges to fixed stay times. The authors fix the reproduction number at $2.5$ and additionally use the same initial growth rate for both models under comparison. This may lead to differing model parameters, e.g., regarding the stay time in the infectious compartment, see~\cite[Appendix]{diekmann_discrete-time_2021}. Future research could also consider which results are obtained with a fixed initial growth rate combined with different values of the reproduction number.
\end{remark}

\begin{figure}[!bt]
    \centering
    \includegraphics[width=0.97\textwidth]{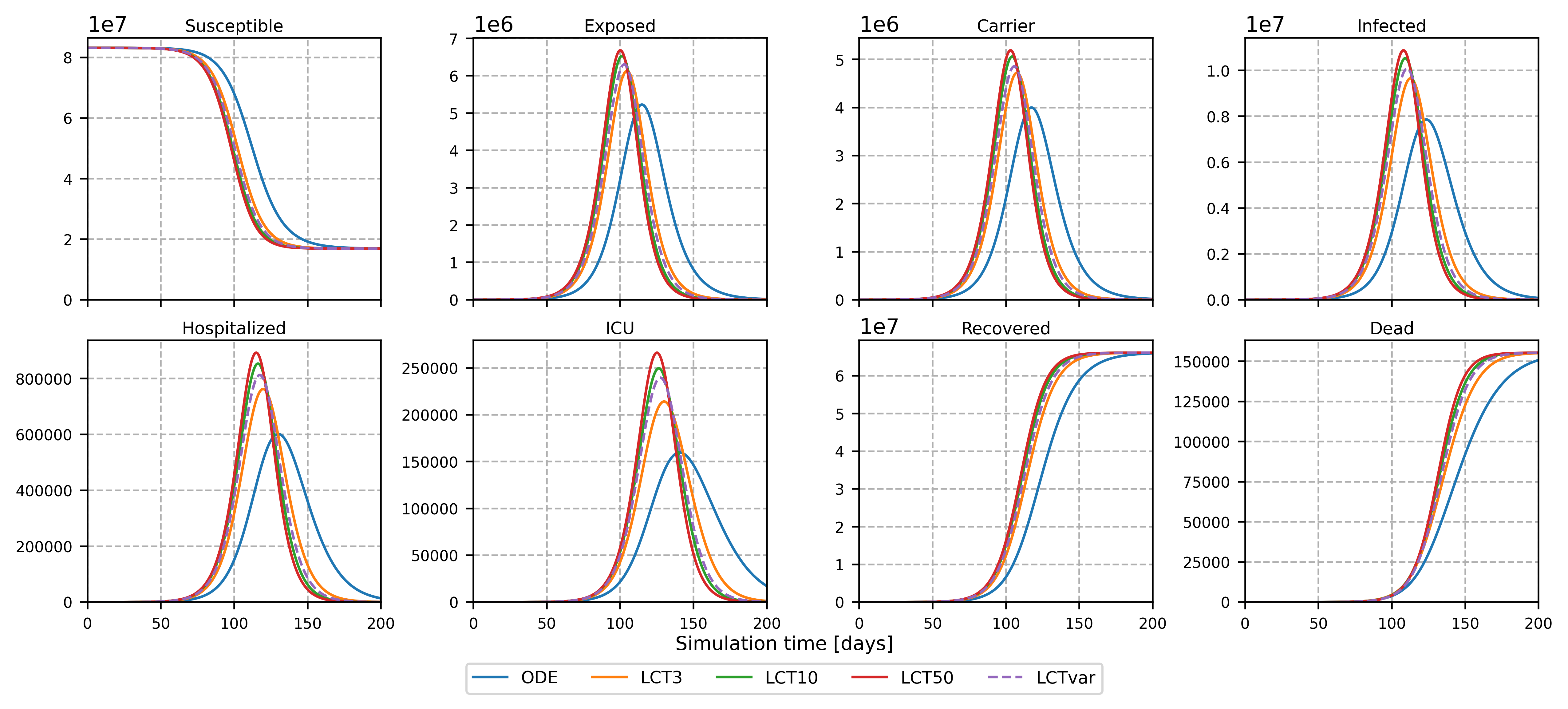}
\caption{\textbf{Compartment sizes to compare predicted epidemic peaks.} 
Presentation of the simulation results for the number of individuals in the eight compartments of the ODE model in comparison with LCT models with different numbers of subcompartments \LP{and $\mathcal{R}_{\text{eff}}(\LP{0})\approx2$}. \MK{Notation as in~\cref{fig:changepoints}.}
}
    \label{fig:compartments_peak_2.0}
\end{figure}

\cref{fig:compartments_peak_2.0} depicts the predicted number of individuals in each compartment for $\mathcal{R}_{\text{eff}}(\LP{0})\approx 2$. The figure shows the significance of the assumption regarding the number of subcompartments for the prediction of the maximum capacity needed in hospitals or the maximum number of intensive care beds required.
\LP{For $\mathcal{R}_{\text{eff}}(\LP{0})\approx 2$, we observe that models with more subcompartments predict higher peak values. Consequently, when using the ODE model to predict the required hospital and intensive care bed capacity, the estimated number of beds will be lower than the predictions from the LCT models. Given that the Erlang distributed stay time assumption employed in the LCT models is regarded as more realistic, the true required bed capacity may in fact be closer to the peak predicted by one of these models. Therefore, relying on the exponential stay time assumption in the ODE model may result in an underestimation of the necessary capacity, leading to an increased risk of the healthcare system being overwhelmed.}
A similar result is also obtained in~\cite{kissler_projecting_2020}. However, our findings regarding the predicted peaks for the daily new transmission suggest that this may not always be the case and that we may obtain an opposite behavior for a different parameterization.
\begin{table}[bt]
\centering
\def\arraystretch{1.0}
\begin{tabular}{>{\columncolor{gray!20}}c >{\columncolor{gray!20}}c | >{\centering}p{2cm} | >{\centering}p{2cm}>{\centering}p{2cm}>{\centering}p{2cm}>{\centering\arraybackslash}p{2cm} }
\toprule
   \rowcolor{gray!20} && ODE & LCT$3$ & LCT$10$ & LCT$50$ & \vphantom{0}{\LP{LCTvar} } \\
\midrule[0.6pt]
    $\mathcal{R}_{\text{eff}}(\LP{0})\approx2$ & final size & \vphantom{0}\LP{$66277719$} & \vphantom{0}\LP{$66271490$} & \vphantom{0}\LP{$66268730$} & \vphantom{0}\LP{$66267697$} & \vphantom{0}\LP{$66269604$}\\ 
      &  rel. diff. & -- & \vphantom{0}\LP{$-0.0094\%$} & \vphantom{0}\vphantom{0}\LP{$-0.0136\%$} & \vphantom{0}\LP{$-0.0151\%$} & \vphantom{0}\LP{$-0.0122\%$}\\
\midrule
   $\mathcal{R}_{\text{eff}}(\LP{0})\approx4$ & final size & \vphantom{0}\LP{$81507936$} & \vphantom{0}\LP{$81506779$} & \vphantom{0}\LP{$81506459$} & \vphantom{0}\LP{$81506385$} & \vphantom{0}\LP{$81506510$}\\
   &  rel. diff.  & -- & \vphantom{0}\LP{$-0.0014\%$} & \vphantom{0}\LP{$-0.0018\%$} & \vphantom{0}\LP{$-0.0019\%$} & \vphantom{0}\LP{$-0.0018\%$}\\
\midrule
    $\mathcal{R}_{\text{eff}}(\LP{0})\approx10$ & final size & \vphantom{0}\LP{$83151260$} & \vphantom{0}\LP{$83151255$} & \vphantom{0}\LP{$83151254$} & \vphantom{0}\LP{$83151254$} & \vphantom{0}\LP{$83151254$}\\
   &  rel. diff. &-- & \vphantom{0}\LP{$-0.000006\%$} & \vphantom{0}\LP{$-0.000007\%$} & \vphantom{0}\LP{$-0.000007\%$} & \vphantom{0}\LP{$-0.000007\%$}\\
\bottomrule
\end{tabular}
\caption{\textbf{Predicted final sizes of different LCT models and relative differences to the ODE result.} 
Comparison of the predicted final size $N(0)-S_{\infty}$ for different effective reproduction numbers \LP{$\mathcal{R}_{\text{eff}}(\LP{0})$} and numbers of subcompartments. The final size is rounded to the nearest integer in each case. Additionally, the relative differences of the LCT models to the result obtained with the ODE models are shown for each reproduction number. For all cases, the final size is calculated using the number of Susceptibles at $t=500$, as the compartment size changes not significantly anymore.}
\label{tab:final_size}
\end{table}

For the Susceptible and Recovered compartments, it can be observed that the curves reach a comparable level for all subcompartment selections at the end of the simulation, although the curves differ in their shape. 
This leads us to a comparison of the final size, which is the total number of individuals who become infected over the course of the epidemic
\begin{align*}
    N(0)-S_{\infty},\quad\text{ with }\quad S_{\infty}=\lim_{t\to\infty}{S(t)};
\end{align*}
cf.~\cite{brauer_mathematical_2019}. \cref{tab:final_size} shows absolute values for the final size for the effective reproduction numbers $2$, $4$ and $10$ for various assumptions regarding the number of subcompartments. To compare the predicted final size of the epidemic from different models, the relative deviations of the LCT models to the result of the ODE model are given for each reproduction number.
We see that, indeed, the relative differences between the ODE model and the LCT models are close to zero for all subcompartment choices. 

In~\cite{ma_generality_2006}, it is stated that the final size is independent of the number of subcompartments used for a SEIR model in the latent and infectious state. \LP{Additionally, the author of~\cite{Brauer_Age-of-infection_2008} identifies a final size relation for an IDE-based model that remains independent of assumptions regarding the stay time distribution and depend only on the reproduction number and initial conditions.} Our numerical experiments align with \LP{these findings, demonstrating that the final size in our model is dependent only on the reproduction number and is unaffected by the number of subcompartments chosen}.
 \subsection{Impact of age resolution}\label{sec:ageresolution}
\begin{figure}[htb]
    \centering
    \includegraphics[width=0.375\textwidth]{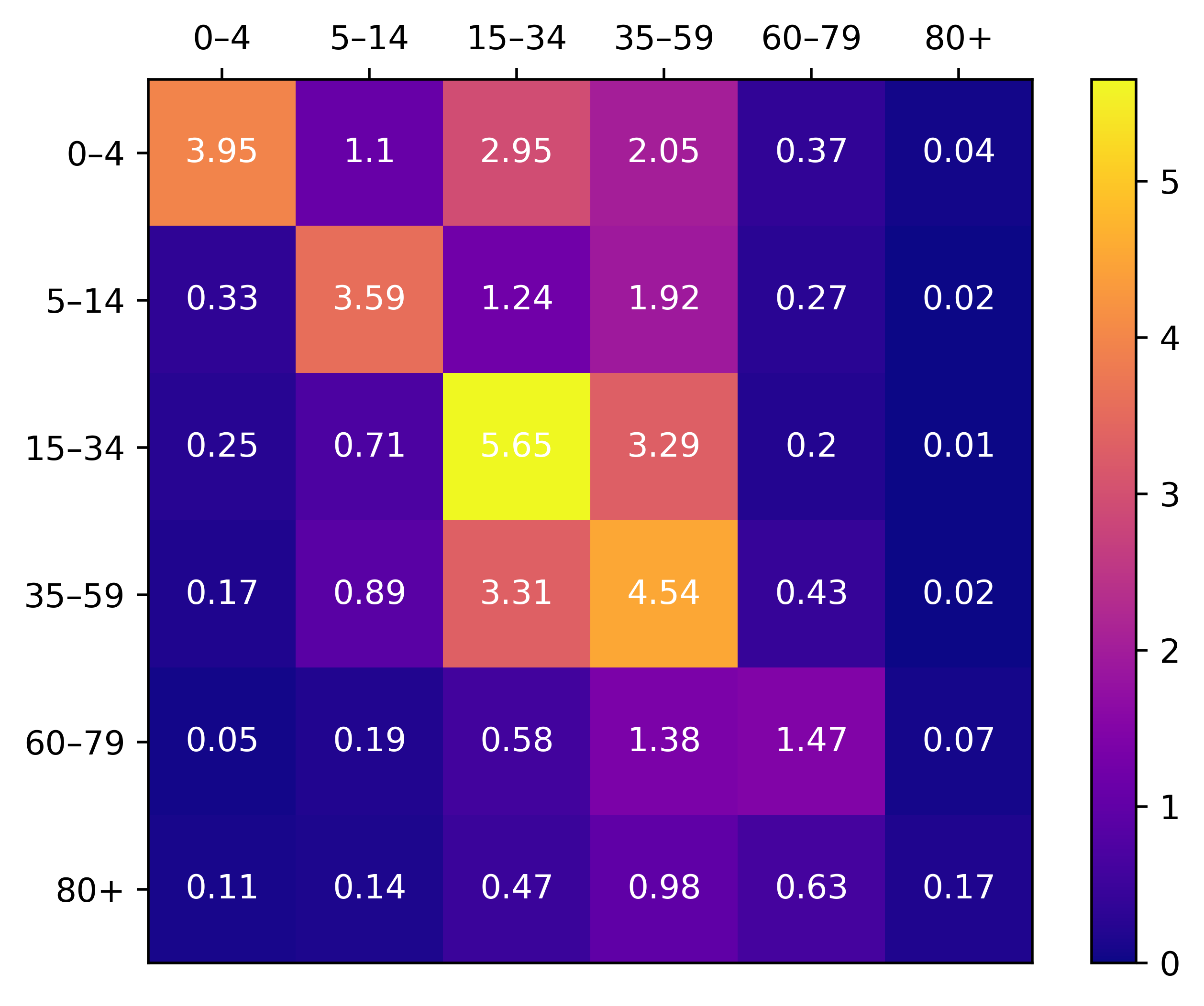}
\caption{\textbf{Age-resolved contact matrix for Germany.} For each RKI age group represented in the rows, the average number of daily contacts with the respective age group is provided in the columns. The contact data are based on~\cite{prem_projecting_2017} combined with~\cite{fumanelli_inferring_2012}.}
    \label{fig:contacts}
\end{figure}
\MK{In the previous sections, we have excluded age groups and demographic resolution to separate the effect of the distribution assumption in ODE and different LCT models from the effect of demographic resolution. In this section, we consider an age-resolved modeling approach, showcasing fundamental differences in the outcomes compared to a non-age-resolved approach.} For this purpose, we compare a model not stratified by age (using only one age group for the total population) with a model that is resolved according to the age groups used in the RKI data. \MK{The findings of the previous sections will then be driven by the age groups that contribute most to a particular situation of epidemic dynamics.}

The (baseline) contact pattern for Germany shown in~\cref{fig:contacts} is based on~\cite{prem_projecting_2017} combined with~\cite{fumanelli_inferring_2012} for school contacts. For more details, also see~\cite{kuhn_assessment_2021}. The subpopulation-weighted average of the total number of contacts is $7.69129$. We use the LCT$10$ model to demonstrate the impact of the age resolution. In all simulations, the number of individuals in the Exposed compartment is set to $100$ and their number is distributed uniformly across the subcompartments. The remainder of the German population is assigned to the Susceptible compartment. We consider two distinct age-resolved simulations and one simulation without age resolution:
\begin{enumerate}[I]
    \item \label{item:uni} \textbf{A15--34 scenario:} In the first simulation, we assume that all initial infections circulate within the age group A$15$--$34$. Accordingly, $100$ Exposed individuals are assigned to the age group of $15$--$34$ years. 
    \item \label{item:retirement} \textbf{A80+ scenario:} In the second simulation, we assume that all initial infections circulate within the age group of the oldest people. Accordingly, $100$ Exposed individuals are assigned to the age group of $80+$ years. 
    \item\label{item:noage} \textbf{Non-age-resolved scenario:} For comparison, the simulations are run without age groups. 
\end{enumerate}
Here, the idea is that if either scenario~\ref{item:uni} or~\ref{item:retirement} occurs in reality, both are translated into case~\ref{item:noage} if we use a model without age resolution. We show that it can be significant for the qualitative and quantitative simulation results to capture differences in the age groups in the model when confronted with contact patterns similar to Germany. Note that substantial differences are expected if age groups mix more homogeneously and if the off-diagonal entries of the contact matrix are more pronounced or if diagonal entries differ less among each other.

\begin{figure}[!bt]
    \centering
    \includegraphics[width=0.97\textwidth]{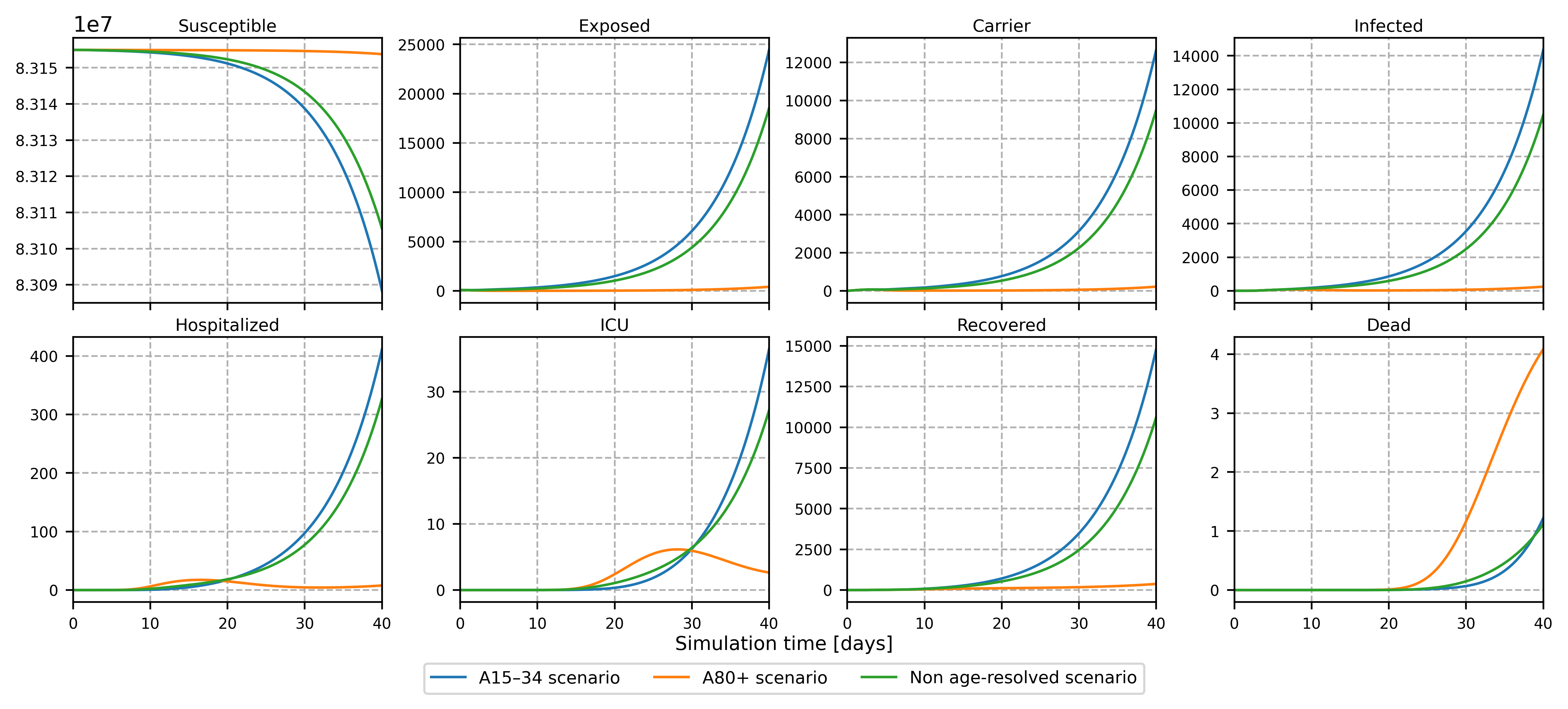}
\caption{\textbf{Comparison of different scenarios with and without age resolution.} Comparison of a non-age-resolved model with the results of an age-resolved model, where the $100$ initially Exposed individuals are assumed to be in different age groups. We use $n=10$ subcompartments for each compartment $Z\in\mathcal{A}$. Age-resolved simulation results are aggregated for visualization purposes.}
    \label{fig:age}
\end{figure}
\begin{figure}[!bt]
    \centering
    \includegraphics[width=0.97\textwidth]{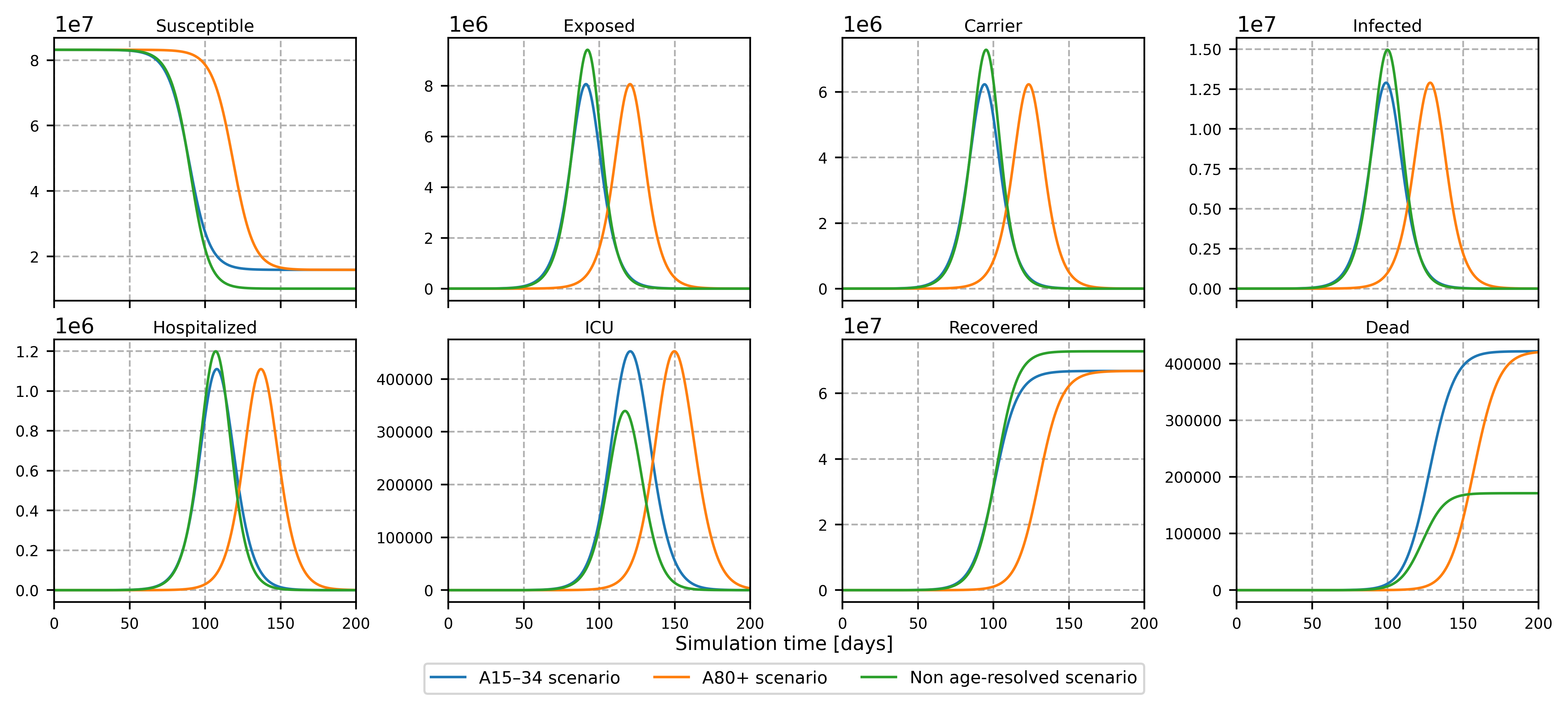}
\caption{\LP{\textbf{Comparison of different scenarios with and without age resolution to compare epidemic peaks.}} \LP{Comparison of a non-age-resolved model with the results of an age-resolved model for $200$ simulation days to compare epidemic peaks. The $100$ initially Exposed individuals are assumed to be in different age groups. We use $n=10$ subcompartments for each compartment $Z\in\mathcal{A}$. Age-resolved simulation results are aggregated for visualization purposes.}}
    \label{fig:age-long}
\end{figure}
The results \LP{for $40$ simulation days} are depicted in~\cref{fig:age}. \LP{The predictions with the relative small time frame prepare the ground for intervention and mitigation actions.} To compare the simulations with and without age resolution, we cumulate the age-resolved results in the compartments. Although we initially had the same number of individuals in the Exposed compartment in each experiment, the results differ significantly. The elderly population has a markedly low number of daily contacts, as can be seen in~\cref{fig:contacts}, which is why the number of additional infections \LP{in scenario~\ref{item:retirement}} remains low \LP{over the simulated time period}. The contact rate for the age group $15$--$34$ years is above average. Although the majority of their contacts occur with individuals belonging to age groups with below-average transmission probabilities $\rho_i$, cf.~\cref{fig:contacts} and~\cref{tab:COVID-19parameters}, the observed spread for scenario~\ref{item:uni} is faster than that predicted in the scenario without age groups. The likelihood of being hospitalized or dying from the disease is significantly higher for older people, see also~\cref{tab:COVID-19parameters}. Accordingly, we get nontrivial numbers of deaths and hospitalizations for scenario~\ref{item:retirement}, despite the relatively low infection dynamics.

\LP{\cref{fig:age-long} provides the results of simulations over a longer period of $200$ days. The figure shows that the size and the timing of epidemic peaks are affected by the use of age groups. The peaks of scenario~\ref{item:retirement} are reached later and the peak sizes and the final size vary for the non-age-resolved scenario~\ref{item:noage} compared to the age-resolved scenarios. The peak sizes of both age-resolved scenarios are broadly similar; the peak size of the daily new transmissions of scenario~\ref{item:uni} is approximately $2437090$ compared to $2437071$ for scenario~\ref{item:retirement}.} \MK{However, as noted before, these peaks can be considered counterfactual and unrealistic to be reached since suitable interventions are most likely to be implemented on time. Overall, from this section, we conclude that the integration of age groups and demographic information result in a notable enhancement} of the model's realism, allowing for more \MK{appropriate intervention schemes}.

\subsection{Simulation of COVID-19 in Germany}\label{sec:covid}
In this section, we examine \LP{the applicability of} age-resolved LCT models in a realistic context in comparison with an age-resolved ODE model. \LP{We demonstrate how reported data can be used to initialize the models and apply the reported data to compare our simulation results to them.} For this, we consider the spread of COVID-19 in Germany in October 2020. We again use the epidemiological model parameters defined in~\cref{sec:parameters_data} and apply the contact matrix for Germany described in~\cref{sec:ageresolution} or~\cref{fig:contacts}.

We define initial values for our models based on the aforementioned daily reported and age-resolved data by the RKI in Germany~\cite{RKI_data_2024}. In order to set the initial population of the LCT models, we extend an initialization scheme for an ODE model proposed in~\cite[Appendix S1\LP{~Text}]{zunker_novel_2024} using the LCT (but neglecting different protection levels for susceptible individuals). In the following description, we fix one age group and omit the corresponding age index. However, the scheme is applied to each age group in the implementation using the age-resolved data.

The RKI provides daily data~\cite{RKI_data_2024} regarding the cumulative confirmed cases, which we denote by $\Sigma_{I,\text{RKI}}(t)$. For simplicity, we assume that the reported cases reflect the number of mildly symptomatic individuals. We scale the number of cumulative confirmed cases by a factor $1/d$ to consider a detection ratio $d\in(0,1]$. Hence, the terms $\Sigma_{I,\text{RKI}}(t)$ need to be scaled accordingly by $1/d$. In the following description, we assume $d=1$. For the initialization, we assume, that each individual stays exactly the mean stay time $T_Z/n_Z$ in each subcompartment $Z_j$ for $Z\in\mathcal{A}$; c.f.~\cref{remark:epi_meaning}. We begin by defining the initial values for the disease state $I$ and then proceed to apply this method to the remaining compartments. By the mean stay time assumption for the initialization, individuals that are in subcompartment $I_j$ at time $t_0$ are those who developed mild symptoms between the time points $t_0 - (j-1) \, T_I/n_I$ and $t_0 - j \, T_I/n_I$. Consequently, the number of individuals in subcompartment $I_j$ at time $t_0$ is 
\begin{align}\label{eq:initI}
    I_j(t_0) = \Sigma_{I,\text{RKI}}\left(t_0 - (j-1)\,\frac{T_I}{n_I}\right) - \Sigma_{I,\text{RKI}}\left(t_0 - j\,\frac{T_I}{n_I}\right).
\end{align}
Note that the number of cumulative confirmed cases is reported once per day, but the times at which we evaluate $\Sigma_{I,\text{RKI}}$ does not necessarily correspond to these time points. If the calculation requires data between two consecutive days, the reported RKI data is interpolated linearly.

For the remaining compartments $Z\in\mathcal{A}$, the consideration can be applied analogously. Considering the transition probabilities, we obtain for the respective subcompartments the equations
\begin{align}\label{eq:initECHU}
\begin{aligned}
    E_j(t_0) &= \frac{1}{\mu_C^I} \left(\Sigma_{I,\text{RKI}}\left(t_0 + T_C + (n_E - j + 1) \,\frac{T_E}{n_E}\right) - \Sigma_{I,\text{RKI}}\left(t_0 + T_C + (n_E - j)\,\frac{T_E}{n_E}\right)\right)\\
    C_j(t_0) &= \frac{1}{\mu_C^I} \left(\Sigma_{I,\text{RKI}}\left(t_0 + (n_C - j + 1)\,\frac{T_C}{n_C}\right) - \Sigma_{I,\text{RKI}}\left(t_0 + (n_C - j) \,\frac{T_C}{n_C}\right)\right)\\
    H_j(t_0) &= \mu_I^H \left( \Sigma_{I,\text{RKI}}\left(t_0 - T_I - (j-1)\,\frac{T_H}{n_H}\right) - \Sigma_{I,\text{RKI}}\left(t_0 - T_I - j\,\frac{T_H}{n_H}\right)\right)\\
    U_j(t_0) &= \mu_H^U \mu_I^H \left( \Sigma_{I,\text{RKI}}\left(t_0 - T_I - T_H - (j-1)\,\frac{T_U}{n_U}\right) - \Sigma_{I,\text{RKI}}\left(t_0 - T_I - T_H - j\,\frac{T_U}{n_U}\right)\right).
\end{aligned} 
\end{align} 

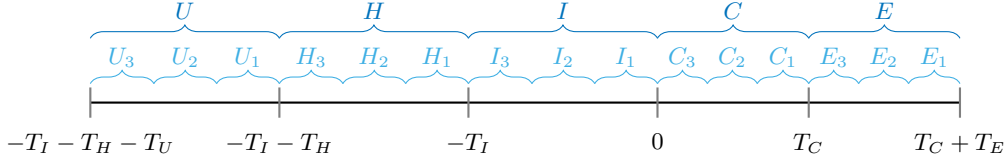
\begin{figure}[tb]
	\centering
	\begin{tikzpicture}[align=center]
    \small
		\draw[ thick] (-7.5,0) -- (4,0);
  
        \node[below] at (0,-0.3){$0$};
		\draw[gray, thick] (0,0.2) -- (0,-0.2);
  
        \draw [CornflowerBlue,decorate,  decoration = {brace, amplitude=5pt}] (0,0.25) --  (0.66,0.25) node[pos=0.5,above=3pt,CornflowerBlue]{\small $C_3$};
        \draw [CornflowerBlue,decorate,  decoration = {brace, amplitude=5pt}] (0.66,0.25) --  (1.32,0.25) node[pos=0.5,above=3pt,CornflowerBlue]{$C_2$};
        \draw [CornflowerBlue,decorate,  decoration = {brace, amplitude=5pt}] (1.32,0.25) --  (2,0.25) node[pos=0.5,above=3pt,CornflowerBlue]{$C_1$};

        \draw [RoyalBlue, decorate,  decoration = {brace, amplitude=5pt}] (0,0.85) --  (2,0.85) node[pos=0.5,above=4pt, RoyalBlue]{$C$};
        
        \node[below] at (2,-0.3){$T_C$};
		\draw[gray, thick] (2,0.2) -- (2,-0.2);
  
        \draw [CornflowerBlue,decorate,  decoration = {brace, amplitude=5pt}] (2,0.25) --  (2.66,0.25) node[pos=0.5,above=3pt,CornflowerBlue]{\small $E_3$};
        \draw [CornflowerBlue,decorate,  decoration = {brace, amplitude=5pt}] (2.66,0.25) --  (3.32,0.25) node[pos=0.5,above=3pt,CornflowerBlue]{$E_2$};
        \draw [CornflowerBlue,decorate,  decoration = {brace, amplitude=5pt}] (3.32,0.25) --  (4,0.25) node[pos=0.5,above=3pt,CornflowerBlue]{$E_1$};

        \draw [RoyalBlue, decorate,  decoration = {brace, amplitude=5pt}] (2,0.85) --  (4,0.85) node[pos=0.5,above=4pt, RoyalBlue]{$E$};
        
        \node[below] at (4,-0.3){$T_C+T_E$};
		\draw[gray, thick] (4,0.2) -- (4,-0.2);

        \draw [CornflowerBlue,decorate,  decoration = {brace, amplitude=5pt}] (-0.83,0.25) --  (0,0.25) node[pos=0.5,above=3pt,CornflowerBlue]{\small $I_1$};
        \draw [CornflowerBlue,decorate,  decoration = {brace, amplitude=5pt}] (-1.66,0.25) --  (-0.83,0.25) node[pos=0.5,above=3pt,CornflowerBlue]{$I_2$};
        \draw [CornflowerBlue,decorate,  decoration = {brace, amplitude=5pt}] (-2.5,0.25) --  (-1.66,0.25) node[pos=0.5,above=3pt,CornflowerBlue]{$I_3$};

        \draw [RoyalBlue, decorate,  decoration = {brace, amplitude=5pt}] (-2.5,0.85) --  (0,0.85) node[pos=0.5,above=4pt, RoyalBlue]{$I$};

        \node[below] at (-2.5,-0.3){$-T_I$};
		\draw[gray, thick] (-2.5,0.2) -- (-2.5,-0.2);
  
       \draw [CornflowerBlue,decorate,  decoration = {brace, amplitude=5pt}] (-3.33,0.25) --  (-2.5,0.25) node[pos=0.5,above=3pt,CornflowerBlue]{\small $H_1$};
        \draw [CornflowerBlue,decorate,  decoration = {brace, amplitude=5pt}] (-4.16,0.25) --  (-3.33,0.25) node[pos=0.5,above=3pt,CornflowerBlue]{$H_2$};
        \draw [CornflowerBlue,decorate,  decoration = {brace, amplitude=5pt}] (-5,0.25) --  (-4.16,0.25) node[pos=0.5,above=3pt,CornflowerBlue]{$H_3$};

        \draw [RoyalBlue, decorate,  decoration = {brace, amplitude=5pt}] (-5,0.85) --  (-2.5,0.85) node[pos=0.5,above=4pt, RoyalBlue]{$H$};

        \node[below] at (-5,-0.3){$-T_I-T_H$};
		\draw[gray, thick] (-5,0.2) -- (-5,-0.2);

        \draw [CornflowerBlue,decorate,  decoration = {brace, amplitude=5pt}] (-5.83,0.25) --  (-5,0.25) node[pos=0.5,above=3pt,CornflowerBlue]{\small $U_1$};
        \draw [CornflowerBlue,decorate,  decoration = {brace, amplitude=5pt}] (-6.66,0.25) --  (-5.83,0.25) node[pos=0.5,above=3pt,CornflowerBlue]{$U_2$};
        \draw [CornflowerBlue,decorate,  decoration = {brace, amplitude=5pt}] (-7.5,0.25) --  (-6.66,0.25) node[pos=0.5,above=3pt,CornflowerBlue]{$U_3$};

        \draw [RoyalBlue, decorate,  decoration = {brace, amplitude=5pt}] (-7.5,0.85) --  (-5,0.85) node[pos=0.5,above=4pt, RoyalBlue]{$U$};

        \node[below] at (-7.5,-0.3){$-T_I-T_H-T_U$};
		\draw[gray, thick] (-7.5,0.2) -- (-7.5,-0.2);
  
    \normalsize
	\end{tikzpicture}
\caption{\textbf{Illustration of the initialization method for LCT models.} Visualization of the determination of the initial values for the LCT model for the starting time $t_0=0$. The diagram is created for a model with $n_z=3$ subcompartments for all $z\in\mathcal{A}$. The brackets indicate that the number of cases confirmed in that time interval determines the initial value for the specified compartment or subcompartment. The spaces between the times are set arbitrarily and do not correspond to the parameters used. This representation is inspired by~\cite{koslow_appropriate_2022}.}
\label{fig:initialization}
\end{figure}
\begin{figure}[bt]
\centering
    \begin{minipage}[t]{0.4\textwidth}
    \centering
     \includegraphics[width=\textwidth]{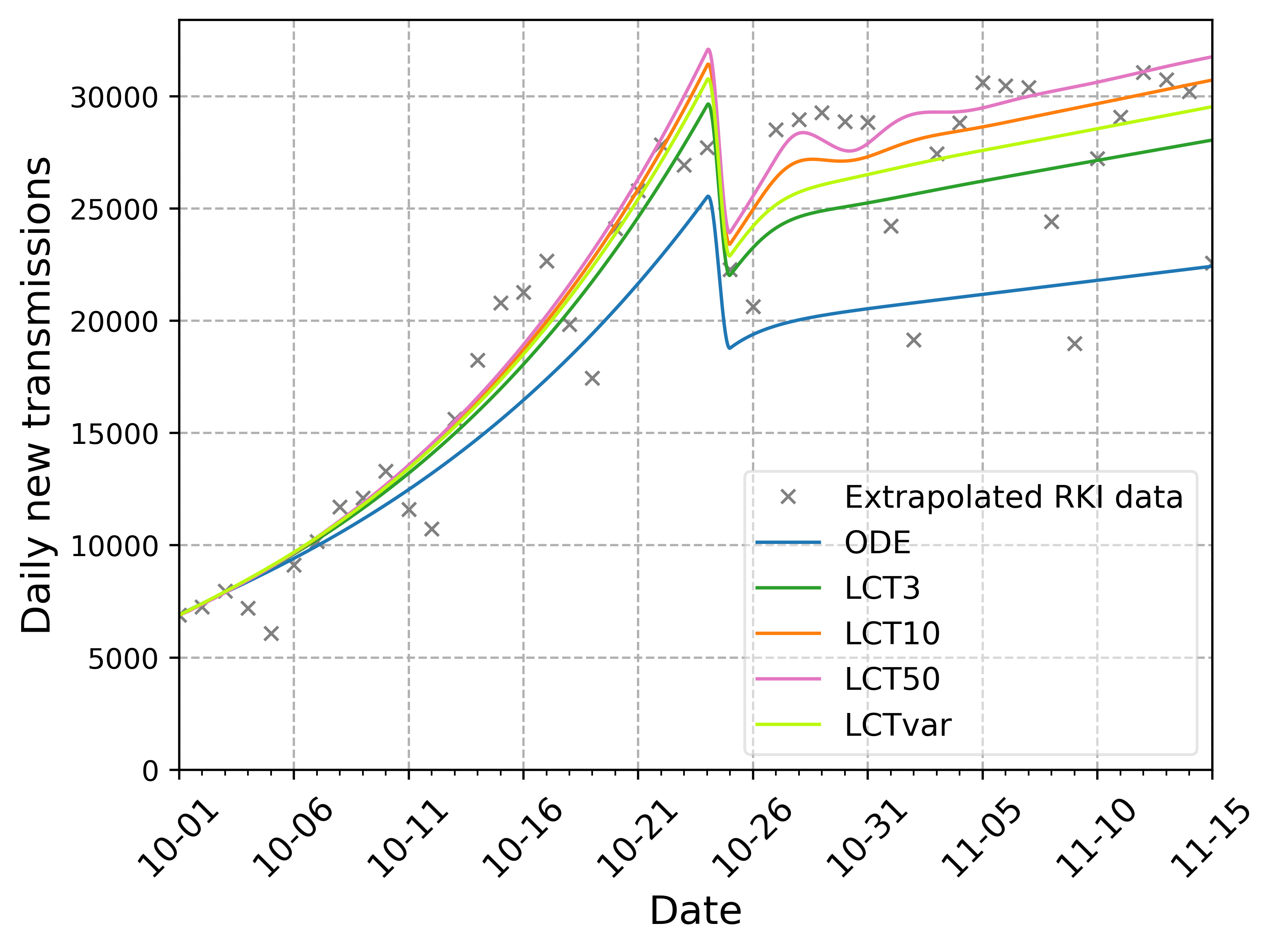}
\end{minipage}
\begin{minipage}[t]{0.4\textwidth}
 \centering
    \includegraphics[width=\textwidth]{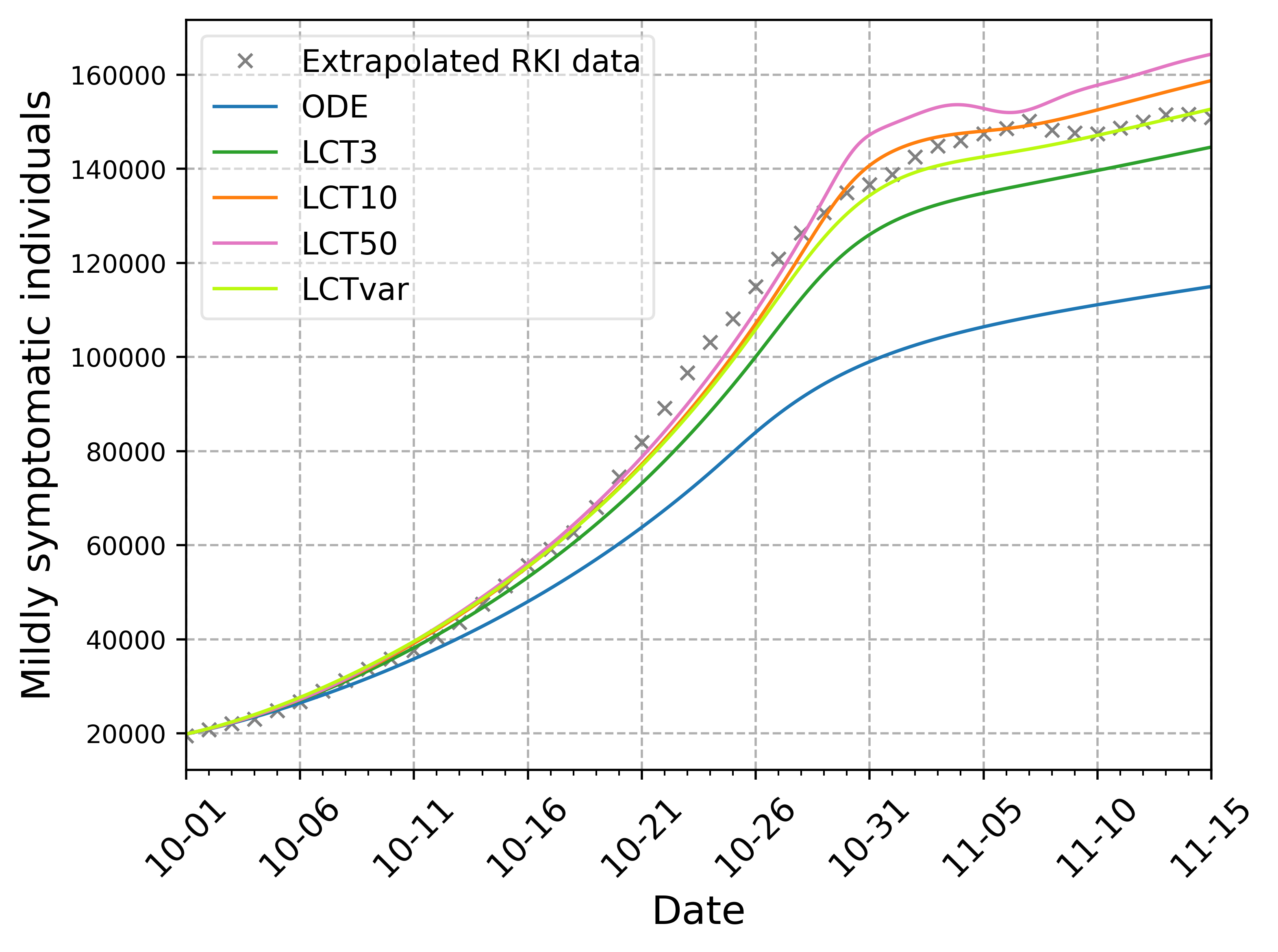}
\end{minipage}  
    \centering
    \begin{minipage}[t]{0.4\textwidth}
    \centering
     \includegraphics[width=\textwidth]{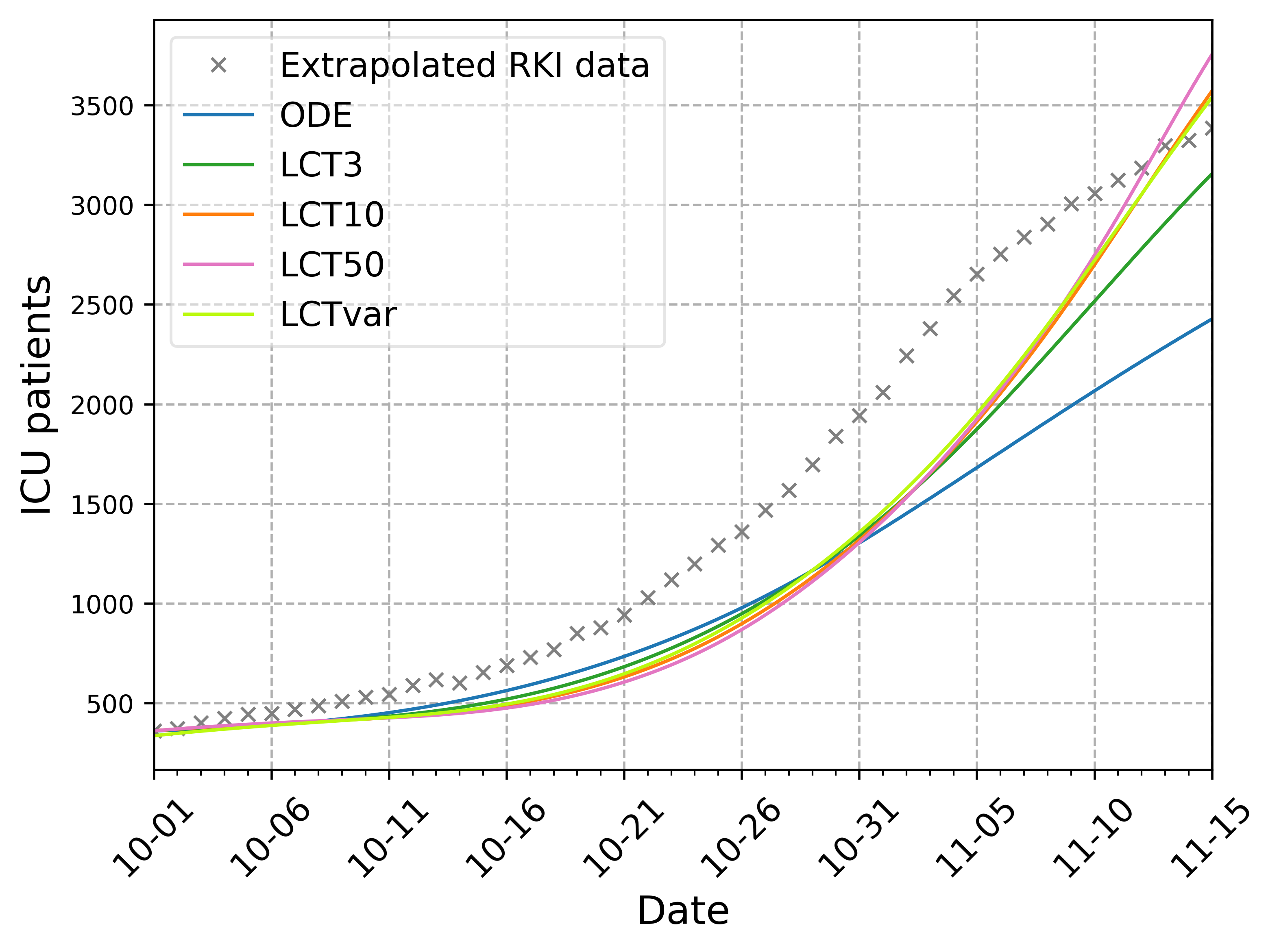}
\end{minipage}
\begin{minipage}[t]{0.4\textwidth}
 \centering
    \includegraphics[width=\textwidth]{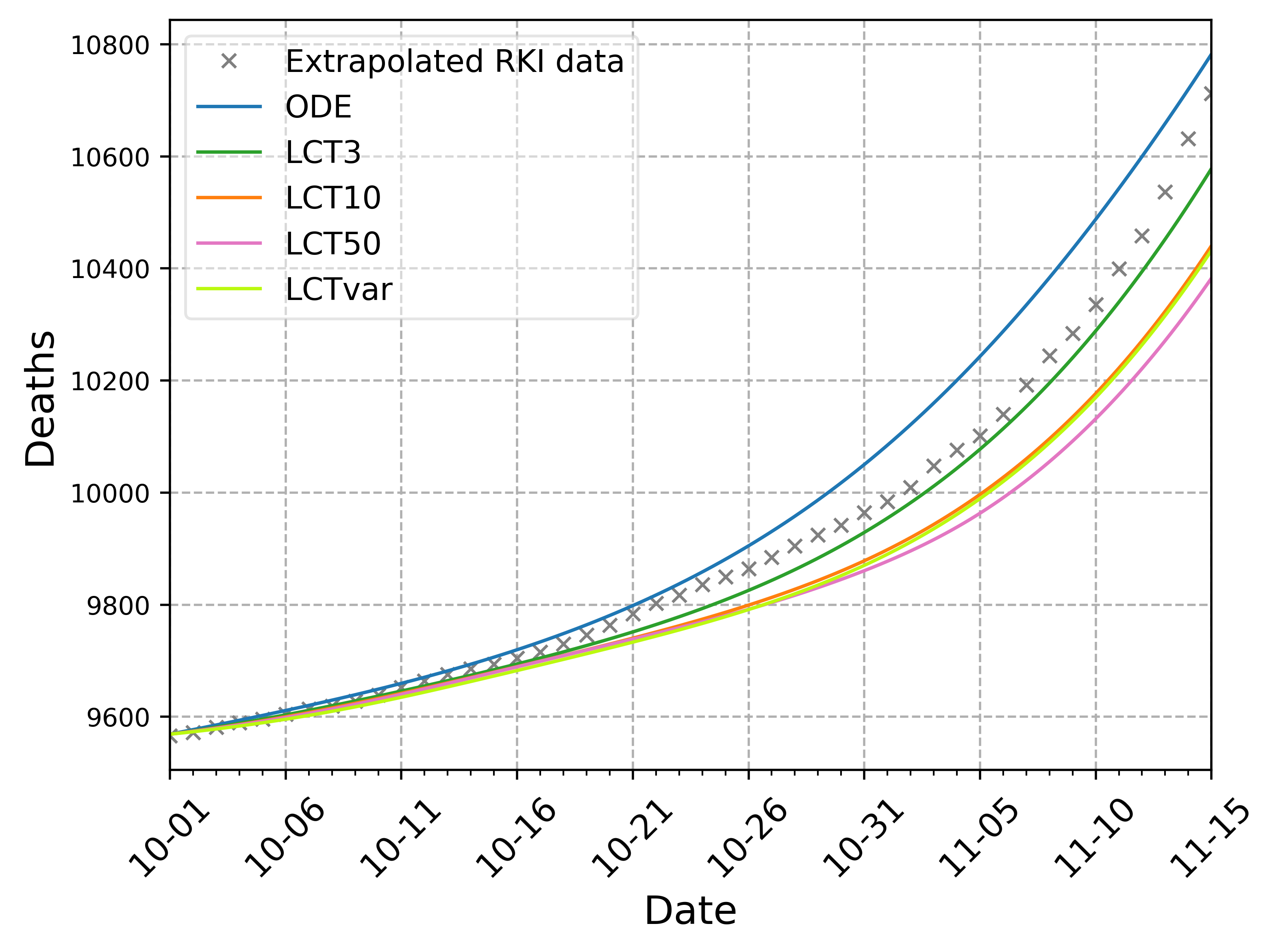}
\end{minipage}  
\caption{\textbf{Results for COVID-19 in Germany with start date Oct 1, 2020.} Comparison of extrapolated RKI data with the simulation results of a simple ODE model and of LCT models with different assumptions regarding the number of subcompartments. We compare the number of daily new transmissions (top left), the number of mildly symptomatic
individuals (top right), the number of patients in intensive care units (bottom left) and the number of deaths (bottom right). 
\MK{Notation as in~\cref{fig:changepoints}.}}
    \label{fig:october_realistic}
\end{figure}
\cref{fig:initialization} provides a schematic illustration of the relevant time intervals of the confirmed case data for the subcompartments of each compartment, exemplified by $n_Z=3$ for all $Z\in\mathcal{A}$. While the number of patients in intensive care units is provided by an additional daily report~\cite{divi2024}, the data set does not include age-specific data. Therefore, we use the reported number of patients~\cite{divi2024} to scale the result of~\eqref{eq:initECHU} for compartment $U$ for each subcompartment and age group, ensuring that the initial total number of ICU patients is consistent to the reported data.

The compartments $S$, $R$ and $D$ are computed analogously to the ODE model in~\cite{koslow_appropriate_2022}. For the sake of completeness, we provide the formulas below. As the reported data for the deaths, $\Sigma_{D,\text{RKI}}(t)$, contain the deaths reporting only the day of the first positive test and the assumed day of infection -- instead of the day of death, we include a time shift. The number of deaths is
\begin{align} \label{eq:initD}
    D(t_0)=\Sigma_{D,\text{RKI}}(t_0-T_I-T_H-T_U).
\end{align} 
The number of recovered individuals equals the total confirmed cases less deaths and the currently infected people, i.e. 
\begin{align*}
    R(t_0)= \Sigma_{I,\text{RKI}}(t_0)-I(t_0)-H(t_0)-U(t_0)-D(t_0).
\end{align*} Lastly, the set of susceptible individuals is set using the other compartment sizes, 
\begin{align*}
    S(t_0)=N-\sum_{Z\in\mathcal{Z}\setminus S}Z(t_0).
\end{align*}

We compare our simulation results to extrapolated RKI data. To extrapolate the given data, we apply the method defined for the initial values for each simulation day (\LP{without} division in subcompartments). The same scaling on basis of the detection ratio for the reported data as for the initialization is \LP{used}. We adapt equation~\eqref{eq:initI} for each simulation day~$t$ to extrapolated data for mildly symptomatic individuals and time shift the number of deaths according to~\eqref{eq:initD}. A non-age-resolved number of ICU patients is set directly using~\cite{divi2024}. Furthermore, we again look at the number of daily new transmissions $\sigma_{E,\text{RKI}}$, which is the number of people transiting to compartment $E$ within one day. Therefore, according to~\eqref{eq:initECHU}, we use for simulation time~$t$ the formula
\begin{align*}
   \sigma_{E,\text{RKI}}(t) = \frac{1}{\mu_C^I} \left(\Sigma_{I,\text{RKI}}(t + T_C +T_E) - \Sigma_{I,\text{RKI}}(t + T_C + T_E-1)\right).
\end{align*}

We start our simulation on Oct. 1, 2020, and simulate for 45 days. In order to simulate the impact of NPIs, we manually adapt the contact rate during the simulation if the data indicate a different trend. We, however, keep the implemented change points minimal and only allow one change point in the simulation period as new NPIs were neither decreed on a daily nor weekly basis. Firstly, the daily contacts are scaled such that, in the beginning of the simulation, the simulation results for the number of daily new transmissions align with the extrapolated RKI data when all age groups are aggregated. Furthermore, on Oct. 25, 2020 the data indicate a trend change, such that we reduce the contacts by $30\%$, which represents the implementation of a NPI. In addition, we assume a detection ratio of $d = 1/1.2$ over the whole period.

The simulation results are depicted in~\cref{fig:october_realistic}. We \LP{accumulate} the age-resolved simulation results. \LP{We see that all models roughly depict the trends of the infection dynamics. Suitable parameter fitting for the many degrees of freedom and additional numbers of change points can realistically match the simulation results from each of the models to the curves. Given the prior findings that Erlang distributed stay times are more realistic than those of exponential distributions, it can be deduced that the the LCT models are better suited for matching the observed cases than the simple ODE model.}

\subsection{Run time analysis}\label{sec:runtime}
\begin{figure}[b]
    \centering
    \begin{minipage}[t]{0.315\textwidth}
    \centering
     \includegraphics[width=\textwidth]{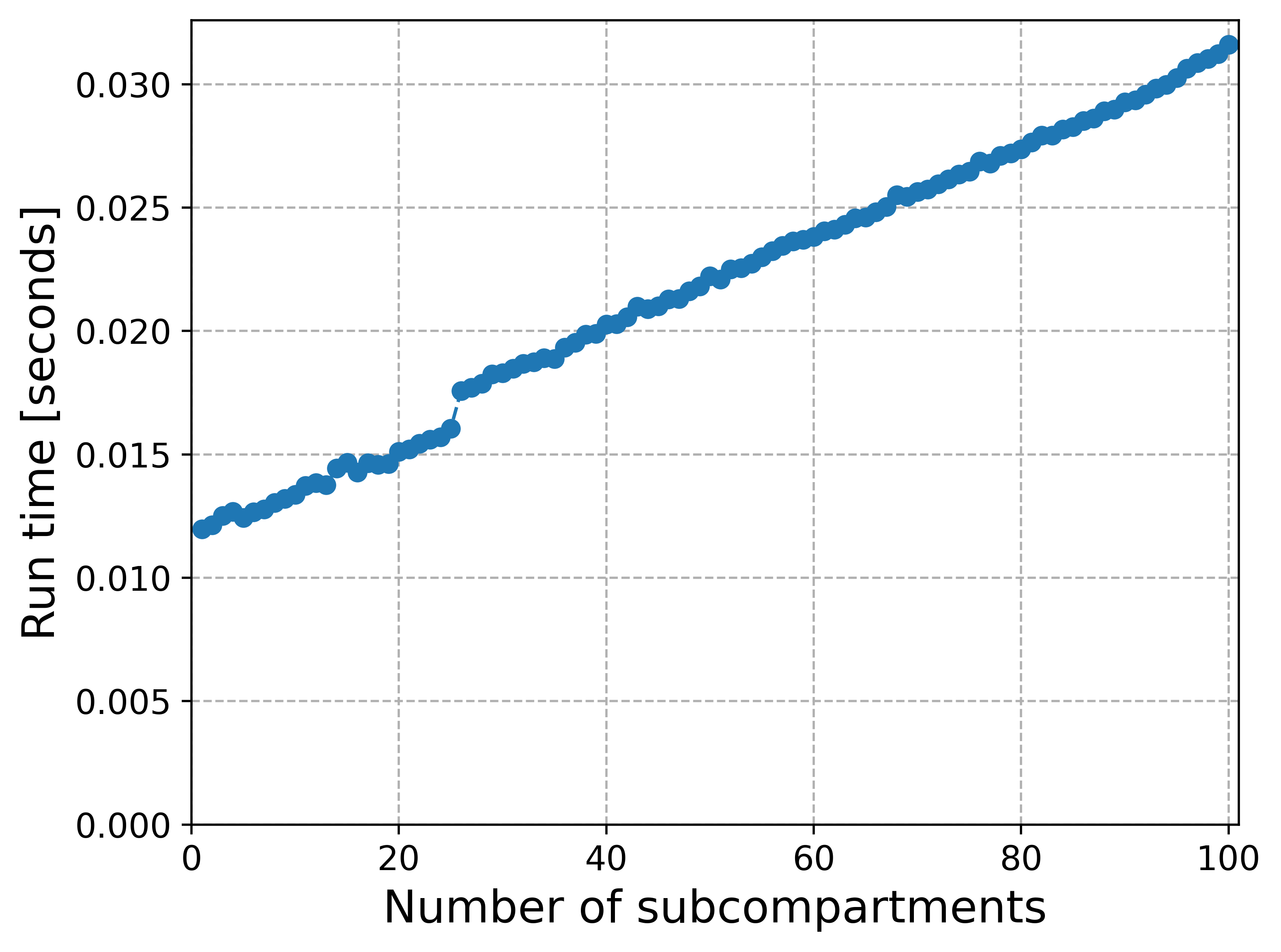}
\end{minipage}
\begin{minipage}[t]{0.315\textwidth}
 \centering
    \includegraphics[width=\textwidth]{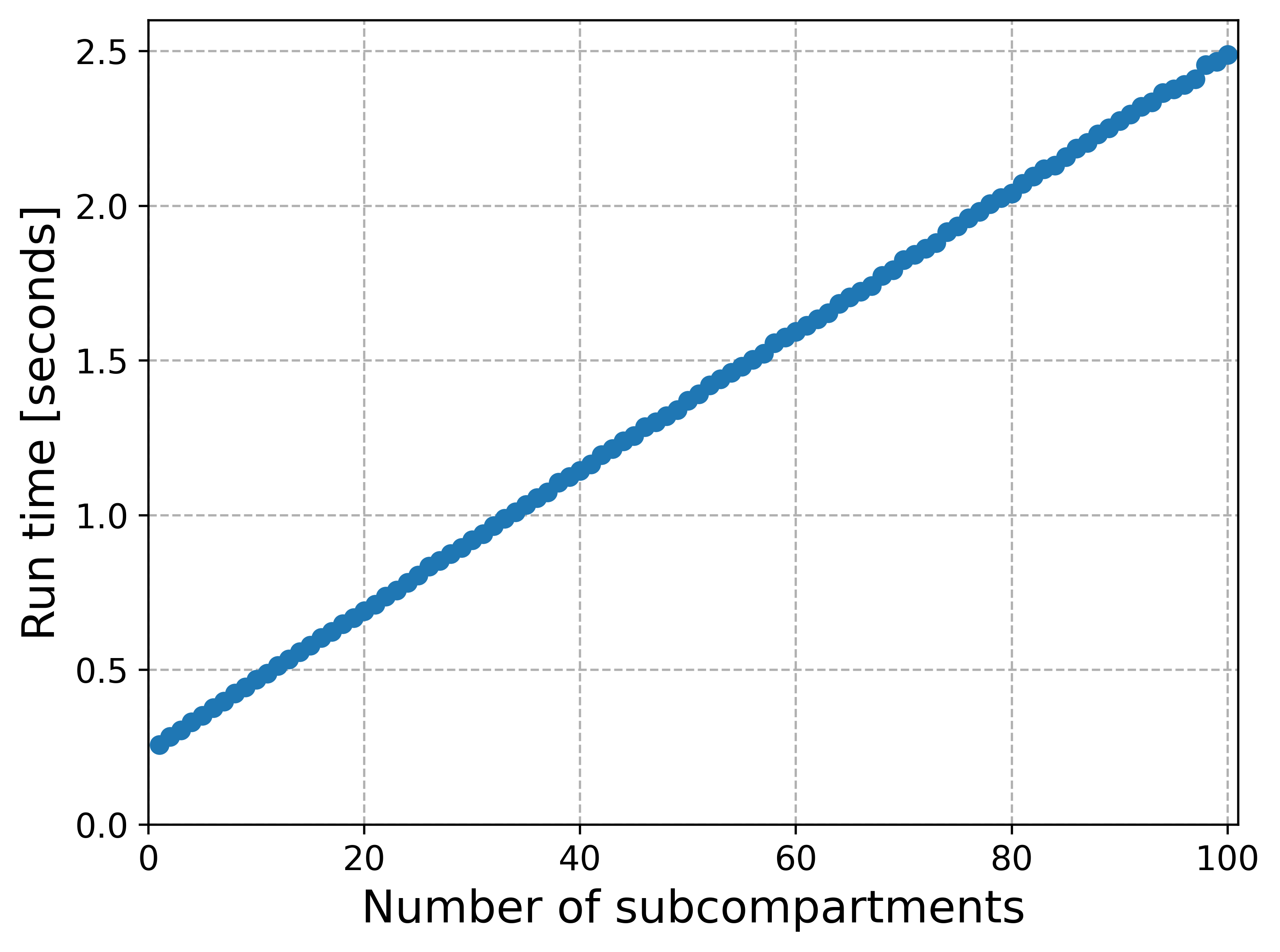}
\end{minipage}  
\begin{minipage}[t]{0.34\textwidth}
 \centering
    \includegraphics[width=\textwidth]{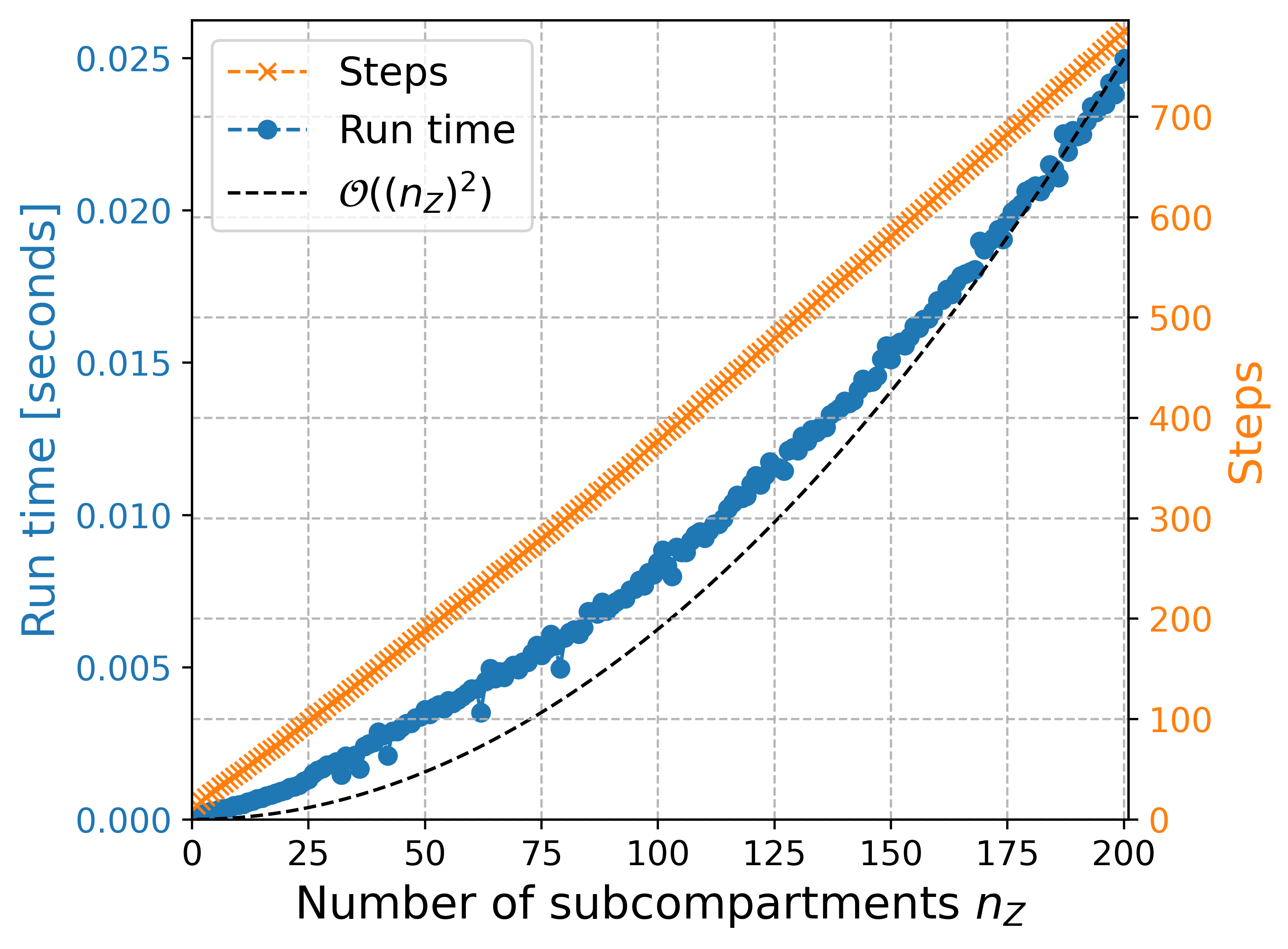}
\end{minipage}  
\caption{\textbf{Run time for the LCT model with different numbers of subcompartments.} Run time for a solver with a fixed step size of $10^{-2}$ with optimization flag \texttt{-O3} (left), with optimization flag \texttt{-O0} (center), and run time along with the required time steps of an adaptive solver (right) to compute a numerical solution of the LCT model~\eqref{eq:LCTSECIR}. The numbers of subcompartments are ranging from $1$ to $100$ and for the adaptive case to $200$. For each subcompartment number, the average time of $100$ runs is shown. Simulations are executed without age resolution and for $20$ simulation days.
}
\label{fig:runtime}
\end{figure}

\begin{figure}[bt]
    \centering
    \begin{minipage}[t]{0.48\textwidth}
    \centering
     \includegraphics[width=\textwidth]{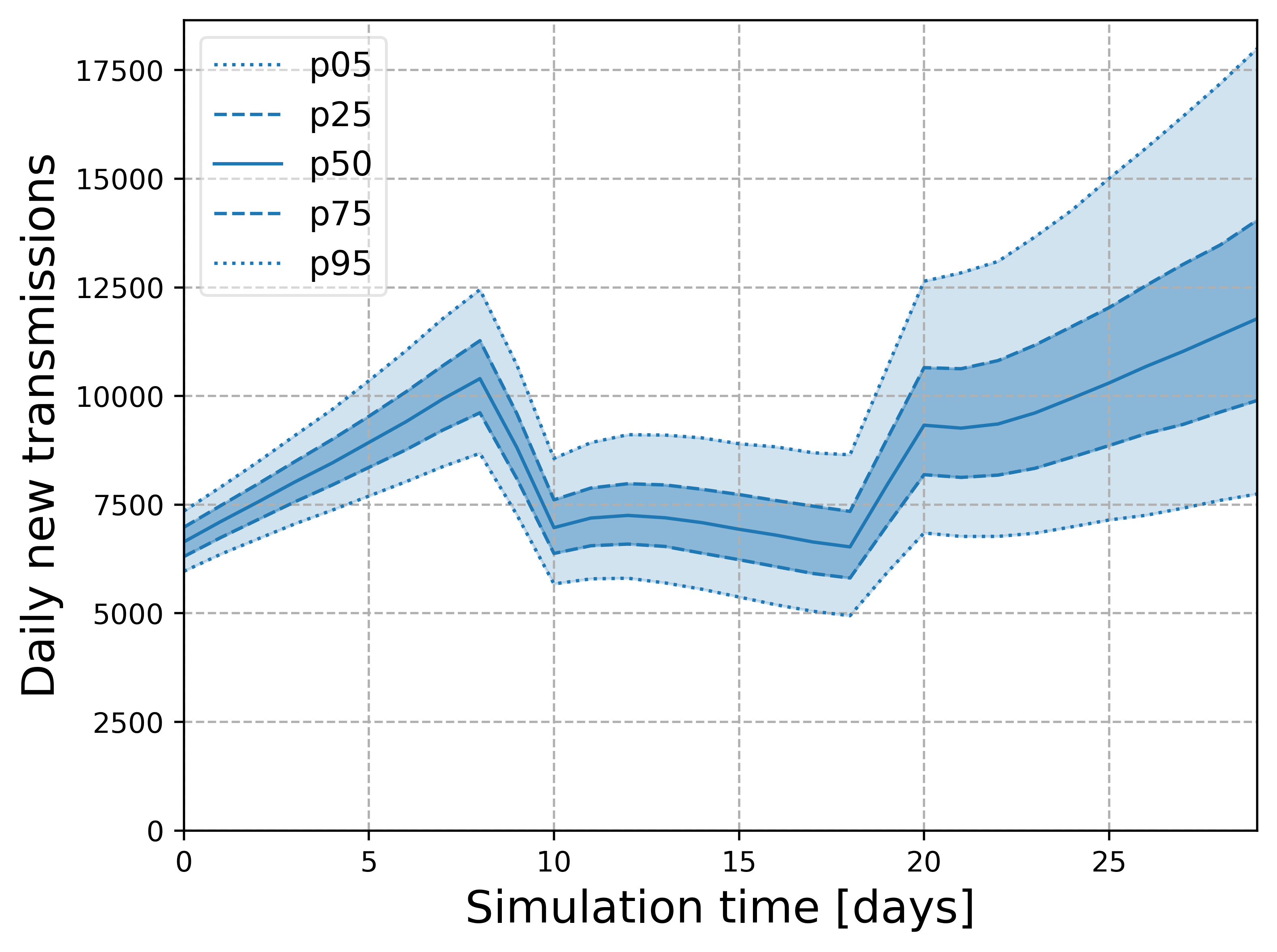}
\end{minipage}
\begin{minipage}[t]{0.48\textwidth}
 \centering
    \includegraphics[width=\textwidth]{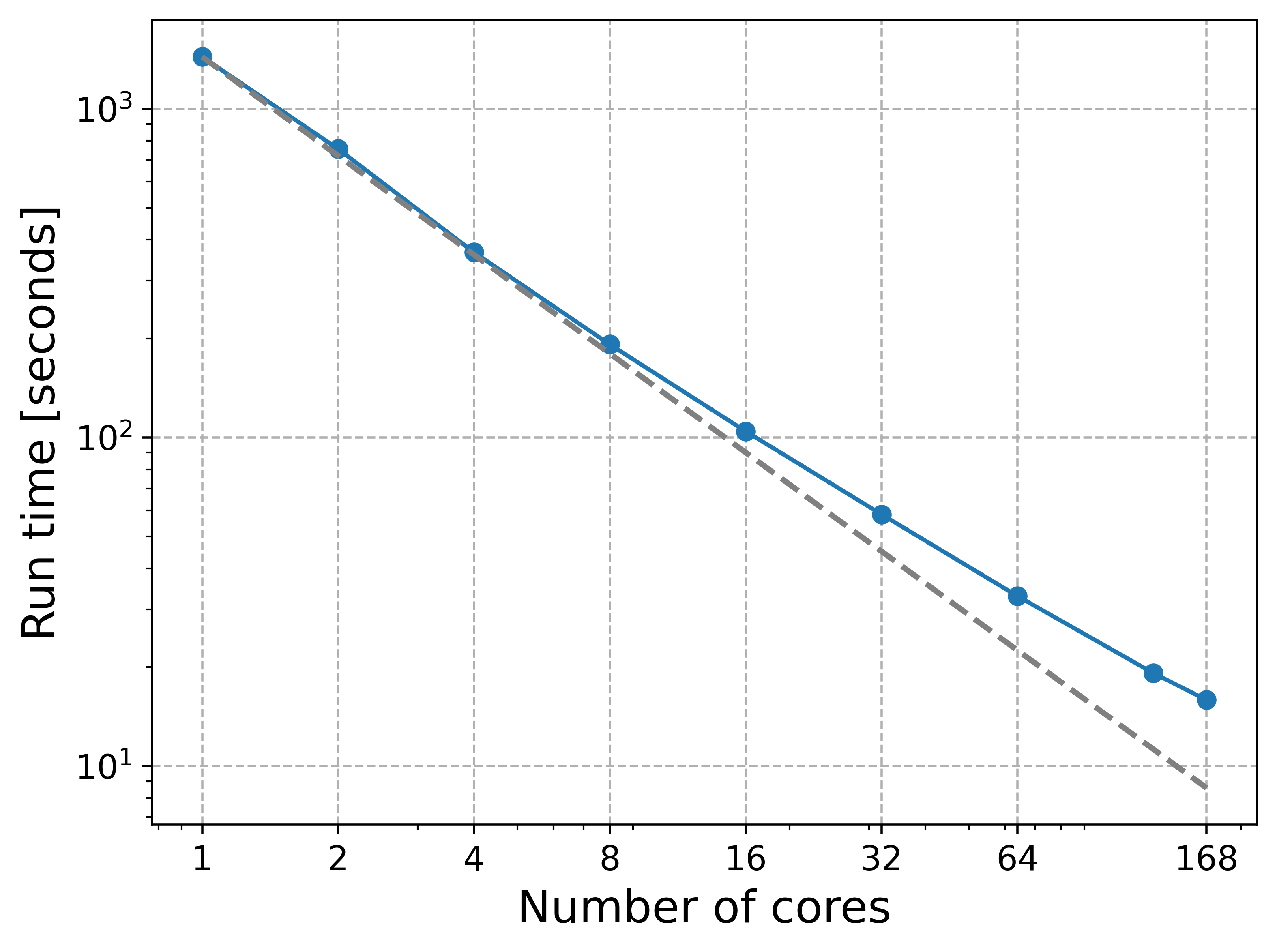}
\end{minipage}  
\caption{\MK{\textbf{Visualization of percentiles from ensemble run simulation and strong scaling behavior.} Visualization of percentiles p5, p25, p50, p75, p95 for an ensemble run simulation of 30 days, $16\,384$ simulations and two contact change points (left) and strong scaling behavior from 1 to 168 cores (right).}
}
\label{fig:scaling}
\end{figure}

\MK{In this section, we eventually study the performance of the implemented methods in MEmilio. We} present the run time behavior of our implementation of the LCT model~\eqref{eq:LCTSECIR} \MK{and consider the parallel speedup through MPI-parallelized ensemble runs}. 

\MK{We start by considering the sequential runtime first. Our objective was to implement an efficient and flexible LCT model class from which we can create LCT models with arbitrary numbers of subcompartments and whose simulation's} run time increases at most linearly in the number of subcompartments employed. The authors of~\cite{hurtado_building_2021} indicate that researchers typically hard-code the numbers of subcompartments and write multiple ODE functions to consider different subcompartment numbers. However, our model allows the flexibility to set the number of subcompartments for each compartment $Z\in\mathcal{A}$ and for each age group independently. Subcompartment realizations will be created upon compile-time\MK{, substantially reducing potential run-time overhead.}

For the run time analysis, we use one age group and set the parameters to the previously described values, including the single value contact rate described in~\cref{sec:ageresolution}. In this analysis, we once more set the number of subcompartments, $n_Z$, to an equal value for all $Z\in\mathcal{A}$. The initial values are set to reasonable values according to reported data by RKI. All run time measurements are conducted on an Intel Xeon "Skylake" Gold $6132$, $2.60$ GHz with $14$ cores per socket \MK{and four sockets}. 

\cref{fig:runtime} (left) depicts the run time taken to compute a numerical solution for the LCT model~\eqref{eq:LCTSECIR} under various assumptions regarding the number of subcompartments ($n_Z=1,\dots,100$). For each number of subcompartments, we compute the average time needed over $100$ runs. The simulations are performed for $20$ days each and a Runge-Kutta scheme of fifth order with a fixed step size of $10^{-2}$ is used as mentioned in~\cref{sec:parameters_data}. The run time increases mostly linearly with the numbers of subcompartments used. With the \texttt{-O3} optimization flag, we observe a jump in the run time from a number of $25$ to $26$ subcompartments. This occurs since the compiler includes additional optimizations for \textit{small} vector sizes. By setting the optimization flag \texttt{-O0}, a strictly linear curve without jumps is obtained, see~\cref{fig:runtime}~(center).

Additionally, we benchmark a simulation using an adaptive Runge-Kutta Cash Karp 5(4) solver~\cite{cash_variable_1990} without any limitations on the step size, while otherwise maintaining consistent model assumptions. The results regarding the run time and the number of time steps utilized are illustrated in~\cref{fig:runtime}~(right). We observe that the number of time steps increases linearly with the number of subcompartments. The investigation conducted without adaptive step sizing (i.e., with a constant number of steps) indicates a linear increase in run time. Therefore, we conclude that the run time per time step increases linearly with the number of subcompartments $n_Z$. In conjunction with the linear growth in the number of steps, it is a logical hypothesis that the run time for the adaptive procedure increases quadratically in the number of subcompartments. As illustrated in~\cref{fig:runtime}~(right), the curve for the run time indeed has a shape that is consistent with that of a quadratic function.

\MK{From~\cref{fig:runtime} (left), we see that 2000 equally-sized time steps corresponding to 20 simulation days with step size $10^{-2}$ of a model with more than 500 subcompartments (in total, 100 for compartments $E$, $C$, $I$, $H$, and $U$, 1 for $S$, $R$, and $D$) can be computed in approximately $0.03$ seconds. Using a more realistic, adaptive step size solver, we can even reduce this time by a factor of 3 to 4 and to approximately $0.008$ seconds. This means that, even in a sequential context, more than 7500 simulations can be run in less than one minute on a rather old Xeon Skylake processor, making the model thus suitable to be run and fitted on any consumer laptop. Furthermore, for smaller numbers of subcompartments, as often used in the literature, the runtime is several magnitudes smaller and almost 5000 executions of the LCT3 model can be run within one second on the mentioned processor.}

\MK{Eventually, we demonstrate the potential speedup through distributed memory parallelism that we provide with the novel MEmilio model. Therefore, we use the parameters from the prior section and introduce two contact change points throughout a simulation period of 30 days where we vary the initial parameters in a range of $\pm10~\%$. As this section is purely compute performance oriented, we refrain from providing more details on the epidemiological setting here; the simulation, however, is open source and can be found with our submission. We use a taylored random number generator (RNG) to ensure the correct generation of random numbers on all parallel ranks ensuring no dependencies of the individual simulations through the RNG. We provide percentiles of an ensemble run with $16\,384$ simulations in~\cref{fig:scaling} (left) and the strong scaling behavior in~\cref{fig:scaling} (right). While in the ensemble run settings, the particular simulations can be run in parallel, a nonnegligible overhead comes from the collection and computation of the percentiles and statistics at the end of all simulations. From~\cref{fig:scaling} (right), we see that we obtain very good speedup with three nodes of the Intel Xeon "Skylake" Gold $6132$, providing, in total, 168 compute cores. We speed up the process by a factor of approximately $90.7$.}

\section{Discussion}
ODE-based models are a popular approach for modeling the spread of infectious diseases. However, simple ODE-based models implicitly assume that the stay time in each compartment is exponentially distributed, which is unrealistic from an epidemiological point of view. We used the LCT to set up a \MK{SECIR-type} model that allows for Erlang distributed stay times in the compartments and thus generalizes simple ODE models. This allows for simulations with more realistic model behavior without the need to formulate complex models based on integro-differential equations. The resulting system is still formulated as an ODE system, which allows the use of already existing and efficient ODE solvers. 

To choose an appropriate number of subcompartments, we only need the mean and variance of the stay time distribution. This can be advantageous, especially in the beginning of the spread of a disease, when the realistic stay time distributions might not be known yet. 
Even if the variance is unknown, using a number of subcompartments higher than one already leads to a better prediction~\cite{krylova_effects_2013}. While we can include more realistic model assumptions by using an LCT model instead of a simple ODE model and do not require exact knowledge of the stay time distribution as in the case of IDE models, it should be encouraged to also report variances from population cohort studies. 

\MK{One goal of this study was to investigate the complex dependence and relationship between epidemic timings and peaks and the chosen number of subcompartments in an LCT model and to show that a general statement might be impossible to obtain. Here, we considered an advanced but fixed SECIR-type model and our findings might not directly apply to other models or model types. However, the key finding to thoroughly study these relationships with particular models and parameters before expectations on epidemic peaks and timings can be formulated, should be taken into account when developing LCT models.}

To further relax assumptions on the stay time distribution while maintaining the use of ODEs, the Generalized Linear Chain Trick as presented in~\cite{hurtado_generalizations_2019} can be applied. With this, it is possible to include phase-type distributions, which lie dense in the set of positive-valued distributions and thus allows for an even more flexible choice of stay time distributions.

The assumptions used in the COVID-19 inspired scenario could be enhanced. Firstly, the contact matrix was scaled globally, i.e., all age groups were scaled in the same way, at the beginning of the simulation as well as during the simulation when modeling the implementation of a NPI.
The contacts for the age groups could be scaled independently so that the simulation results match the extrapolated \LP{reported} data in each age group. Moreover, one could think of introducing age specific detection ratios to include different testing strategies, for example, testing in schools where the younger age groups are tested more often than average.

However, the focus of this paper was to consider the implications of the LCT in a broader context than the application to a particular disease. A review of the literature revealed that the statements made in other publications are not universally valid.

\section{Conclusion}
In this paper, we used the LCT to propose an age-resolved model that includes detailed infection states and allows for Erlang distributed stay time distributions. The proposed LCT model is a generalization of a simple ODE model, wherein the assumption of an exponentially distributed stay time is replaced with a more realistic one. We furthermore analyzed several properties related to the LCT model. To demonstrate the importance of the distribution assumption, a collection of numerical experiments was conducted. 

Our analyses indicate that the LCT model naturally incorporates a \MK{delay or} lag time between a change in the contact rate and a corresponding change in the number of daily new transmissions, as well as in the Carrier and Infected compartment sizes. We observed that an increase in the number of subcompartments leads to a longer \MK{delay}. One notable outcome is that the low variance of the stay times in models employing a high number of subcompartments may result in the emergence of wave patterns. Moreover, \MK{for our particular SECIR-type model,} we found that the comparison of the timing and size of epidemic peaks for varying numbers of subcompartments is significantly influenced by the effective reproduction number and, consequently, by the selected parameters. While we numerically showed that the final size of an epidemic is not affected by the assumption of different numbers of subcompartments, simple ODE-based models can lead to overly optimistic or pessimistic predictions of the epidemic peaks as well as predict the true peak time too early or too late such that no general statement is possible. \MK{However, let us note that the particular influences of parameters and effective reproduction number depend on the particular model and model type and that the magnitudes of the influence might differ for other models.}

We observed that the inclusion of age resolution enhances the accuracy of the simulations, enabling a more precise representation of the observed dynamics due to the incorporation of age-dependent parameters and initialization. In a scenario inspired by COVID-19, \LP{we showed how different LCT models can be used to make predictions for realistic applications.} It was demonstrated that the run time per step increases linearly with the number of subcompartments. Hence, the application of LCT models does not result in extensive additional costs. In our numerical results, we furthermore found that time-to-solution for an adaptive Runge-Kutta scheme scales quadratically with the number of subcompartments. 

All in all, we have seen that applying the LCT to obtain Erlang distributed stay times \MK{can lead} to widely different simulation results. \MK{In this paper, we demonstrated the complex relationship for a SECIR-type model. While it is unclear to which extent these findings directly apply to other models and model types, it might be overly optimistic to expect a simple relation between the epidemic peak size and timings and the number of subcompartments.} Therefore, when using mathematical modeling, one should pay careful attention to the underlying model assumptions as well as to the choice of parameters. 

\MK{In addition, a key contribution to the submitted paper is the modular software code for simple ODE and flexible LCT models that we provide. The modular structure of the code allows to extend non-age-resolved models with any age stratification with just two additional lines or parameters. Similar to age stratification, the number of subcompartments can be chosen flexible for any compartment and all choices made by the user are evaluated on compile-time such that the resulting C++ code runs highly efficient on consumer laptops and cluster infrastructure. Here, we simulated $16\,384$ runs and computed summary statistics in only $16$ seconds with 168 compute cores. Using the three particular compute nodes, we obtained a speedup of approximately $90.7$.}

The findings can be easily applied to epidemiological models for other infectious diseases such as Ebola, see~\cite{wang_evaluations_2017}. The concept of the LCT can be applied to other models based on ODEs with other applications, such as population dynamics~\cite{macdonald_time_1978}. Therefore, the results presented herein are not limited to our infectious disease model, but have broader applicability. 

\section*{Acknowledgements}
This work was supported by the Initiative and Networking Fund of the Helmholtz Association (grant agreement number KA1-Co-08, Project LOKI-Pandemics) and by the German Federal Ministry for Digital and Transport under grant agreement FKZ19F2211A (Project PANDEMOS). It was furthermore supported by the German Federal Ministry of Education and Research under grant agreement 031L0297B (Project INSIDe) and the Deutsche Forschungsgemeinschaft (DFG, German Research Foundation) (grant agreement 528702961).

\section*{Competing interests}
\noindent The authors declare to not have any competing interests.

\section*{Data availability}
\noindent The MEmilio repository is publicly available under \url{https://github.com/SciCompMod/memilio}. All model functionality is available with MEmilio v1.3.0~\url{https://zenodo.org/records/14237545}\LP{. A detailed description on how to reproduce the simulation results presented in this paper, in addition to the complete set of plot files, is accessible at~\url{https://github.com/SciCompMod/memilio-simulations}}.

\section*{Author Contributions}
    \noindent\textbf{Conceptualization:}  Lena Plötzke, Martin Kühn\\
    \noindent\textbf{Data Curation:}  Lena Plötzke, Anna Wendler\\
    \noindent\textbf{Formal Analysis:}  Lena Plötzke, Anna Wendler, Martin Kühn\\
    \noindent\textbf{Funding Acquisition:} Martin Kühn\\
    \noindent\textbf{Investigation:} Lena Plötzke, Anna Wendler, Martin Kühn\\
    \noindent\textbf{Methodology:} Lena Plötzke, Martin Kühn\\
    \noindent\textbf{Project Administration:} Martin Kühn\\
    \noindent\textbf{Resources:} Martin Kühn\\
    \noindent\textbf{Software:}  Lena Plötzke, Ren\'{e} Schmieding, Anna Wendler\\
    \noindent\textbf{Supervision:} Martin Kühn \\
    \noindent\textbf{Validation:} All authors \\
    \noindent\textbf{Visualization:} Lena Plötzke, Anna Wendler\\
    \noindent\textbf{Writing – Original Draft:} Lena Plötzke, Anna Wendler\\
    \noindent\textbf{Writing – Review \& Editing:} All authors
\bibliographystyle{elsarticle-num}

\begin{thebibliography}{10}
	\expandafter\ifx\csname url\endcsname\relax
	\def\url#1{\texttt{#1}}\fi
	\expandafter\ifx\csname urlprefix\endcsname\relax\def\urlprefix{URL }\fi
	\expandafter\ifx\csname href\endcsname\relax
	\def\href#1#2{#2} \def\path#1{#1}\fi
	
	\bibitem{hethcote_mathematics_2000}
	H.~W. Hethcote, The {Mathematics} of {Infectious} {Diseases}, SIAM Review
	42~(4) (May 2000).
	\newblock \href {https://doi.org/10.1137/S0036144500371907}
	{\path{doi:10.1137/S0036144500371907}}.
	
	\bibitem{daszak_workshop_2020}
	P.~Daszak, J.~Amuasi, C.~das Neves, D.~Hayman, T.~Kuiken, B.~Roche,
	C.~Zambrana-Torrelio, P.~Buss, H.~Dundarova, Y.~Feferholtz, G.~Foldvari,
	E.~Igbinosa, S.~Junglen, Q.~Liu, G.~Suzan, M.~Uhart, C.~Wannous,
	K.~Woolaston, P.~Mosig~Reidl, K.~O'Brien, U.~Pascual, P.~Stoett, H.~Li, H.~T.
	Ngo, Workshop {Report} on {Biodiversity} and {Pandemics} of the
	{Intergovernmental} {Platform} on {Biodiversity} and {Ecosystem} {Services}
	({IPBES}), Tech. rep., IPBES Secretariat, Bonn, Germany (Oct. 2020).
	\newblock \href {https://doi.org/10.5281/zenodo.7432079}
	{\path{doi:10.5281/zenodo.7432079}}.
	
	\bibitem{who_statistics_2023}
	{World Health Organization Team Data, Analytics \& Delivery},
	\href{https://www.who.int/publications/i/item/9789240074323}{World health
		statistics 2023: monitoring health for the {SDGs}, sustainable development
		goals}, Tech. rep., World Health Organization, Geneva (2023).
	\newline\urlprefix\url{https://www.who.int/publications/i/item/9789240074323}
	
	\bibitem{wallinga_how_2007}
	J.~Wallinga, M.~Lipsitch, How generation intervals shape the relationship
	between growth rates and reproductive numbers, Proceedings of the Royal
	Society B: Biological Sciences 274~(1609) (Feb. 2007).
	\newblock \href {https://doi.org/10.1098/rspb.2006.3754}
	{\path{doi:10.1098/rspb.2006.3754}}.
	
	\bibitem{hurtado_generalizations_2019}
	P.~J. Hurtado, A.~S. Kirosingh, Generalizations of the ‘{Linear} {Chain}
	{Trick}’: incorporating more flexible dwell time distributions into mean
	field {ODE} models, Journal of Mathematical Biology 79~(5) (Oct. 2019).
	\newblock \href {https://doi.org/10.1007/s00285-019-01412-w}
	{\path{doi:10.1007/s00285-019-01412-w}}.
	
	\bibitem{kuhn_assessment_2021}
	M.~J. Kühn, D.~Abele, T.~Mitra, W.~Koslow, M.~Abedi, K.~Rack, M.~Siggel,
	S.~Khailaie, M.~Klitz, S.~Binder, L.~Spataro, J.~Gilg, J.~Kleinert,
	M.~Häberle, L.~Plötzke, C.~D. Spinner, M.~Stecher, X.~X. Zhu, A.~Basermann,
	M.~Meyer-Hermann, Assessment of effective mitigation and prediction of the
	spread of {SARS}-{CoV}-2 in {Germany} using demographic information and
	spatial resolution, Mathematical Biosciences 339 (Sep. 2021).
	\newblock \href {https://doi.org/10.1016/j.mbs.2021.108648}
	{\path{doi:10.1016/j.mbs.2021.108648}}.
	
	\bibitem{pei_differential_2020}
	S.~Pei, S.~Kandula, J.~Shaman, Differential effects of intervention timing on
	{COVID}-19 spread in the {United} {States}, Science Advances 6~(49) (Dec
	2020).
	\newblock \href {https://doi.org/10.1126/sciadv.abd6370}
	{\path{doi:10.1126/sciadv.abd6370}}.
	
	\bibitem{chen_compliance_2021}
	X.~Chen, A.~Zhang, H.~Wang, A.~Gallaher, X.~Zhu, Compliance and containment in
	social distancing: mathematical modeling of {COVID}-19 across townships,
	International Journal of Geographical Information Science 35~(3) (Mar. 2021).
	\newblock \href {https://doi.org/10.1080/13658816.2021.1873999}
	{\path{doi:10.1080/13658816.2021.1873999}}.
	
	\bibitem{levin_effects_2021}
	M.~W. Levin, M.~Shang, R.~Stern, Effects of short-term travel on {COVID}-19
	spread: {A} novel {SEIR} model and case study in {Minnesota}, PLOS ONE 16~(1)
	(Jan. 2021).
	\newblock \href {https://doi.org/10.1371/journal.pone.0245919}
	{\path{doi:10.1371/journal.pone.0245919}}.
	
	\bibitem{liu_modelling_2022}
	J.~Liu, G.~P. Ong, V.~J. Pang, Modelling effectiveness of {COVID}-19 pandemic
	control policies using an {Area}-based {SEIR} model with consideration of
	infection during interzonal travel, Transportation Research Part A: Policy
	and Practice 161 (Jul. 2022).
	\newblock \href {https://doi.org/10.1016/j.tra.2022.05.003}
	{\path{doi:10.1016/j.tra.2022.05.003}}.
	
	\bibitem{zunker_novel_2024}
	H.~Zunker, R.~Schmieding, D.~Kerkmann, A.~Schengen, S.~Diexer, R.~Mikolajczyk,
	M.~Meyer-Hermann, M.~J. Kühn, Novel travel time aware metapopulation models
	and multi-layer waning immunity for late-phase epidemic and endemic
	scenarios, PLOS Computational Biology 20~(12) (Dez 2024).
	\newblock \href {https://doi.org/10.1371/journal.pcbi.1012630}
	{\path{doi:10.1371/journal.pcbi.1012630}}.
	
	\bibitem{medlock_integro-differential-equation_2004}
	J.~P. Medlock, Integro-differential-equation models in ecology and
	epidemiology, Ph.D. thesis, University of Washington (2004).
	
	\bibitem{messina_non-standard_2022}
	E.~Messina, M.~Pezzella, A.~Vecchio, A non-standard numerical scheme for an
	age-of-infection epidemic model, Journal of Computational Dynamics 9~(2)
	(Apr. 2022).
	\newblock \href {https://doi.org/10.3934/jcd.2021029}
	{\path{doi:10.3934/jcd.2021029}}.
	
	\bibitem{wendler2024nonstandardnumericalschemenovel}
	A.~Wendler, L.~Plötzke, H.~Tritzschak, M.~J. Kühn, A nonstandard numerical
	scheme for a novel {SECIR} integro-differential equation-based model allowing
	nonexponentially distributed stay times, Submitted for publication. (2024).
	\newblock \href {https://doi.org/10.48550/arXiv.2408.12228}
	{\path{doi:10.48550/arXiv.2408.12228}}.
	
	\bibitem{collier_parallel_2013}
	N.~Collier, M.~North, Parallel agent-based simulation with {Repast} for {High}
	{Performance} {Computing}, SIMULATION 89~(10) (Oct. 2013).
	\newblock \href {https://doi.org/10.1177/0037549712462620}
	{\path{doi:10.1177/0037549712462620}}.
	
	\bibitem{willem_optimizing_2015}
	L.~Willem, S.~Stijven, E.~Tijskens, P.~Beutels, N.~Hens, J.~Broeckhove,
	Optimizing agent-based transmission models for infectious diseases, BMC
	Bioinformatics 16~(1) (Dec. 2015).
	\newblock \href {https://doi.org/10.1186/s12859-015-0612-2}
	{\path{doi:10.1186/s12859-015-0612-2}}.
	
	\bibitem{bershteyn_implementation_2018}
	A.~Bershteyn, J.~Gerardin, D.~Bridenbecker, C.~W. Lorton, J.~Bloedow, R.~S.
	Baker, G.~Chabot-Couture, Y.~Chen, T.~Fischle, K.~Frey, J.~S. Gauld, H.~Hu,
	A.~S. Izzo, D.~J. Klein, D.~Lukacevic, K.~A. McCarthy, J.~C. Miller, A.~L.
	Ouedraogo, T.~A. Perkins, J.~Steinkraus, Q.~A. ten Bosch, H.-F. Ting,
	S.~Titova, B.~G. Wagner, P.~A. Welkhoff, E.~A. Wenger, C.~N. Wiswell, {for
		the Institute for Disease Modeling}, Implementation and applications of
	{EMOD}, an individual-based multi-disease modeling platform, Pathogens and
	Disease 76~(5) (Jul. 2018).
	\newblock \href {https://doi.org/10.1093/femspd/fty059}
	{\path{doi:10.1093/femspd/fty059}}.
	
	\bibitem{Kerkmann_ABM_24}
	D.~Kerkmann, S.~Korf, K.~Nguyen, D.~Abele, A.~Schengen, C.~Gerstein, J.-H.
	G\"obbert, A.~Basermann, M.~Meyer-Hermann, M.~J. K\"uhn, Agent-based modeling
	for realistic reproduction of human mobility and contact behavior to evaluate
	test and isolate in epidemic infectious disease spread, Accepted for
	publication in Computers in Biology and Medicine. (2024).
	\newblock \href {https://doi.org/10.48550/arXiv.2410.08050}
	{\path{doi:10.48550/arXiv.2410.08050}}.
	
	\bibitem{bradhurst_hybrid_2015}
	R.~A. Bradhurst, S.~E. Roche, I.~J. East, P.~Kwan, M.~G. Garner, A hybrid
	modeling approach to simulating foot-and-mouth disease outbreaks in
	{Australian} livestock, Frontiers in Environmental Science 3 (Mar. 2015).
	\newblock \href {https://doi.org/10.3389/fenvs.2015.00017}
	{\path{doi:10.3389/fenvs.2015.00017}}.
	
	\bibitem{hunter_hybrid_2020}
	E.~Hunter, B.~Mac~Namee, J.~Kelleher, A {Hybrid} {Agent}-{Based} and {Equation}
	{Based} {Model} for the {Spread} of {Infectious} {Diseases}, Journal of
	Artificial Societies and Social Simulation 23~(4) (Oct 2020).
	\newblock \href {https://doi.org/10.18564/jasss.4421}
	{\path{doi:10.18564/jasss.4421}}.
	
	\bibitem{bicker_hybrid_2025}
	J.~Bicker, R.~Schmieding, M.~Meyer-Hermann, M.~Kühn, Hybrid metapopulation
	agent-based epidemiological models for efficient insight on the individual
	scale: A contribution to green computing, Infectious Disease Modelling 10
	(Jun. 2025).
	\newblock \href {https://doi.org/10.1016/j.idm.2024.12.015}
	{\path{doi:10.1016/j.idm.2024.12.015}}.
	
	\bibitem{robertson_bayesian_2024}
	C.~Robertson, C.~Safta, N.~Collier, J.~Ozik, J.~Ray, Bayesian calibration of
	stochastic agent based model via random forest, { } (Jun. 2024).
	\newblock \href {https://doi.org/10.48550/arXiv.2406.19524}
	{\path{doi:10.48550/arXiv.2406.19524}}.
	
	\bibitem{schmidt_surrogates_2024}
	A.~Schmidt, H.~Zunker, A.~Heinlein, M.~J. Kühn, Graph neural network
	surrogates to leverage mechanistic expert knowledge towards reliable and
	immediate pandemic response, Submitted for publication. (2024).
	\newblock \href {https://doi.org/10.48550/arXiv.2411.06500}
	{\path{doi:10.48550/arXiv.2411.06500}}.
	
	\bibitem{donofrio_mixed_2004}
	A.~d’Onofrio, Mixed pulse vaccination strategy in epidemic model with
	realistically distributed infectious and latent times, Applied Mathematics
	and Computation 151~(1) (Mar. 2004).
	\newblock \href {https://doi.org/10.1016/S0096-3003(03)00331-X}
	{\path{doi:10.1016/S0096-3003(03)00331-X}}.
	
	\bibitem{lloyd_realistic_2001}
	A.~L. Lloyd, Realistic {Distributions} of {Infectious} {Periods} in {Epidemic}
	{Models}: {Changing} {Patterns} of {Persistence} and {Dynamics}, Theoretical
	Population Biology 60~(1) (Aug. 2001).
	\newblock \href {https://doi.org/10.1006/tpbi.2001.1525}
	{\path{doi:10.1006/tpbi.2001.1525}}.
	
	\bibitem{wearing_appropriate_2005}
	H.~J. Wearing, P.~Rohani, M.~J. Keeling, Appropriate {Models} for the
	{Management} of {Infectious} {Diseases}, PLoS Medicine 2~(7) (Jul. 2005).
	\newblock \href {https://doi.org/10.1371/journal.pmed.0020174}
	{\path{doi:10.1371/journal.pmed.0020174}}.
	
	\bibitem{feng_epidemiological_2007}
	Z.~Feng, D.~Xu, H.~Zhao, Epidemiological {Models} with {Non}-{Exponentially}
	{Distributed} {Disease} {Stages} and {Applications} to {Disease} {Control},
	Bulletin of Mathematical Biology 69~(5) (Jul. 2007).
	\newblock \href {https://doi.org/10.1007/s11538-006-9174-9}
	{\path{doi:10.1007/s11538-006-9174-9}}.
	
	\bibitem{krylova_effects_2013}
	O.~Krylova, D.~J.~D. Earn, Effects of the infectious period distribution on
	predicted transitions in childhood disease dynamics, Journal of The Royal
	Society Interface 10~(84) (Jul. 2013).
	\newblock \href {https://doi.org/10.1098/rsif.2013.0098}
	{\path{doi:10.1098/rsif.2013.0098}}.
	
	\bibitem{wang_evaluations_2017}
	X.~Wang, Y.~Shi, Z.~Feng, J.~Cui, Evaluations of {Interventions} {Using}
	{Mathematical} {Models} with {Exponential} and {Non}-exponential
	{Distributions} for {Disease} {Stages}: {The} {Case} of {Ebola}, Bulletin of
	Mathematical Biology 79~(9) (Sep. 2017).
	\newblock \href {https://doi.org/10.1007/s11538-017-0324-z}
	{\path{doi:10.1007/s11538-017-0324-z}}.
	
	\bibitem{kermack_contribution_1927}
	W.~O. Kermack, A.~G. McKendrick, G.~T. Walker, A contribution to the
	mathematical theory of epidemics, Proceedings of the Royal Society of London.
	Series A, Containing Papers of a Mathematical and Physical Character
	115~(772) (Aug. 1927).
	\newblock \href {https://doi.org/10.1098/rspa.1927.0118}
	{\path{doi:10.1098/rspa.1927.0118}}.
	
	\bibitem{breda_formulation_2012}
	D.~Breda, O.~Diekmann, W.~F. De~Graaf, A.~Pugliese, R.~Vermiglio, On the
	formulation of epidemic models (an appraisal of {Kermack} and {McKendrick}),
	Journal of Biological Dynamics 6~(sup2) (Sep. 2012).
	\newblock \href {https://doi.org/10.1080/17513758.2012.716454}
	{\path{doi:10.1080/17513758.2012.716454}}.
	
	\bibitem{macdonald_time_1978}
	N.~MacDonald, Time Lags in Biological Models, Lecture {Notes} in
	{Biomathematics}, Springer Berlin, Heidelberg, 1978.
	\newblock \href {https://doi.org/10.1007/978-3-642-93107-9}
	{\path{doi:10.1007/978-3-642-93107-9}}.
	
	\bibitem{feng_mathematical_2016}
	Z.~Feng, Y.~Zheng, N.~Hernandez-Ceron, H.~Zhao, J.~W. Glasser, A.~N. Hill,
	Mathematical models of {Ebola}—{Consequences} of underlying assumptions,
	Mathematical Biosciences 277 (Jul. 2016).
	\newblock \href {https://doi.org/10.1016/j.mbs.2016.04.002}
	{\path{doi:10.1016/j.mbs.2016.04.002}}.
	
	\bibitem{champredon_equivalence_2018}
	D.~Champredon, J.~Dushoff, D.~J.~D. Earn, Equivalence of the
	{Erlang}-{Distributed} {SEIR} {Epidemic} {Model} and the {Renewal}
	{Equation}, SIAM Journal on Applied Mathematics 78~(6) (Jan. 2018).
	\newblock \href {https://doi.org/10.1137/18M1186411}
	{\path{doi:10.1137/18M1186411}}.
	
	\bibitem{contento_integrative_2023}
	L.~Contento, N.~Castelletti, E.~Raimúndez, R.~Le~Gleut, Y.~Schälte,
	P.~Stapor, L.~C. Hinske, M.~Hoelscher, A.~Wieser, K.~Radon, C.~Fuchs,
	J.~Hasenauer, Integrative modelling of reported case numbers and
	seroprevalence reveals time-dependent test efficiency and infectious
	contacts, Epidemics 43 (Jun. 2023).
	\newblock \href {https://doi.org/10.1016/j.epidem.2023.100681}
	{\path{doi:10.1016/j.epidem.2023.100681}}.
	
	\bibitem{rozhnova_model-based_2021}
	G.~Rozhnova, C.~H. Van~Dorp, P.~Bruijning-Verhagen, M.~C.~J. Bootsma, J.~H.
	H.~M. Van De~Wijgert, M.~J.~M. Bonten, M.~E. Kretzschmar, Model-based
	evaluation of school- and non-school-related measures to control the
	{COVID}-19 pandemic, Nature Communications 12~(1) (Mar. 2021).
	\newblock \href {https://doi.org/10.1038/s41467-021-21899-6}
	{\path{doi:10.1038/s41467-021-21899-6}}.
	
	\bibitem{blyuss_effects_2021}
	K.~B. Blyuss, Y.~N. Kyrychko, Effects of latency and age structure on the
	dynamics and containment of {COVID}-19, Journal of Theoretical Biology 513
	(Mar 2021).
	\newblock \href {https://doi.org/10.1016/j.jtbi.2021.110587}
	{\path{doi:10.1016/j.jtbi.2021.110587}}.
	
	\bibitem{overton_epibeds_2022}
	C.~E. Overton, L.~Pellis, H.~B. Stage, F.~Scarabel, J.~Burton, C.~Fraser,
	I.~Hall, T.~A. House, C.~Jewell, A.~Nurtay, F.~Pagani, K.~A. Lythgoe,
	{EpiBeds}: {Data} informed modelling of the {COVID}-19 hospital burden in
	{England}, PLOS Computational Biology 18~(9) (Sep. 2022).
	\newblock \href {https://doi.org/10.1371/journal.pcbi.1010406}
	{\path{doi:10.1371/journal.pcbi.1010406}}.
	
	\bibitem{birrell_real-time_2021}
	P.~Birrell, J.~Blake, E.~Van~Leeuwen, N.~Gent, D.~De~Angelis, Real-time
	nowcasting and forecasting of {COVID}-19 dynamics in {England}: the first
	wave, Philosophical Transactions of the Royal Society B: Biological Sciences
	376~(1829) (Jul. 2021).
	\newblock \href {https://doi.org/10.1098/rstb.2020.0279}
	{\path{doi:10.1098/rstb.2020.0279}}.
	
	\bibitem{plotzke_ma_2023}
	L.~Plötzke, \href{https://elib.dlr.de/203691/}{{Der} {Linear} {Chain} {Trick}
		in der epidemiologischen {Modellierung} als {Kompromiss} zwischen
		gewöhnlichen und {Integro}-{Differentialgleichungen}}, Masterarbeit,
	Universität zu Köln (Dec. 2023).
	\newline\urlprefix\url{https://elib.dlr.de/203691/}
	
	\bibitem{cassidy_numerical_2022}
	T.~Cassidy, P.~Gillich, A.~R. Humphries, C.~H. Van~Dorp, Numerical methods and
	hypoexponential approximations for gamma distributed delay differential
	equations, IMA Journal of Applied Mathematics 87~(6) (Dec. 2022).
	\newblock \href {https://doi.org/10.1093/imamat/hxac027}
	{\path{doi:10.1093/imamat/hxac027}}.
	
	\bibitem{lloyd_dependence_2001}
	A.~L. Lloyd, The dependence of viral parameter estimates on the assumed viral
	life cycle: limitations of studies of viral load data, Proceedings of the
	Royal Society of London. Series B: Biological Sciences 268~(1469) (Apr.
	2001).
	\newblock \href {https://doi.org/10.1098/rspb.2000.1572}
	{\path{doi:10.1098/rspb.2000.1572}}.
	
	\bibitem{ma_generality_2006}
	J.~Ma, D.~J.~D. Earn, Generality of the {Final} {Size} {Formula} for an
	{Epidemic} of a {Newly} {Invading} {Infectious} {Disease}, Bulletin of
	Mathematical Biology 68~(3) (Apr. 2006).
	\newblock \href {https://doi.org/10.1007/s11538-005-9047-7}
	{\path{doi:10.1007/s11538-005-9047-7}}.
	
	\bibitem{memiliov1.3}
	M.~J. Kühn, D.~Abele, D.~Kerkmann, S.~Korf, H.~Zunker, A.~Wendler, J.~Bicker,
	K.~Nguyen, R.~Schmieding, L.~Plötzke, P.~Lenz, M.~Betz, C.~Gerstein,
	A.~Schmidt, R.~Hannemann-Tamas, N.~Waßmuth, P.~Johannssen, H.~Tritzschak,
	D.~Richter, M.~Klitz, W.~Koslow, S.~Binder, M.~Siggel, J.~Kleinert, K.~Rack,
	A.~Lutz, M.~Meyer-Hermann, {MEmilio v1.3.0 - A high performance Modular
		EpideMIcs simuLatIOn software}, \url{https://elib.dlr.de/201660/} (Nov.
	2024).
	\newblock \href {https://doi.org/10.5281/zenodo.14237545}
	{\path{doi:10.5281/zenodo.14237545}}.
	
	\bibitem{regionaldatenbank_deutschland_fortschreibung}
	{Regionaldatenbank Deutschland},
	\href{https://www.regionalstatistik.de/genesis//online?operation=table&code=12411-04-02-4-B&bypass=true&levelindex=1&levelid=1721805645378#abreadcrumb}{Fortschreibung
		des {Bevölkerungsstandes}: 12411-04-02-4-b {Bevölkerung} nach {Geschlecht}
		und {Altersjahren} (79) - {Stichtag} 31.12. - (ab 2011) regionale {Ebenen}},
	key date used: 31.12.2020 (2024).
	\newline\urlprefix\url{https://www.regionalstatistik.de/genesis//online?operation=table&code=12411-04-02-4-B&bypass=true&levelindex=1&levelid=1721805645378#abreadcrumb}
	
	\bibitem{RKI_data_2024}
	{Robert Koch-Institut}, {SARS-CoV-2} {I}nfektionen in {D}eutschland (2024).
	\newblock \href {https://doi.org/10.5281/zenodo.4681153}
	{\path{doi:10.5281/zenodo.4681153}}.
	
	\bibitem{divi2024}
	{Robert Koch-Institut}, {Intensivkapazitäten und
		COVID-19-Intensivbettenbelegung in Deutschland} (2024).
	\newblock \href {https://doi.org/10.5281/zenodo.13236164}
	{\path{doi:10.5281/zenodo.13236164}}.
	
	\bibitem{keeling_understanding_2002}
	M.~J. Keeling, B.~T. Grenfell, Understanding the persistence of measles:
	reconciling theory, simulation and observation, Proceedings of the Royal
	Society of London. Series B: Biological Sciences 269~(1489) (Feb 2002).
	\newblock \href {https://doi.org/10.1098/rspb.2001.1898}
	{\path{doi:10.1098/rspb.2001.1898}}.
	
	\bibitem{RKI-Lagebericht-15-10-2020}
	{Robert Koch-Institut}, \href{2020, https://www.rki.de/DE/
		Content/InfAZ/N/Neuartiges_Coronavirus/Situationsberichte/Okt_2020/
		2020-10-15-de.pdf?__blob=publicationFile}{Täglicher {Lagebericht} des {RKI}
		zur {Coronavirus}-{Krankheit}-2019 ({COVID}-19) am 15.10.2020}, Tech. rep.,
	{} (2020).
	\newline\urlprefix\url{2020, https://www.rki.de/DE/
		Content/InfAZ/N/Neuartiges_Coronavirus/Situationsberichte/Okt_2020/
		2020-10-15-de.pdf?__blob=publicationFile}
	
	\bibitem{dey_lag_2021}
	T.~Dey, J.~Lee, S.~Chakraborty, J.~Chandra, A.~Bhaskar, K.~Zhang, A.~Bhaskar,
	F.~Dominici, Lag time between state-level policy interventions and change
	points in {COVID}-19 outcomes in the {United} {States}, Patterns 2~(8) (Aug
	2021).
	\newblock \href {https://doi.org/10.1016/j.patter.2021.100306}
	{\path{doi:10.1016/j.patter.2021.100306}}.
	
	\bibitem{guglielmi_identification_2023}
	N.~Guglielmi, E.~Iacomini, A.~Viguerie, Identification of time delays in
	{COVID}-19 data, Epidemiologic Methods 12~(1) (Jan 2023).
	\newblock \href {https://doi.org/10.1515/em-2022-0117}
	{\path{doi:10.1515/em-2022-0117}}.
	
	\bibitem{blythe_distributed_1988}
	S.~P. Blythe, R.~M. Anderson, Distributed {Incubation} and {Infectious}
	{Periods} in {Models} of the {Transmission} {Dynamics} of the {Human}
	{Immunodeficiency} {Virus} ({HIV}), Mathematical Medicine and Biology 5~(1)
	(Mar 1988).
	\newblock \href {https://doi.org/10.1093/imammb/5.1.1}
	{\path{doi:10.1093/imammb/5.1.1}}.
	
	\bibitem{diekmann_discrete-time_2021}
	O.~Diekmann, H.~G. Othmer, R.~Planqué, M.~C.~J. Bootsma, The discrete-time
	{Kermack}–{McKendrick} model: {A} versatile and computationally attractive
	framework for modeling epidemics, Proceedings of the National Academy of
	Sciences 118~(39) (Sep. 2021).
	\newblock \href {https://doi.org/10.1073/pnas.2106332118}
	{\path{doi:10.1073/pnas.2106332118}}.
	
	\bibitem{kissler_projecting_2020}
	S.~M. Kissler, C.~Tedijanto, E.~Goldstein, Y.~H. Grad, M.~Lipsitch, Projecting
	the transmission dynamics of {SARS}-{CoV}-2 through the postpandemic period,
	Science 368~(6493) (May 2020).
	\newblock \href {https://doi.org/10.1126/science.abb5793}
	{\path{doi:10.1126/science.abb5793}}.
	
	\bibitem{brauer_mathematical_2019}
	F.~Brauer, C.~Castillo-Chavez, Z.~Feng, Mathematical {Models} in
	{Epidemiology}, Vol.~69 of Texts in {Applied} {Mathematics}, Springer New
	York, New York, NY, 2019.
	\newblock \href {https://doi.org/10.1007/978-1-4939-9828-9}
	{\path{doi:10.1007/978-1-4939-9828-9}}.
	
	\bibitem{Brauer_Age-of-infection_2008}
	F.~Brauer, Age-of-infection and the final size relation, Mathematical
	biosciences and engineering : MBE 5~(4) (Oct. 2008).
	\newblock \href {https://doi.org/10.3934/mbe.2008.5.681}
	{\path{doi:10.3934/mbe.2008.5.681}}.
	
	\bibitem{prem_projecting_2017}
	K.~Prem, A.~R. Cook, M.~Jit, Projecting social contact matrices in 152
	countries using contact surveys and demographic data, PLoS computational
	biology 13~(9) (Sep 2017).
	\newblock \href {https://doi.org/10.1371/journal.pcbi.1005697}
	{\path{doi:10.1371/journal.pcbi.1005697}}.
	
	\bibitem{fumanelli_inferring_2012}
	L.~Fumanelli, M.~Ajelli, P.~Manfredi, A.~Vespignani, S.~Merler, Inferring the
	{Structure} of {Social} {Contacts} from {Demographic} {Data} in the
	{Analysis} of {Infectious} {Diseases} {Spread}, PLoS Computational Biology
	8~(9) (Sep 2012).
	\newblock \href {https://doi.org/10.1371/journal.pcbi.1002673}
	{\path{doi:10.1371/journal.pcbi.1002673}}.
	
	\bibitem{koslow_appropriate_2022}
	W.~Koslow, M.~J. Kühn, S.~Binder, M.~Klitz, D.~Abele, A.~Basermann,
	M.~Meyer-Hermann, Appropriate relaxation of non-pharmaceutical interventions
	minimizes the risk of a resurgence in {SARS}-{CoV}-2 infections in spite of
	the {Delta} variant, PLOS Computational Biology 18~(5) (May 2022).
	\newblock \href {https://doi.org/10.1371/journal.pcbi.1010054}
	{\path{doi:10.1371/journal.pcbi.1010054}}.
	
	\bibitem{hurtado_building_2021}
	P.~J. Hurtado, C.~Richards, Building mean field {ODE} models using the
	generalized linear chain trick \& {Markov} chain theory, Journal of
	Biological Dynamics 15~(sup1) (May 2021).
	\newblock \href {https://doi.org/10.1080/17513758.2021.1912418}
	{\path{doi:10.1080/17513758.2021.1912418}}.
	
	\bibitem{cash_variable_1990}
	J.~R. Cash, A.~H. Karp, A variable order {Runge}-{Kutta} method for initial
	value problems with rapidly varying right-hand sides, ACM Transactions on
	Mathematical Software 16~(3) (Sep. 1990).
	\newblock \href {https://doi.org/10.1145/79505.79507}
	{\path{doi:10.1145/79505.79507}}.
	
\end{thebibliography}
\biboptions{sort&compress}

\end{document}